\documentclass[a4paper,11pt,reqno,english]{amsart}

\usepackage{babel}
\usepackage[T1]{fontenc}
\usepackage{lmodern}
\rmfamily
\DeclareFontShape{T1}{lmr}{b}{sc}{<->ssub*cmr/bx/sc}{}
\DeclareFontShape{T1}{lmr}{bx}{sc}{<->ssub*cmr/bx/sc}{}

\RequirePackage{amsmath}
\RequirePackage{amssymb}
\RequirePackage{amsfonts}
\RequirePackage{mathtools}
\RequirePackage{esint}
\RequirePackage{mathrsfs}
\RequirePackage{bm}
\RequirePackage{upgreek}

\numberwithin{equation}{section}
\usepackage{enumitem}

\newcommand{\cG}{\mathcal{G}}
\newcommand{\cH}{\mathcal{H}}
\newcommand{\cP}{\mathcal{P}}
\newcommand{\cS}{\mathcal{S}}

\newcommand{\bR}{\mathbb{R}}
\newcommand{\R}{\bR}
\newcommand{\bS}{\mathbb{S}}

\newcommand{\sC}{\mathscr{C}}
\newcommand{\sP}{\mathscr{P}}
\newcommand{\sU}{\mathscr{U}}
\newcommand{\sA}{\mathscr{A}}
\newcommand{\Heat}{\mathbf{H}}

\newcommand{\eps}{\varepsilon}
\newcommand{\even}{\mathrm{even}}
\newcommand{\odd}{\mathrm{odd}}
\newcommand{\pb}{\mathrm{par}}

\DeclareMathOperator{\supp}{supp}

\DeclareMathOperator{\tr}{tr}
\DeclareMathOperator{\Ker}{Ker}

\DeclareMathOperator{\dist}{dist}

\newcommand{\de}{\partial}
\newcommand{\ssubset}{\subset\joinrel\subset}   

\usepackage[hidelinks,
			hyperindex = true]{hyperref}
\usepackage{aliascnt}
\usepackage[noabbrev]{cleveref}

\usepackage{amsthm}

\theoremstyle{plain}
\newtheorem{theorem}{Theorem}[section]
\newaliascnt{corollary}{theorem}
\newtheorem{corollary}[corollary]{Corollary}
\aliascntresetthe{corollary}
\newaliascnt{lemma}{theorem}
\newtheorem{lemma}[lemma]{Lemma}
\aliascntresetthe{lemma}
\newaliascnt{proposition}{theorem}
\newtheorem{proposition}[proposition]{Proposition}
\aliascntresetthe{proposition}
\newaliascnt{definition}{theorem}
\newtheorem{definition}[definition]{Definition}
\aliascntresetthe{definition}
\newaliascnt{counterexample}{theorem}

\aliascntresetthe{counterexample}
\newaliascnt{conjecture}{theorem}

\aliascntresetthe{conjecture}
\newaliascnt{remark}{theorem}
\newtheorem{remark}[remark]{Remark}
\aliascntresetthe{remark}
\theoremstyle{remark}

\Crefname{claim}{claim}{claims}
\Crefname{claim}{Claim}{Claims}
\Crefname{assumption}{assumption}{assumptions}
\Crefname{assumption}{Assumption}{Assumptions}

\title[The generic structure of the singular set
in the Stefan problem]{The generic structure of the singular set\\
in the Stefan problem}
\author{Giacomo Colombo and Alessio Figalli}\date{}

\address{ETH Z\"urich, Department of Mathematics, R\"amistrasse 101, 8092, Z\"urich, Switzerland}
\email{giacomo.colombo@math.ethz.ch}  
\email{alessio.figalli@math.ethz.ch}
 
\begin{document}

\begin{abstract}
    We show that, in every space dimension $n\ge2$, the singular set $\Sigma$ of a generic (in the Baire-category sense) solution of the Stefan problem is countably parabolically $(n-2)$-rectifiable. In particular, $\dim_\pb(\Sigma)\le n-2$. 
    
    This is achieved by proving the following sharp result: in every space dimension, points of the maximal stratum of a generic solution are discrete in space-time.
\end{abstract}

\maketitle
\section{Introduction}
The Stefan problem \cite{Stefan} is one of the most classical free boundary problems and it aims to describe the evolution of free interfaces such as ice melting into water. The temperature $\theta\ge0$ evolves with the heat equation $\partial_t\theta=\Delta\theta$ in the water region $\{\theta>0\}$, with the additional boundary condition $|\nabla\theta|^2 = \partial_t\theta$ at the interface between ice and water $\partial\{\theta>0\}$ (called the free boundary). 

Assuming a non-zero speed of the free boundary and using the Baiocchi-Duvaut transformation $u(x,t) \coloneqq \int_0^t\theta(x,s)ds$ as in \cite{Duvaut} (see also \cite{surveyStefan}), one can check that $\partial\{\theta>0\}= \partial\{u>0\}$ and the Stefan problem becomes locally equivalent to the parabolic obstacle problem
\begin{equation}\label{eq:stefan}
    \begin{cases}
        (\Delta-\partial_t) u = \chi_{\{u>0\}}      & \text{in } \Omega\times[0,T],\\
        u\ge 0 \text{ and }\partial_tu\ge0          & \text{in } \Omega\times[0,T],\\
        \partial_tu > 0                             & \text{in } \{u > 0\}
    \end{cases}\qquad\Omega\subset\R^n,
\end{equation}
with the initial condition
\begin{equation}\label{eq:stefan-initial}
    u(\cdot,0)\equiv0\qquad\text{in }\Omega.
\end{equation}
Hence, from now on, we will focus on \eqref{eq:stefan}--\eqref{eq:stefan-initial}.

The modern regularity theory for the free boundary in this problem began with the seminal work \cite{Caffarelli77}. The main results therein state that the free boundary $\partial\{u>0\}$ is locally a graph of class $C^\infty$ in space and time, outside of a closed set $\Sigma$ of singular points at which the contact set $\{u=0\}$ has zero density.

Then, the focus of subsequent studies has been the singular part of the free boundary. Though this part may exist in all dimensions, and examples in the stationary case show that it can be pretty wild \cite{schaefferSingularities}, one can still prove some structure. In particular, denoting by $\Sigma_t$ the singular set at time $t$, variants of techniques used in the elliptic counterpart give $C^1$ regularity in space and $C^{1/2}$ in time of the singular set \cite{parabolicMonneau1,parabolicMonneau2} and show the optimal bound (in space) on the dimension of the singular set $\dim(\Sigma_t)\le n-1$.
Recently, the optimal bound in space-time on the singular set was established in \cite{FRS}, where the authors proved $\dim_\pb(\Sigma)\le n-1$. This bound has been refined in \cite{genericStefan4D} to show rectifiability, and extended to Lipschitz right hand sides in \cite{boundSingularSetLipRHS}.

In the last years, there has been growing interest in the generic behavior of the singular set in many geometric variational problems \cite{HS,CMS,genericRegularityIsoperimetricRegions,chodosh2025genericregularityminimizinghypersurfaces,multeplicity1MCF,huiskenConjectureMCF}, and more specifically free boundary problems \cite{FRS,FRSObst,genericStefan4D,monneauGeneric2d,genericSignorini3d,fernandezreal2023genericpropertiesfreeboundary,fernandezreal2025linearizedequationgenericregularity,carducci2024genericregularityfreeboundaries,carducci2026singularsetstefanproblem}.
For the Stefan problem, however, these techniques yield only {generic regularity in time}: as shown in \cite{FRS,genericStefan4D}, when $n\le4$, the singular set $\Sigma_t$ is empty for a.e. $t$.\footnote{Specifically, the set $\cS$ of times for which $\Sigma_t$ is not empty satisfies $\dim_\cH \cS = 0$ for $n=2$, $\cH^{1/2}(\cS)=0$ for $n=3$ and $\cH^1(\cS)=0$ for $n=4$ \cite{FRS,genericStefan4D}.}
These results shed little light on the generic structure of the whole free boundary in space-time.

\smallskip

The starting point of this paper is the following question:
\begin{center} {\it Can one improve, for a generic solution, the optimal bound $\dim_\pb(\Sigma)\le n-1$? }
\end{center}
Two features set this question apart from those previously considered in the free boundary literature. First, in contrast with the elliptic setting, singularities in the Stefan problem are persistent, as is also the case in other parabolic problems (such as the mean curvature flow). Second, all currently existing techniques rely on monotone perturbations, ultimately exploited through the maximum principle; for the Stefan problem, however, a translation in time is a monotone perturbation that does not change the structure of the free boundary, so one must impose further conditions on admissible perturbations to guarantee that they genuinely modify the singular set.

To state our result precisely, we first recall the structure of the singular set $\Sigma = \Sigma(u)$ of a solution $u$. At every singular point $(x_0,t_0)\in\Sigma(u)$, there is a unique quadratic blow-up
\[
    p_{x_0,t_0}(x)=\tfrac12\, x\cdot A_{x_0,t_0}x,\qquad A_{x_0,t_0}\in \R^{n\times n}\text{ symmetric},\quad A_{x_0,t_0}\ge0,\quad \tr A_{x_0,t_0}=1,
\]
in the sense that $$u(x_0+rx,t_0+r^2t) = r^2p_{x_0,t_0}(x)+o(r^2)\qquad \text{as $r\to0^+$}$$ \cite{Caffarelli77,parabolicMonneau1,parabolicMonneau2}; by contrast, at regular points, the blow-up is the one-sided profile $\tfrac12(x\cdot\nu)_+^2$. Stratifying the singular set by the dimension of the kernel, for $m\in\{0,\dots,n-1\}$ we set
\[
    \Sigma_m(u) \coloneqq \{(x_0,t_0)\in\Sigma(u)\,:\,\dim\Ker A_{x_0,t_0}=m\},
\]
so that $\Sigma(u)=\bigcup_{m}\Sigma_m(u)$ and $\dim_\pb\Sigma_m(u)\le m$ \cite{FRS}. With this notation, the \emph{maximal stratum} $\Sigma_{n-1}(u)$ corresponds to the rank-one case $A=\nu\otimes\nu$, namely to the blow-up $\tfrac12(x\cdot\nu)^2$.

The main result of this paper is the following.
\begin{theorem}\label{theorem:generic regularity}
    Let $\Omega\subset\R^n$ be a bounded smooth domain.
    In the space of solutions of \eqref{eq:stefan}--\eqref{eq:stefan-initial}, endowed with the topology
inherited from $C^0(\overline\Omega\times[0,T])$, there is a residual set $\cG$ such that, for every $u\in\cG$, the maximal stratum $\Sigma_{n-1}(u)$ is discrete in space-time.
\end{theorem}

Since the lower strata always satisfy $\dim_\pb\Sigma_m(u)\le m\le n-2$ and are countably parabolically rectifiable \cite{genericStefan4D}, we deduce the following.
\begin{corollary}\label{corollary:rectifiability}
    If $n\ge2$, for every $u\in\cG$ the singular set $\Sigma(u)$ is countably parabolically $(n-2)$-rectifiable. In particular,
    \[
        \dim_\pb\Sigma(u)\le n-2.
    \]
\end{corollary}
We emphasize that, in contrast with the generic-in-time results of \cite{FRS,genericStefan4D}, no restriction on the dimension $n$ is required.

\subsection{Strategy of the proof}
To prove \Cref{theorem:generic regularity}, we shall analyse how the
maximal stratum $\Sigma_{n-1}$ behaves under perturbations.
Our starting point is that, around a singular point in the maximal
stratum, a solution can be approximated to arbitrarily high order
by the sum of two independent solutions at a regular point
(see \Cref{theorem:splitting at singular points} and
\Cref{corollary:uniform smooth expansion maximal stratum}).
Thus, we need to investigate the stability properties of the free
boundary at such points, as well as the regularity of its
perturbations.\footnote{We note that, as a byproduct, we prove that, for every solution, the
projection of the maximal stratum onto the time axis has Hausdorff
dimension zero; see \Cref{remark:generic maximal stratum fixed solution}.}

For this heuristic discussion, we can think of the free boundaries
associated with perturbed solutions $u_\eps$ (so $u_0=u$) as smooth graphs
\[
    \{x_n=f_\eps^\pm(x',t)\},
\]
choosing coordinates so that the origin belongs to $\Sigma_{n-1}(u)$
and the hypersurfaces $\{x_n=f_0^\pm(x',0)\}$ are tangent to
$\{x_n=0\}$ at $x=0$.
Also, for each fixed $\eps$, the two graphs approach each other as $t$
increases and meet tangentially at points of $\Sigma_{n-1}(u_\eps)$.
Note that points at which the Hessian
\[
    D_{x'}^2(f_0^+-f_0^-)(0,0)
\]
is nondegenerate are isolated within the maximal stratum (see \Cref{lem:stratification maximal stratum}), so our aim is to show that singular points at which this Hessian has a kernel of dimension $m\ge1$ do not persist under generic perturbations.
Denoting the corresponding strata by $\Sigma_{n-1}^m$, as made precise in \Cref{subsec:stratify}, we shall eliminate them successively for $m=n-1,n-2,\ldots,1$, so that only $\Sigma_{n-1}^0$ survives generically.

The basic argument underlying many generic regularity results relies on constructing perturbations under which singular points move sufficiently rapidly as the parameter changes.
More precisely, we want to construct perturbations for which the value and $x'$-gradient of
\[
    \left.\partial_\eps(f_\eps^+-f_\eps^-)\right|_{\eps=0}
\]
at $(0,0)$ cannot be simultaneously generated by a spatial rotation and a space-time translation of $u$ (see \Cref{definition:transversality}).
For such transversal perturbations, we show that singular points persisting in the same stratum must move with infinite speed (compared with the size of the perturbation).
Quantitatively, one of our key estimates shows the following: if $z_0\in\Sigma_{n-1}^m(u_0)$ with $m\ge1$, the perturbation is transversal at $z_0$, and $z_\eps\in\Sigma_{n-1}^m(u_\eps)$ converges to $z_0$, then
\begin{equation}\label{eq:introduction explain cleaning}
    |\eps|=O\!\left(\dist_\pb(z_\eps,z_0)^N\right)
    \qquad\text{for every }N\ge1,
\end{equation}
where $\dist_\pb$ denotes the parabolic distance (see \Cref{prop:cleaning with transversality}).

This type of argument was introduced in the free-boundary context in \cite{FRSObst}, where estimates such as \eqref{eq:introduction explain cleaning} are called \emph{cleaning estimates}, and is central to subsequent generic regularity results in free boundary problems.
However, in these works, as also in recent generic regularity results for area-minimizing hypersurfaces
(see \cite{CMS,CMSImproved,chodosh2025genericregularityminimizinghypersurfaces}),
the cleaning estimate relies on a strictly monotone family of perturbations.
For the Stefan problem, instead, time translations
$u_\eps(x,t)=u(x,t+\eps)$ give a monotone family, and this
monotonicity is preserved when they are combined with spatial
translations and rotations whose velocities with respect to $\eps$
are sufficiently small.
Since all these perturbations preserve the structure of the singular
set, monotonicity alone cannot imply
\eqref{eq:introduction explain cleaning}.

Because of this, we consider a differentiable family $u_\eps$ and study
its first variation
\[
    \varphi\coloneqq
    \left.\partial_\eps u_\eps\right|_{\eps=0}.
\]
The condition $\partial_tu>0$ in $\{u>0\}$ allows us to characterize
$\varphi$ through a linear problem
(see \eqref{eq:linearized problem} and \Cref{lem:existence linear}),
in contrast with the nonlinear equations obtained by linearizing at
singular profiles in \cite{FS19,CinftyObstacle,branchingTwoPhaseBernoulli}.
Writing $u^\pm$ for the two solutions with free boundaries
$\{x_n=f_0^\pm(x',t)\}$ and $\varphi^\pm$ for the restrictions of
$\varphi$ to $\{u^\pm>0\}$, we have the expansions (see \Cref{prop:C3beta expansion linear})
\begin{equation}\label{eq:introduction expansion perturbation}
    \varphi^\pm
    =\big(\alpha_\varphi^\pm+x\cdot v_\varphi^\pm
        +Q_\varphi^\pm(x,t)\big)\partial_nu^\pm
    -x_n(v_\varphi^\pm\cdot\nabla u^\pm)
        +\text{higher order terms}.
\end{equation}
Here $\alpha_\varphi^\pm\in\R$, $v_\varphi^\pm\perp\bm e_n$, and
$Q_\varphi^\pm$ are parabolically $2$-homogeneous polynomials satisfying
$(\Delta-\partial_t)(x_nQ_\varphi^\pm)=0$.
The coefficients $\alpha_\varphi^\pm$ and $v_\varphi^\pm$ describe the
variations of the heights and normals, respectively, while
$Q_\varphi^\pm$ accounts for the quadratic deformation of the two graphs $f_0^\pm$.
Our goal is to find conditions on these coefficients which exclude
variations generated by translations and rotations and allow us to
prove \eqref{eq:introduction explain cleaning}.
Note that the coefficients $\alpha_\varphi^\pm$ alone cannot provide
such a condition, since any pair of values can be obtained by combining
a time translation with a translation in the normal direction
(see \Cref{lem:jets rototranslations}), so we also need to consider
the coefficients $v_\varphi^\pm$.

So, let us assume that $\varphi$ is transversal at the origin and that
\eqref{eq:introduction explain cleaning} fails.
At parabolic scale $r$, the terms involving $Q_\varphi^\pm$ contribute at order $|\eps|r^3$ to the perturbed solution, and therefore the estimate
\[
    \|u_\eps-u-\eps\varphi\|_{L^2}=o(|\eps|)
\]
is insufficient to compare their coefficients.
In addition, as there is no a priori relation between the displacement of the singular point $\dist_\pb(z_\eps,0)$ and $\eps$, terms of arbitrarily high order in the expansion of the translated solution cannot a priori be neglected relative to $\eps$.
To overcome this, we first construct nonlinear perturbations of $u$ adapted to the coefficients of $\varphi$, and then, using an improvement of flatness approach, we show that these perturbations approximate the solutions $u_\eps$ with an error $o(|\eps|r^3)$ (see \Cref{theorem:improvement of flatness}). We can therefore use these approximations in place of $u_\eps$ to understand how the free boundary changes at the singular points $z_\eps$.

In particular, this construction gives an $O(|\eps|)$ bound for the difference between the Hessians of the polynomial approximations of $f_\eps^+-f_\eps^-$ and $f_0^+-f_0^-$ at the singular points $z_\eps$, after a common translation in space-time and a rotation (see \eqref{eq:cleaning full Hessian comparison}).
On the other hand, persistence in $\Sigma_{n-1}^m$ requires the perturbed Hessian to have rank $n-1-m$, so (after a further reduction) we can use the \L ojasiewicz inequality to show that $\alpha_\varphi^+-\alpha_\varphi^-$ and $v_\varphi^+-v_\varphi^-$ coincide with those generated by a common translation in space-time, contradicting transversality and therefore proving \eqref{eq:introduction explain cleaning} (see \Cref{lem:lojasiewicz} and \Cref{prop:cleaning with transversality}).

\smallskip

At this point, the missing ingredient is to understand how to obtain transversal perturbations by changing only the boundary data of the original Stefan problem.
To this end, we first exploit the fine analysis of the free boundary near the maximal stratum to construct locally transversal perturbations (see \Cref{lem:local construction}),
and then approximate them in $L^2$ by global solutions with smooth boundary data supported in $\{u>0\}\cap(\partial\Omega\times(0,T))$.
This extension requires a Runge approximation argument, which is made delicate by the fact that the free boundary may contain singular points.
To achieve this, we need to carefully exploit the time monotonicity of $\{u(\cdot,t)>0\}$ and the zero capacity of the singular set in order to establish the ingredients needed to run the standard duality argument used in the proof of the Runge approximation (see \Cref{prop:runge}).
This is the argument behind \Cref{theorem:construction transversal perturbation}.

Finally, the cleaning estimate \eqref{eq:introduction explain cleaning} and
\Cref{theorem:construction transversal perturbation} allow us to
successively eliminate the strata
$\Sigma_{n-1}^{n-1},\ldots,\Sigma_{n-1}^1$ on any prescribed compact
subset of $\Omega\times(0,T)$ by arbitrarily small perturbations, so
openness and Baire's theorem yield
\Cref{theorem:generic regularity}.

\subsection*{Use of AI}
This project began in fall 2022, and the first complete draft of the manuscript was finished in June 2026, entirely without AI assistance. An abstract describing the results was submitted in July 2026 to the ``17th International Conference on Free Boundary Problems: Theory and Applications (FBP 2026)'', where the results were presented by the first author in September 2026. During subsequent revisions, AI tools were used exclusively to improve the clarity and presentation of the manuscript. All mathematical ideas, results, and arguments were developed independently by the authors.

\subsection{Structure of the paper}
In \Cref{section:preliminaries} we collect preliminary regularity
results and stratify the maximal stratum according to the cubic
expansion. In \Cref{section:linearized problem} we establish existence
and stability for the linearized problem, together with cubic
expansions at points of the maximal stratum and a Runge approximation
result. In \Cref{section:improvement of flatness} we relate the cubic
expansions of perturbed solutions to those of the first variation,
with estimates uniform in the displacement of the singular point.
These estimates are used in \Cref{section:cleaning}, where we define
transversality and prove the cleaning estimate
\eqref{eq:introduction explain cleaning}, while
\Cref{section:construction transverse pertubations} is devoted to
constructing transversal perturbations through changes of the boundary
data. Finally, in \Cref{section:main result} we combine the results obtained
to prove \Cref{theorem:generic regularity}.

\subsection{Acknowledgments} 
Both authors were partially supported by the European Research Council (ERC) under the European Union’s Horizon 2020 research and innovation programme, through the grant agreement No. 721675 ``Regularity and Stability in Partial Differential Equations''.

\section{Preliminary results}
\label{section:preliminaries}

    We denote by $\mathbf H$ the heat operator:
    \[
        \Heat \coloneqq \Delta-\partial_t.
    \]
    We equip $\R^{n+1}=\R^n\times \R$ with the parabolic distance
    \[
    \begin{aligned}
        &\dist_\pb((x,t),(y,s))\coloneqq \sqrt{|x-y|^2+|s-t|},\\
        &|(x,t)|_\pb\coloneqq\dist_\pb((x,t),(0,0))
        =\sqrt{|x|^2+|t|}.
    \end{aligned}
    \]
    Given $(x_0,t_0)\in\R^{n+1}$ we define the parabolic cylinders
    \[
        \sC_r(x_0,t_0) \coloneqq B_r(x_0)\times(t_0-r^2,t_0].
    \]
    Also, we use the notation
    \[
        \cP\coloneqq \big\{p(x) = \tfrac12 x\cdot Ax\,:\,
        A\in\R^{n\times n}\text{ is symmetric and nonnegative with }\tr(A)=1\big\}.
    \]

\subsection{Classical estimates}
As general references on the Stefan problem, we refer to \cite{parabolicObstacle,surveyStefan}.
By \cite{Caffarelli77}, a solution $u$ of
\[
    \Heat u=\chi_{\{u>0\}},\qquad u\ge0,\qquad \partial_tu\ge0
    \quad\text{in }\sC_1,
\]
with $(0,0)\in\de\{u>0\}$, enjoys the following regularity and non-degeneracy estimates:
\begin{equation}\label{eq:optimal regularity}
    \sup_{\sC_{1/2}} \bigl(|D^2u| + |\partial_tu|\bigr)\le C_n\|u\|_{L^\infty(\sC_1)},
    \qquad \sup_{B_r(0)} u(\cdot,0)\ge c_nr^2
    \quad \forall\,r\in(0,1),
\end{equation}
where $C_n,c_n>0$ are dimensional constants.
Also, if $u,v$ solve \eqref{eq:stefan} in $\sC_1$ then
\begin{equation}\label{eq:comparison principle}
    \Heat |u-v|\ge0\quad\text{in }\sC_1,
\end{equation}
as can be shown following, for instance, the proof of \cite[Lemma 6.1]{FRS}.
This implies in particular the comparison principle for \eqref{eq:stefan} (i.e., solutions that are ordered on a parabolic boundary are ordered also inside).

We will also make use of the following Caccioppoli-type estimate (see \cite[Lemma 6.1]{FRS}):
\begin{lemma}\label{lem:caccioppoli}
    Let $u_i\colon \sC_2 \to\R$, $i= 1,2$, solve
    \begin{equation}\label{eq:approx stefan}\begin{cases}
        \Heat u_i = \chi_{\{u_i>0\}}(1+\eps_i(x,t))\\
        u_i\ge0,\\
        \partial_tu_i\ge0,
    \end{cases}\qquad\text{in }\sC_2,\end{equation}
    with $|\eps_i(x,t)|\le \bar \eps<\frac1{100}$, and denote $w\coloneqq u_1-u_2$. Then
    \[
        \|w\|_{L^\infty(\sC_1)}\le C(\|w\|_{L^2(\sC_2)}+\bar\eps)
    \]
    and
    \[
        \left(\int_{\sC_1}|\nabla w|^2+|w\Heat w|+|\partial_tw|^2\right)^{1/2}\le C(\|w\|_{L^2(\sC_2)}+\bar\eps)
    \]
    where $C$ is a dimensional constant.
    \end{lemma}
    As already discussed in the introduction, blow-ups are classified. More precisely, Caffarelli's dichotomy \cite{Caffarelli77,parabolicMonneau1,parabolicMonneau2} states that, at a free boundary point  $(x_0,t_0)\in\de\{u>0\}$, one of the following holds:
    \begin{enumerate}[label = (\alph*)]
        \item either
        \[
            r^{-2}u(x_0+rx,t_0+r^2t) = \tfrac12(x\cdot\nu)_+^2 + o(1)\qquad\text{in }L^\infty(\sC_1)
        \]
        for some $\nu\in\bS^{n-1}$, and the free boundary is locally a smooth graph;
        \item or for some unique $p_{x_0,t_0}\in\cP$,         \[
            r^{-2}u(x_0+rx,t_0+r^2t) = p_{x_0,t_0}(x) + o(1)\qquad\text{in }L^\infty(\sC_1).
        \]
    \end{enumerate}
    We use the stratification $\Sigma(u)=\bigcup_{k=0}^{n-1}\Sigma_k(u)$
defined in the introduction.

\subsection{Some results on the free boundary}
    We collect some results on the size of the singular set (see \cite[Proposition 5.7, Corollary 8.8, Lemma 8.9]{FRS}):
    \begin{lemma}
    We have $\dim_\pb(\Sigma_k(u))\le k$.
    \end{lemma}
   Next, define the arrival time and its domain by
    \[
        \tau(x)\coloneqq\inf\{t\in[0,T]:u(x,t)>0\},\qquad
        \Omega_+\coloneqq\{x\in\Omega:0<\tau(x)<T\}.
    \]
    By the assumptions $\partial_tu>0$ in $\{u>0\}$ and $u(\cdot,0)\equiv0$, it holds 
    \[
        \de\{u>0\}\cap(\Omega\times(0,T))
        = \{(x,\tau(x)):x\in\Omega_+\}.
    \]
    We collect some properties of $\tau$.
    \begin{lemma}\label{lem:zero capacity singular set}
       Let $\Sigma$ be the singular set and set
\[
    S\coloneqq\{x\in\Omega_+:(x,\tau(x))\in\Sigma\}.
\]
 Then $S$ is relatively closed in $\Omega_+$, and $\tau\in C^\infty(\Omega_+\setminus S)$ with $\nabla\tau\neq0$. In addition,
$ \Sigma$ has zero $W^{1,2}(\R^{n+1})$-capacity.
    \end{lemma}
    \begin{proof}
        The first assertions follow from Caffarelli's dichotomy and the regular free-boundary theory, see \cite[Theorem 3.1]{FRS}. 
Next, since $\Sigma$ has locally $\sigma$-finite Euclidean $\cH^{n-1}$ measure in $\R^{n+1}$ (by \cite[Theorem 1.2]{genericStefan4D}), it has zero $W^{1,2}$-capacity (see for instance \cite[Theorem 4.16]{EvansGariepy}).
    \end{proof}

\subsection{The structure of the free boundary near the maximal stratum}
    We collect some consequences of the smooth expansion near singular points of the maximal stratum proved in \cite{FRS}.
    First, we recall the polynomial Ansätze given by \cite[Definitions 13.3 \& 13.4]{FRS}.
    \begin{definition}\label{definition:polynomial ansatz}
        Given an integer $J\ge3$, we say that a function $\sP$ is an Ansatz at order $J$ if there is a rotation $\mathsf{R}\in SO(n)$ and $\alpha\in\R$ such that
        \[
            \sP(x,t) = \tfrac12(\sA)_+^2(\mathsf{R}(x+\alpha\bm{e_n}),t)
        \]
        where $\sA$ is a polynomial of the form
        \[
            \sA(x,t) = x_n + \sum_{\ell=2}^{J-1} Q_\ell + x_n\sum_{\ell=2}^{J-1}R_\ell
        \]
        and such that:
        \begin{enumerate}[label=(\roman*)]
            \item $Q_\ell$, $R_\ell$ are parabolically $\ell$-homogeneous polynomials;
            \item $\Heat(x_n Q_\ell)\equiv0$;
            \item $\Heat \Big(\tfrac12\sA^2(x,t)\Big) = 1 + O(|(x,t)|_\pb^{J})$.
        \end{enumerate}
    \end{definition}
    The polynomials $R_\ell$ are determined recursively once $Q_2,\ldots,Q_\ell$ are fixed, see \cite[Definitions 13.3 \& 13.4]{FRS}.

  Also, given $\beta,R>0$ and $(x_0,t_0)\in\R^n\times\R$ we define
    \begin{equation}\label{eq:supercritical paraboloids}
        \Omega^\beta_R(x_0,t_0) \coloneqq \{(x,t)\in \R^n\times\R\,:\,|x-x_0|<R(t_0-t)^{\frac{1}{2+\beta}},\,t\le t_0\}.
    \end{equation}
    We also write $\Omega^\beta_R$ when $(x_0,t_0)=(0,0)$.
    We will make use of the following statement, proved in \cite[Theorem 13.5]{FRS}.
    \begin{theorem}\label{theorem:smooth expansion}
        Given $\beta>0$ and $k\ge4$ there exist constants $\eps_\circ,\bar r$ depending only on $n,\beta,k$ such that the following holds.

        Let $u$ satisfy $\Heat u=\chi_{\{u>0\}}$, $u\ge0$, and $\partial_tu\ge0$ in $\Omega^\beta_1\cap\sC_1$. Suppose that $0\in\partial\{u>0\}$ and
        \[
            \|u-\tfrac12(x_n)^2_+\|_{L^\infty(\Omega^\beta_1\cap\sC_1)}\le \eps_\circ.
        \]
        Then, for some Ansatz $\sP_k$ of order $k$
        \[
            \|u-\sP_k\|_{L^\infty(B_r\times(-r^2,-r^2/100))}\le r^{k+1/2}
        \]
        for $r\in(0,\bar r)$.
    \end{theorem}

    The expansion given by \Cref{theorem:smooth expansion} yields $C^\infty$ regularity of the free boundary and of the solution in $\overline{\{u>0\}}$. We will use the following restatement of the parabolic Whitney extension theorem.
        
    \begin{lemma}\label{lem:whitney}
        Let $E$ be a compact subset of $\R^{n+1}$ and let $M\ge0$ be an integer. For each $z=(x,t)\in E$, let $P_{2M,z}$ be a polynomial of parabolic degree at most $2M$. Suppose that, whenever $z=(x,t),w=(y,s)\in E$ and $|\alpha|+2j\le2M$,
        \[
            |\partial_x^\alpha\partial_t^jP_{2M,z}(0,0)-\partial_x^\alpha\partial_t^jP_{2M,w}(x-y,t-s)|
            \le A\dist_\pb(z,w)^{2M+1-|\alpha|-2j}.
        \]
        Then there exists $F\in C^{2M,M}(\R^{n+1})$ whose parabolic $2M$-jet at every $z=(x,t)\in E$ is $P_{2M,z}$, i.e.,
                \[
            F(x+\xi,t+\tau)=P_{2M,z}(\xi,\tau)+o(|(\xi,\tau)|_\pb^{2M}).
        \]
    \end{lemma}
    \begin{proof}
        This is \cite[Theorem B.1]{DanielliPTOP}, applied with modulus
        $\omega(r)=Ar$.
    \end{proof}
    
    We next use \Cref{lem:whitney} to show that the set where  $u$ admits
    uniform expansions of order $2M+1$ can be covered with $C^{2M,M}$ graphs.
        \begin{lemma}\label{lem:smooth covering}
        Let $u$ satisfy $\Heat u=\chi_{\{u>0\}}$, $u\ge0$, and $\partial_tu\ge0$ in $B_2\times(-2,2)$, and let $M\ge3$, $C>0$, and $r_\circ\in(0,1/100)$. There are $r_1=r_1(n,M,C,r_\circ)\in(0,r_\circ/10)$ and $\eps_\circ=\eps_\circ(n,M)>0$ such that the following holds. 
        
        Let $E\Subset\partial\{u>0\}\cap(B_1\times(-1,1))$ satisfy $\operatorname{diam}_\pb E<r_1$, and suppose that for each $z=(x_z,t_z)\in E$ there is an Ansatz $\sP_z$ of order $2M+1$ satisfying
        \begin{align}
            &\|u(x_z+\cdot,t_z+\cdot)-\sP_z\|_{L^\infty(B_r\times(-r^2,-r^2/100))}
            \le Cr^{2M+2}\qquad \forall\,r\in(0,r_\circ),\label{eq:u P}\\
            &\|\sP_z-\tfrac12(x_n)_+^2\|_{L^2(\sC_1)}\le\eps_\circ. \label{eq:P xn}
        \end{align}
        Then, writing $x=(x',x_n)$, there is $f\in C^{2M,M}(\R^{n-1}\times\R)$ such that
        \[
            E\subset\{x_n=f(x',t)\}.
        \]
    \end{lemma}
    \begin{proof}
        For $z=(x_z,t_z)\in E$, let $\sA_z$ denote the signed polynomial associated with $\sP_z$, so that
        \[
            \sP_z(x,t)=\tfrac12\bigl(\sA_z(x,t)\bigr)_+^2,
        \]
        and set
        \[
            \widetilde\sP_z\coloneqq\tfrac12\sA_z^2,
            \qquad
            B_z\coloneqq\partial_n\widetilde\sP_z.
        \]
        Let $z_i=(x_i,t_i)\in E$, $i=0,1$, and let $\rho\coloneqq\dist_\pb(z_0,z_1)$. By \eqref{eq:P xn} applied at $z=z_0$ and $z=z_1$, if $\eps_\circ$ is small enough we can find a cylinder $Q$ of parabolic size comparable to $\rho$ which lies both in the positive phases of the two Ans\"atze and in the regions where both expansions hold. Thus, for $(x,t)\in Q$,  \eqref{eq:u P} gives
        \begin{equation}\label{eq:base point whitney}
            \|\widetilde{\sP}_{z_0}(\cdot-x_0,\cdot-t_0)
            -\widetilde{\sP}_{z_1}(\cdot-x_1,\cdot-t_1)\|_{L^\infty(Q)}
            \le C\rho^{2M+2}.
        \end{equation}
        Parabolic rescaling and equivalence of norms for polynomials give
        \begin{equation}\label{eq:whitney derivative compatibility}
            \left|\partial_x^\alpha\partial_t^jB_{z_0}(0,0)
            -\partial_x^\alpha\partial_t^jB_{z_1}(x_0-x_1,t_0-t_1)\right|
            \le C\rho^{2M+1-|\alpha|-2j}
            \qquad \forall\,|\alpha|+2j\le2M.
        \end{equation}
        Since $z\in E\subset\partial\{u>0\},$ we have $\sA_z(0,0)=0$ and $B_z(0,0)=0$. Moreover, by \eqref{eq:P xn} and polynomial norm equivalence, after decreasing $\eps_\circ$ and $r_1$,
        \[
            \tfrac12\le\partial_nB_z(x,t)\le\tfrac32
            \qquad\text{for }|(x,t)|_\pb\le2r_1.
        \]
 Write $x_i=(x_i',x_{i,n}) \in \R^{n-1}\times \R$ and define
        \[
            d\coloneqq\dist_\pb((x_0',t_0),(x_1',t_1)),\qquad
            h\coloneqq|x_{0,n}-x_{1,n}|.
        \]
Since $B_{z_1}(0,0)=0$ and $\partial_nB_{z_1}\ge\tfrac12$,
the mean value theorem and \eqref{eq:whitney derivative compatibility} with $|\alpha|=j=0$ give
\[
    \frac{h}2
    \le \bigl|B_{z_1}((x_{0,n}-x_{1,n})\bm e_n,0)\bigr|
    \le |B_{z_1}(x_0-x_1,t_0-t_1)|+Cd
    \le C\rho^{2M+1}+Cd.
\]
Since $\rho^2=d^2+h^2$, decreasing $r_1$ gives $\rho\le Cd$;
in particular, the projection $\pi(x,t)\coloneqq(x',t)$
is injective on $E$.

Since $\partial_nB_z(0,0)>0$, the equation
$B_z(y',x_n-x_{z,n},s)=0$ defines $x_n$ as a smooth
function of $(y',s)$ near the origin; let $G_z$ denote
its parabolic Taylor polynomial of degree $2M$.
Implicit differentiation and Taylor's formula, together with
\eqref{eq:whitney derivative compatibility} and $\rho\le Cd$, give
\[
\bigl|\partial_{y'}^\alpha\partial_s^jG_{z_0}(0,0)
    -\partial_{y'}^\alpha\partial_s^j
      G_{z_1}(x_0'-x_1',t_0-t_1)\bigr|
    \le Cd^{2M+1-|\alpha|-2j}
    \qquad\forall\,|\alpha|+2j\le2M.
\]
Applying \Cref{lem:whitney} on $\pi(E)$ gives
$f\in C^{2M,M}(\R^{n-1}\times\R)$ with
\[
    f(x_z',t_z)=G_z(0,0)=x_{z,n}
    \qquad\forall\,z\in E,
\]
which proves the result.
    \end{proof}
    We will need the following consequence.
\begin{lemma}\label{lem:smoothness close maximal stratum}
    Suppose that $u$ solves \eqref{eq:stefan} in $\Omega^\beta_2\cap\sC_2$,  $(0,0)\in\partial\{u>0\}$, the blow-up of $u$ at the origin is $\tfrac12(x_n)_+^2$, and 
    \[
        \|u-\tfrac12(x_n)_+^2\|_{L^\infty\big(\Omega^\beta_2\cap\sC_2\big)}\le \eps_\circ.
    \]
 Given an integer $M \geq 0,$ if $\eps_\circ=\eps_\circ(n,\beta,M)$ is small enough, there is $f\in C^{2M,M}(\R^{n-1}\times\R)$ such that
    \[
        \{u>0\}=\{x_n>f(x',t)\}
        \quad\text{in }\Omega^\beta_1\cap\sC_{1/2}.
    \]
    Moreover, there is $C=C(n,\beta)$ such that
    \begin{equation}\label{eq:linear background estimates}
        |\nabla u(x,t)-x_n\bm{e_n}|\le C|(x,t)|_\pb^2,
        \qquad
        |D^2u(x,t)-\bm{e_n}\otimes\bm{e_n}|\le C|(x,t)|_\pb
    \end{equation}
    in $\{u>0\}\cap\Omega^\beta_1\cap\sC_1$.
\end{lemma}
\begin{proof}
    Set $\Gamma\coloneqq\partial\{u>0\}$ and $N\coloneqq\max\{M,3\}$. Using the inclusion
    \[
        \Omega^\beta_1(x_z,t_z)\subset\Omega^\beta_2
        \qquad \forall\,z=(x_z,t_z)\in\Gamma\cap\Omega^\beta_1\cap\sC_1,
    \]
    we deduce that the function
   $4u(x_z+x/2,t_z+t/4)$
    satisfies the hypotheses of \Cref{theorem:smooth expansion}. 
    Applying this theorem at order $2N+2$ and truncating at order $2N+1$, we obtain an order-$(2N+1)$ Ansatz $\sP_z$ satisfying $\|\sP_z-\tfrac12(x_n)_+^2\|_{L^2(\sC_1)}\le\eps_\circ$ and
    \[
        \|u(x_z+r\,\cdot,t_z+r^2\,\cdot)-\sP_z(r\,\cdot,r^2\,\cdot)\|_{L^\infty(B_1\times(-1,-1/100))}
        \le C_Nr^{2N+2}\qquad \forall\,r\in(0,r_{2N+2}),
    \]
    for some $r_{2N+2}=r_{2N+2}(n,\beta,N)>0$.
    Applying \Cref{lem:smooth covering} gives local $C^{2N,N}$ graphs,
    and the smallness of $\|\sP_z-\tfrac12(x_n)_+^2\|_{L^2(\sC_1)}$
    implies $\partial_n\sA_{\sP_z}(0,0)>0$, so the local graphs agree
    on overlapping domains and define a single $C^{2N,N}$ graph,
    which proves the first assertion.

    For the second assertion, applying \Cref{theorem:smooth expansion}
    at order four and truncating at order three gives Ansätze $\sP_{3,z}$
    with error $O(r^4)$ on $B_r\times(-r^2,-r^2/100)$
    for $0<r<r_4$, uniformly for
    $z\in\Gamma\cap\Omega^\beta_{3/2}\cap\sC_{r_4}$,
    where $r_4=r_4(n,\beta)>0$ is sufficiently small.
    Set $\widetilde{\sP}_{3,z}\coloneqq\tfrac12\sA_{\sP_{3,z}}^2$.
    Applying the first assertion with $M=3$ inside $\Omega^\beta_{7/4}$
    gives uniform $C^3$ free-boundary regularity, so boundary and
    interior estimates (localized as in \cite[Lemma 13.7]{FRS}) yield
    \begin{equation}\label{eq:smoothness third derivative bound}
        |D^3u|\le C
        \qquad\text{in }\{u>0\}\cap\Omega^\beta_{3/2}\cap\sC_{r_4}.
    \end{equation}
    Fix $(x,t)\in\{u>0\}\cap\Omega^\beta_1\cap\sC_{r_4/10}$ and set
    \[
        r\coloneqq|(x,t)|_\pb,\qquad
        x_z\coloneqq(x',f(x',t)),\qquad z\coloneqq(x_z,t).
    \]
    Since $f(0,0)=0$ and $\nabla'f(0,0)=0$, we have
    $f(x',t)=O(r^2)$, $|z|_\pb\le2r$, and $|x-x_z|\le Cr$.
    Hence, \eqref{eq:base point whitney} for the cubic Ans\"atze,
    together with polynomial norm equivalence, gives
    \begin{equation}\label{eq:taylor ansatz}
        |\nabla\widetilde{\sP}_{3,z}(x-x_z,0)-x_n\bm e_n|\le Cr^2,\qquad
        |D^2\widetilde{\sP}_{3,z}(x-x_z,0)-\bm e_n\otimes\bm e_n|\le Cr.
    \end{equation}
    Moreover, since $u$ and $\widetilde{\sP}_{3,z}(\cdot-x_z,\cdot-t)$
    have the same spatial derivatives through order two at $z$,
Taylor's formula and
    \eqref{eq:smoothness third derivative bound} give
    \[
        |\nabla u(x,t)-\nabla\widetilde{\sP}_{3,z}(x-x_z,0)|
      \le Cr^2,\qquad
        |D^2u(x,t)-D^2\widetilde{\sP}_{3,z}(x-x_z,0)|\le Cr.
    \]
    Together with \eqref{eq:taylor ansatz}, this proves the estimates
    near the origin, while the bounds on the rest of
    $\{u>0\}\cap\Omega^\beta_1\cap\sC_1$ follow from standard
    interior and boundary estimates.
\end{proof}

As observed in \cite{FRS}, even if the flatness assumption in \Cref{theorem:smooth expansion} is ``one-sided'', the previous results also apply at singular points in the maximal stratum. More precisely, 
thanks to \cite[Lemma 8.4]{FRS}, any solution with $(0,0)\in\Sigma_{n-1}$ satisfies
\begin{equation}\label{eq:definition nondegeneracy u_t}
    \fint_{\sC_r} \partial_tu\ge cr\qquad\forall\,r\in(0,1),
\end{equation}
with $c>0$.

\begin{theorem}\label{theorem:splitting at singular points}
    For any $M,c>0$ there are $\rho,C,\beta>0$, depending only on $n,M,c$,
    with the following property. Suppose that $u$ solves \eqref{eq:stefan} in
    $B_1\times(-1,1)$, that $(0,0)\in\Sigma_{n-1}(u)$, that
    $\|u\|_{L^\infty(B_1\times(-1,1))}\le M$, and that
    \eqref{eq:definition nondegeneracy u_t} holds with constant $c$.
    Then, after a rotation,
    \[
        u=u^++u^-\qquad\text{in }\Omega^\beta_1\cap\sC_\rho,
    \]
    where $u^\pm$ are solutions of \eqref{eq:stefan} in
    $\Omega^\beta_1\cap\sC_\rho$, $u^+u^-=0$, and
    \[
        |u^+-\tfrac12(x_n)_+^2|+|u^--\tfrac12(x_n)_-^2|
        \le Cr^3\quad\text{in }\Omega^\beta_1\cap\sC_r
        \qquad \forall\,r\in(0,\rho).
    \]
\end{theorem}
\begin{proof}
    The argument in the proof of
    \cite[Theorem 4.3]{genericStefan4D} gives, after decreasing the exponent and
    rotating,
    \[
        \bigg\|u-\tfrac12x_n^2
        -|x_n|\bigg(6at+(a+b)x_n^2
        -3\sum_{i=m+1}^{n-1}b_ix_i^2\bigg)-x_nP_2
        \bigg\|_{L^\infty(\sC_r)}
        \le Cr^{3+2\beta}
        \qquad \forall\,r\in(0,1/2),
    \]
    where $m\in\{0,\dots,n-1\}$, $a,b_i>0$,
    $b\coloneqq\sum_{i=m+1}^{n-1}b_i$, and $P_2$ is parabolically
    two-homogeneous, even in $x_n$, and satisfies $\Heat(x_nP_2)=0$. Also, all
    coefficients are bounded by $C$, and $a\ge\underline a(n,c)>0$ (by \eqref{eq:definition nondegeneracy u_t}).
Thanks to this, we can apply the proof of \cite[Theorem 13.1]{FRS} to conclude.
\end{proof}

To adapt the expansion at singular points, we introduce two-sided Ansätze (cp. \Cref{definition:polynomial ansatz}).
\begin{definition}\label{definition:two-sided Ansatz}
We say that $\sP$ is a two-sided Ansatz of order $J\ge3$ if there are one rotation $\mathsf{R}\in SO(n)$, one normal translation $\alpha\in\R$, and two polynomials $\sA^\pm$ of the form specified in \Cref{definition:polynomial ansatz} such that
    \[
        \sP(x,t)=\tfrac12[\sA^+(\mathsf{R}(x+\alpha\bm e_n),t)]_+^2
        +\tfrac12[\sA^-(\mathsf{R}(x+\alpha\bm e_n),t)]_-^2.
    \]
    Also, we define
        \[
        \sA^\even \coloneqq \tfrac12(\sA^+-\sA^-),\quad \sA^\odd \coloneqq \tfrac12(\sA^++\sA^-).
    \]
\end{definition}

The following corollary gives expansions of arbitrary order at points of
$\Sigma_{n-1}$, with uniform estimates and continuous dependence of the
coefficients on the solution and the singular point.
\begin{corollary}\label{corollary:uniform smooth expansion maximal stratum}
    Let $u$ satisfy the hypotheses of \Cref{theorem:splitting at singular points}. Then, for any integer $J\ge3$, there are $C,\beta,r_J>0$, depending only on $n,\|u\|_{L^\infty},c,$ and $J$, such that
    \[
        \|(u-\sP_J)(r\cdot,r^2\cdot)\|_{L^\infty(B_1\times(-1,-1/100))}
        \le Cr^{J+1/2}\qquad \forall\,r\in(0,r_J),
    \]
    where $\sP_J$ is a suitable two-sided Ansatz at order $J$.

    In particular, suppose that $u_k\to u$ in
    $L^2(B_1\times(-1,1))$ and that $z_k\in\Sigma_{n-1}(u_k)$ satisfies
    $z_k\to(0,0)$.
    If $\sP_{J,k}$ denotes the two-sided Ansatz of $u_k$ at $z_k$ of order $J$, then
    $C,\beta,r_J$ can be chosen
    independently of $k$ and
    \[
        \sP_{J,k}\longrightarrow\sP_J.
    \]
\end{corollary}
\begin{proof}
    To prove the first assertion, it suffices to apply \Cref{theorem:splitting at singular points,theorem:smooth expansion}
    to $u^+$ and to the inward reflection of $u^-$, at
    order $\max\{J,4\}$, and truncate to order $J$. 
    
  Next, if $u_k\to u$ in $L^2$, the local
    $L^2$--$L^\infty$ estimate for subcaloric functions gives a uniform
    bound for $u_k$ on a smaller cylinder, while the proof of
    \cite[Lemma 8.4]{FRS} gives a common constant in
    \eqref{eq:definition nondegeneracy u_t}. Hence the constants $C,\beta$, and $r_J$
    can be chosen independently of $k$, and the coefficients of the Ans\"atze are uniformly bounded.
    Since every limit of $\sP_{J,k}$ is an order-$J$ expansion of $u$ and
    therefore equals $\sP_J$ by uniqueness, the whole sequence
    converges.
\end{proof}

\subsection{Stratification of the maximal stratum}\label{subsec:stratify}

Recall that, thanks to  \cite[Theorem 4.3]{genericStefan4D}, given $(x_0,t_0)\in\Sigma_{n-1}(u)$ we can write, for $t \leq 0$,
\begin{equation}\label{eq:expansion cubic points}
    u(\mathsf{R}x+x_0,t+t_0)
    =\tfrac12x_n^2
    +|x_n|\bigg(6at+(a+b)x_n^2
    -3\sum_{i=m+1}^{n-1}b_ix_i^2\bigg)+p_3(x,t)
    +o(|(x,t)|_\pb^3),
\end{equation}
for some $\mathsf{R}\in SO(n)$ and $m\in\{0,\dots,n-1\}$, where $a,b_i>0$,
$b\coloneqq\sum_{i=m+1}^{n-1}b_i$, and $p_3=x_nP_2$, with $P_2$
parabolically two-homogeneous, even in $x_n$, and satisfying
$\Heat(x_nP_2)=0$.

We denote by $\Sigma_{n-1}^m(u)$ the set of points of $\Sigma_{n-1}(u)$ at
which \eqref{eq:expansion cubic points} holds with this value of $m$. Thus,
$m$ is the dimension of the kernel of the tangential quadratic form in the
cubic term.
The next proposition shows that, near each point of $\Sigma_{n-1}^m$, the maximal
stratum is contained in an $m$-dimensional manifold of any prescribed
finite regularity.
\begin{proposition}\label{lem:stratification maximal stratum}
    Let $u$ satisfy
    \[
        \Heat u=\chi_{\{u>0\}},\qquad u\ge0,\qquad \partial_tu\ge0
        \quad\text{in }B_1\times(-1,1),\qquad
        \partial_tu>0\quad\text{in }\{u>0\},
    \]
    and let $z_0=(x_0,t_0)\in\Sigma_{n-1}^m(u)$. Then for each integer $M\ge0$ there are a neighborhood $U$ of $z_0$ and an $m$-dimensional manifold $\bm M^m\subset U$ of class $C^{2M,M}$ such that
    \[
        \Sigma_{n-1}(u)\cap U\subset\bm M^m.
    \]
    Moreover, for every integer $N\ge1$,
    \begin{equation}\label{eq:cleaning fixed solution}
        |t-t_0|=o(|x-x_0|^N)
        \qquad\text{as }(x,t)\in\Sigma_{n-1}(u)\to(x_0,t_0).
    \end{equation}
\end{proposition}
Before proving the proposition, we introduce the following terminology.
\begin{definition}\label{definition:substratification}
    Let $\Omega\subset\R^n$ be a bounded smooth domain, let $u$ solve \eqref{eq:stefan}, let $K\Subset\Omega\times(0,T)$, and let $m\in\{0,\dots,n-1\}$. We say that $u$ is an $m$-singular solution in $K$ if
    $\Sigma_{n-1}^h(u)\cap K=\emptyset$ for every $h>m$.
\end{definition}

We first note that this is an open condition.
\begin{lemma}\label{lem:stratification closed conditions}
    Let $K\Subset U\Subset B_1\times(-1,1)$, and let $u_k,u$ solve
    \eqref{eq:stefan} on a common neighborhood of $\overline U$, with
    $u_k\to u$ in $L^2(U)$.

    If $z_k\in\Sigma_{n-1}(u_k)\cap K$ and $z_k\to z\in K$, then
    $z\in\Sigma_{n-1}(u)$. If $u$ is $m$-singular in $U$, then $u_k$ is
    $m$-singular in $K$ for all sufficiently large $k$.
\end{lemma}
\begin{proof}
    The first assertion is standard, see for instance \cite{FRS}. The second assertion follows
   from the convergence of the two-sided expansions in
    \Cref{corollary:uniform smooth expansion maximal stratum}.
\end{proof}

Before proving \Cref{lem:stratification maximal stratum}, we recall
the following \L ojasiewicz inequality
(see for instance \cite[Proposition 6.8]{CiteLojasiewicz}).
\begin{lemma}\label{lem:lojasiewicz}
    Let $f$ be a real analytic function on a neighborhood of the origin in $\R^n$, with $f(0)=0$. Then there are constants $C>0$ and $\theta\in(0,1)$ such that
    \begin{equation}\label{eq:lojasiewicz}
        |f(x)|^\theta\le C|\nabla f(x)|
    \end{equation}
    in some neighborhood of $0$.
\end{lemma}
\begin{proof}[Proof of \Cref{lem:stratification maximal stratum}]
    Up to a translation and a rotation, we can assume that $z_0$ is the origin and that
    \eqref{eq:expansion cubic points} holds with $\mathsf{R}=\mathrm{Id}$. We shall write
    \[
        x=(x',x'',x_n)\in\R^m\times\R^{n-1-m}\times\R.
    \]
    We first construct the manifold $\bm M^m$, and then prove
    \eqref{eq:cleaning fixed solution} using the \L ojasiewicz inequality.

    \smallskip\noindent\textbf{Step 1.}
    Fix $M\ge0$ as in the statement and choose
    $J\ge\max\{2M+1,3\}$. By
    \Cref{corollary:uniform smooth expansion maximal stratum}
    and the proof of \Cref{lem:smooth covering}, after decreasing $U$
    we obtain functions $\mathscr G^\pm$ of class $C^{2J,J}$ such that
    \[
        \Sigma_{n-1}(u)\cap U\subset
        \left\{
            x_n=\mathscr G^+(x',x'',t)=\mathscr G^-(x',x'',t),\quad
            \nabla_{x',x''}(\mathscr G^+-\mathscr G^-)(x',x'',t)=0
        \right\}.
    \]
    Set
    \[
        \mathscr G^\even\coloneqq\tfrac12(\mathscr G^+-\mathscr G^-),
        \qquad
        \mathscr G^\odd\coloneqq\tfrac12(\mathscr G^++\mathscr G^-).
    \]
    The cubic expansion in \eqref{eq:expansion cubic points} gives
    \[
        \partial_t\mathscr G^\even(0)=-6a,
        \qquad
        D^2_{x''}\mathscr G^\even(0)
        =6\operatorname{diag}(b_{m+1},\ldots,b_{n-1}),
    \]
    so if we define
    \[
        \Psi(x',x'',t)\coloneqq
        \bigl(\nabla_{x''}\mathscr G^\even(x',x'',t),
        \mathscr G^\even(x',x'',t)\bigr)
    \]
    we have
    \[
        D_{(x'',t)}\Psi(0)=
        \begin{pmatrix}
            6\operatorname{diag}(b_{m+1},\ldots,b_{n-1})&*\\
            0&-6a
        \end{pmatrix}.
    \]
    Since all diagonal elements are non-zero, the implicit function theorem gives a function
    $\Phi=(\Phi_{x''},\Phi_t)\colon\R^m\to \R^{n-1-m}\times\R$ satisfying
    \[
        \nabla_{x''}\mathscr G^\even
        (x',\Phi_{x''}(x'),\Phi_t(x'))=0,
        \qquad
        \mathscr G^\even
        (x',\Phi_{x''}(x'),\Phi_t(x'))=0.
    \]
    Therefore
    \[
        \bm M^m\coloneqq
        \{(x',\Phi_{x''}(x'),
        \mathscr G^\odd(x',\Phi_{x''}(x'),\Phi_t(x')),
        \Phi_t(x'))\}
    \]
    is of class $C^{2M,M}$ and contains every nearby point of
    $\Sigma_{n-1}(u)$.
    By \eqref{eq:expansion cubic points} and the equations defining $\Phi$,
    \[
        T_0\bm M^m=\R^m\times\{0\}.
    \]

    \smallskip\noindent\textbf{Step 2.}
    If $m=0$, the origin is isolated in $\Sigma_{n-1}(u)$,
    so we may assume $m\ge1$. 
    
    Fix $N\ge1$ and take $J$
    large enough in Step 1 that the time component $\Phi_t$ belongs to $C^{N+1}$.
    Differentiating the equation
    $$\mathscr G^\even(x',\Phi_{x''}(x'),\Phi_t(x'))=0$$
    and using $\nabla_{x'}\mathscr G^\even(x',x'',t)=0$
    at points $(x,t)\in\Sigma_{n-1}(u)\cap U$, together with
    $\partial_t\mathscr G^\even(0)=-6a\ne0$,
    we obtain $\nabla\Phi_t(x')=0$.
 Hence, if $P_N$ is the degree-$N$ Taylor polynomial of $\Phi_t$ at the origin,
for  $(x,t)\in\Sigma_{n-1}(u)\cap U$ we have
    \[
        t=\Phi_t(x')=P_N(x')+O(|x'|^{N+1}),\qquad
        \nabla P_N(x')=\nabla\Phi_t(x')+O(|x'|^N)=O(|x'|^N).
    \]
   In particular, thanks to \Cref{lem:lojasiewicz}, 
    \[
        |P_N(x')|^\theta\le C|\nabla P_N(x')|
        \le C|x'|^N,
    \]
     and since $\theta<1$ it follows that
    \[
        |t|=o(|x'|^N)=o(|x|^N)
        \qquad\text{as }(x,t)\in\Sigma_{n-1}(u)\to (0,0),
    \]
    as desired.
\end{proof}

\begin{remark}\label{remark:generic maximal stratum fixed solution}
    If $\pi_t$ denotes projection onto the time axis, \eqref{eq:cleaning fixed solution} and \Cref{prop:gmt} imply that 
    \[
        \dim_{\mathcal H}(\pi_t(\Sigma_{n-1}))=0.
    \]
    In other words, \cite[Theorem 1.2]{genericStefan4D} holds with $\Sigma^\infty\coloneqq\Sigma_{n-1}$.
\end{remark}
\section{Properties of the linearized operator}
\label{section:linearized problem}

Given $\Omega\subset\R^n$ a bounded smooth domain, $T>0$,  and a solution $u$ of  \eqref{eq:stefan}--\eqref{eq:stefan-initial}, we want to study the behaviour of a one-parameter family of solutions $u_\eps$, with $u_0=u$, whose parabolic boundary values depend
differentiably on the parameter.
To do this, we set 
$Q\coloneqq\Omega\times(0,T)$ and consider the linearized problem 
\begin{equation}\label{eq:linearized problem}
    \begin{cases}
        \Heat \varphi = 0    &\text{in }\{u>0\}\cap Q,\\
        \varphi = 0         &\text{in }\{u=0\}\cap Q.
    \end{cases}
\end{equation}
We will prove that such a family $u_\eps$ 
is differentiable in $L^2(Q)$ and that its derivative at
$\eps=0$ solves \eqref{eq:linearized problem}. Also, we 
will show a Runge-approximation
property for solutions of \eqref{eq:linearized problem}, as well as a
$C^{3,\alpha}$ expansion at points in the maximal stratum.

As boundary data for $\varphi$, we consider
\begin{equation}\label{eq:conditions boundary data}
    g\in C^\infty_c\Big(\{u>0\}\cap
       \big(\partial\Omega\times(0,T)\big)\Big),
\end{equation}
extended by zero to the whole parabolic boundary
\[
    \partial_\pb Q=(\overline\Omega\times\{0\})
      \cup(\partial\Omega\times(0,T)).
\]
We assume
\begin{equation}\label{eq:boundary conditions u}
    u\in C^1_{x,t}
       \big(\overline Q\big),\qquad
    \partial_tu\ge c_K>0
       \quad\text{for each }K\ssubset
       \{u>0\}\cap\partial_\pb Q.
\end{equation}

\subsection{Existence and linearization}
We first show the well-posedness of the linearized problem
\eqref{eq:linearized problem}.
Here and in the sequel, $\operatorname{Tr}_{\partial_\pb Q}$ will denote the trace of a Sobolev function on the parabolic boundary of $Q$.

A function
$\varphi\in H^1(Q)$ is called a weak solution of \eqref{eq:linearized problem} if it vanishes a.e. inside $\{u=0\}$ and satisfies
\[
    \int_Q\bigl(\nabla\varphi\cdot\nabla\zeta
        +\partial_t\varphi\,\zeta\bigr)=0
\]
for every $\zeta\in H^1_0(Q)$ vanishing a.e. in $\{u=0\}$.

\begin{lemma}\label{lem:existence linear}
    Given $g$ satisfying \eqref{eq:conditions boundary data}, there is a
    unique weak solution $\varphi_{g,u}\in H^1(Q)$ of
    \eqref{eq:linearized problem} with boundary data $g$.
    Also,
    \[
        \|\varphi_{g,u}\|_{H^1(Q)}\le C(u,\Omega,T,g),\qquad \text{and the map $g\mapsto\varphi_{g,u}$ is linear}.
    \]
   If in addition $u_k$ solve \eqref{eq:stefan}--\eqref{eq:stefan-initial} and $
        \big\|\operatorname{Tr}_{\partial_\pb Q}(u_k-u)
               -\eps_kg\big\|_{L^\infty(\partial_\pb Q)}
        =o(|\eps_k|)
$ for some sequence $\eps_k\to 0$, 
    then 
    \[
        \frac{u_k-u}{\eps_k}\longrightarrow\varphi_{g,u}
        \quad\text{in }L^2(Q).
    \]
\end{lemma}
\begin{proof}
    We first prove existence. By
    \eqref{eq:conditions boundary data} and \eqref{eq:boundary conditions u},
    the boundary data $u+\eps g$ are nonnegative and nondecreasing in time
    when $|\eps|$ is small. Let $U_\eps$ solve
    \eqref{eq:stefan} in $Q$ with parabolic boundary data $u+\eps g$, and note that $|U_\eps-u|\le |\eps|\|g\|_\infty$, so setting $w_\eps:= U_\eps-u$, the comparison principle gives
    \[
      |w_\eps|\le C|\eps|\quad\text{in }Q.
    \]
    On the other hand, $w_\eps=\eps g$ and $\partial_tw_\eps=\eps\partial_tg$ on the lateral boundary.
Thus, since $U_\eps,u>0$ in a fixed neighborhood of $\supp g$, boundary caloric estimates give
    \[
        |w_\eps\partial_\nu w_\eps|
        +|\partial_tw_\eps\partial_\nu w_\eps|
        \le C\eps^2
        \qquad\text{on }\partial\Omega\times(0,T).
    \]
    Also, since the time derivative of a nonnegative monotone solution
    vanishes almost everywhere on its contact set,
    \[
    \begin{split}
        w_\eps\,\Heat w_\eps
          &=U_\eps\,\chi_{\{U_\eps>0,u=0\}}
            +u\,\chi_{\{u>0,U_\eps=0\}}\ge0,\\
        \partial_tw_\eps\,\Heat w_\eps
          &=\partial_tU_\eps\,\chi_{\{U_\eps>0,u=0\}}
            +\partial_tu\,\chi_{\{u>0,U_\eps=0\}}\ge0.
    \end{split}
    \]
    Hence, for a.e. $t\in(0,T)$,
    \[
        \frac d{dt}\int_\Omega w_\eps^2
        +2\int_\Omega|\nabla w_\eps|^2
        \le2\left|\int_{\partial\Omega}
             w_\eps\partial_\nu w_\eps\right|
        \le C\eps^2.
    \]
    Analogously, we can test the equation for $w_\eps$ against $\frac{w_\eps(t)-w_\eps(t-h)}h$, 
   integrate in time, let $h\to 0^+$, and use that $w_\eps(\cdot,0)=0$,
    $\partial_tw_\eps\Heat w_\eps\ge0$, and
    $|\partial_tw_\eps\partial_\nu w_\eps|\le C\eps^2$, to obtain
    \[
        2\int_0^s\!\int_\Omega|\partial_tw_\eps|^2
        +\int_\Omega|\nabla w_\eps(\cdot,s)|^2
        \le C\eps^2
        \qquad\text{for a.e. }s\in(0,T).
    \]
This proves that
    \begin{equation}\label{eq:global caccioppoli}
        \left\|\frac{U_\eps-u}{\eps}\right\|_{H^1(Q)}
        \le C(u,\Omega,T,g).
    \end{equation}
   By \eqref{eq:global caccioppoli}, up to a subsequence,
    $\frac{U_\eps-u}{\eps} \rightharpoonup\varphi$ in $H^1(Q)$ and
    $\frac{U_\eps-u}{\eps} \to\varphi$ in $L^2(Q)$. Note that, since $U_\eps\to u$ uniformly,
    $\Heat\bigl((U_\eps-u)/\eps\bigr)=0$ on every compact subset of $\{u>0\}\cap Q$
    for small $\eps$, so $\Heat\varphi=0$ in $\{u>0\}\cap Q$.
    Moreover, extending $u$ to 0 for $t<0$, we can choose $A>0$ large enough so that, for any $\delta>0$
    $u(\cdot,\cdot-A|\eps|)\le u+\eps g\le
    u(\cdot,\cdot+A|\eps|)$ on $\de\Omega\times(0,T-\delta)$ and $\Omega\times\{0\}$, for $|\eps|$ small. Thus, the comparison principle gives
    \[
        u(\cdot,\cdot-A|\eps|)
        \le U_\eps\le
        u(\cdot,\cdot+A|\eps|)\qquad\text{in }\Omega\times(0,T-\delta),
    \]
    so $C^1$ regularity in time implies
    \begin{equation}\label{eq:linearization pointwise bound}
       \left| \frac{U_\eps-u}{\eps}\right| \le A\partial_tu+o(1)\qquad\text{in }L^\infty(\Omega\times(\delta,T-\delta)),
    \end{equation}
  therefore $|\varphi|\le A\partial_tu$. In particular,
    $\varphi=0$ almost everywhere in $\{u=0\}$.
    Finally, the weak $H^1$ convergence also implies that the trace of $\varphi$ coincides with $g$, which proves existence.
    
    Testing the equation for the difference of two solutions with the
    difference itself gives uniqueness, so $\varphi=\varphi_{g,u}$
    and the whole family converges in $L^2(Q)$. In addition, uniqueness implies linearity.

   Finally, let $U_{\eps_k}$ denote the family
    just constructed. By \eqref{eq:comparison principle} and the parabolic
    maximum principle,
    \[
        \|u_k-U_{\eps_k}\|_{L^\infty(Q)}
        \le
        \|\operatorname{Tr}_{\partial_\pb Q}
          (u_k-U_{\eps_k})\|_{L^\infty(\partial_\pb Q)}
        =o(|\eps_k|),
    \]
    so the asserted $L^2$ convergence follows.
\end{proof}

\begin{remark}\label{remark:existence linear variant}
    Let $Q' = \Omega'\times(a,b)\ssubset Q$, with $\Omega'$ smooth, and let $f\in H^1_0(\Omega'\cap\{u(\cdot,a)>0\})$. Then the same proof of \Cref{lem:existence linear} gives a unique $\Psi\in H^1(Q')$ satisfying
    \[
        \begin{cases}
            \Heat\Psi=0 &\text{in }Q'\cap\{u>0\},\\
            \Psi=0 &\text{in }Q'\cap\{u=0\},\\
            \Psi=0 &\text{on }\partial\Omega'\times(a,b),\\
            \Psi(\cdot,a)=f &\text{in }\Omega'.
        \end{cases}
    \]
    Moreover, with $\|\cdot\|_2$ denoting the norm in $L^2(\Omega')$,
    \[
        \begin{split}
            \operatorname*{ess\,sup}_{a<t<b}\|\Psi(t)\|_2^2
            +\|\nabla_x\Psi\|_{L^2(Q')}^2
                &\le C\|f\|_2^2,\\
            \|\partial_t\Psi\|_{L^2(Q')}^2
            +\operatorname*{ess\,sup}_{a<t<b}\|\nabla_x\Psi(t)\|_2^2
                &\le C\|\nabla f\|_2^2.
        \end{split}
    \]
    Indeed, the only additional check is monotonicity in time of the auxiliary functions $U_\eps$, which follows since $\Delta U_\eps(\cdot,a) - \chi_{\{U_\eps(\cdot,a)>0\}} = \partial_tu+\eps \Delta f\ge0$ for $f$ compactly supported and $|\eps|$ small. The general statement follows by approximation.
\end{remark}

The next result shows the continuous
dependence of $\varphi_{g,u}$ with respect to $u$.

\begin{corollary}\label{lem:controlled background linearized compactness}
    Let $u_k$ and $u$ belong to $C^1_{x,t}(\overline Q)$ and solve \eqref{eq:stefan}--\eqref{eq:stefan-initial} in $Q$.
    Assume that $u_k\to u$ uniformly, fix $g$ satisfying \eqref{eq:conditions boundary data}, and assume that
    $\sup_k\|u_k\|_{C^1_{x,t}(\overline Q)}<+\infty$ and
    $\inf_k\inf_{\supp g}\partial_tu_k>0$.
    Then
    \[
        \varphi_{g,u_k}\to\varphi_{g,u}
        \quad\text{in }L^2(Q).
    \]
\end{corollary}
\begin{proof}
    Equations \eqref{eq:global caccioppoli} and
    \eqref{eq:linearization pointwise bound} give, for some $A,L>0$
    independent of $k$,
    \[
        \sup_k\|\varphi_{g,u_k}\|_{H^1(Q)}<+\infty,\qquad
        |\varphi_{g,u_k}|\le A\partial_tu_k,\qquad
        0\le\partial_tu_k\le L.
    \]
    For a.e. $x$, the set
    $\{t:u(x,t)=0\}$ is an initial interval on which
    $0\le u_k\le\|u_k-u\|_{L^\infty(Q)}$, so integrating
    $\partial_tu_k$ in time gives
    \[
        \int_{\{u=0\}\cap Q}|\varphi_{g,u_k}|^2\le A^2\int_{\{u=0\}\cap Q}|\partial_tu_k|^2
        \le A^2L\int_{\{u=0\}\cap Q}\partial_tu_k
        \le C\|u_k-u\|_{L^\infty(Q)}\to0.
    \]
    It follows that any accumulation point of $\varphi_{g,u_k}$ is a weak solution to \eqref{eq:linearized problem} with background solution $u$ and boundary datum $g$. The result follows by the uniqueness part in \Cref{lem:existence linear}.
\end{proof}

\subsection{\texorpdfstring{$C^{3+\alpha}$}{} expansion at the maximal stratum}
We recall that $\sC_r=B_r\times(-r^2,0]$, while
$\Omega_R^\beta=\Omega_R^\beta(0,0)$ is defined in
\eqref{eq:supercritical paraboloids}.
After decreasing $\beta$ if necessary, we assume that $0<\beta<1$.
In this section we show a $C^{3+\alpha}$ expansion for solutions of
\eqref{eq:linearized problem} in a setting that can be applied also at
singular points in the maximal stratum. 

For this purpose, fix
$u\colon\Omega^\beta_1\cap\sC_1\to[0,+\infty)$ solving
\eqref{eq:stefan}, with $(0,0)$ a free boundary point, and satisfying
\begin{equation}\label{eq:cubic flatness}
    \|u-\tfrac12(x_n)_+^2\|_{L^\infty(\Omega^\beta_1\cap\sC_r)}
       \le Cr^3 \qquad \forall\, r \in (0,1).
\end{equation}
After a fixed initial parabolic rescaling,
\Cref{lem:smoothness close maximal stratum} gives a smooth
free-boundary graph near the origin and the estimates
\eqref{eq:linear background estimates}.
We also note that, thanks to \Cref{theorem:splitting at singular points}, at a singular point in the maximal stratum, after a fixed initial rescaling, this setting applies
separately to $u^+$ and (after reflection) to $u^-$.

Given $\alpha\in\R$, $v \in \R^n$ with $v\perp\bm e_n$, and a parabolically $2$-homogeneous
polynomial $Q$ satisfying $\Heat(x_nQ)\equiv0$, we write
\begin{equation}\label{eq:ansatz linearized problem}
    \upvarphi[\alpha,v,Q]
    \coloneqq(\alpha+x\cdot v+Q)\partial_nu
      -x_n\partial_vu.
\end{equation}
In the case of $u^-$, defining $I_-(x',x_n)\coloneqq(x',-x_n)$ and
$u_\sharp^-\coloneqq u^-\circ I_-$, it follows directly from
\eqref{eq:ansatz linearized problem} that
\begin{equation}\label{eq:lower branch ansatz dictionary}
    \upvarphi_{u_\sharp^-}[\alpha,v,Q]\circ I_-
      =-\upvarphi_{u^-}[\alpha,v,Q\circ I_-].
\end{equation}
The next proposition shows that solutions of \eqref{eq:linearized problem} admit
a cubic expansion of the form \eqref{eq:ansatz linearized problem}.
\begin{proposition}\label{prop:C3beta expansion linear}
    There is $C>0$ such that every weak solution $\varphi$ of
    \eqref{eq:linearized problem} in $\Omega^\beta_1\cap\sC_1$ admits a
    unique triple $(\alpha_*,v_*,Q_*)$ as above satisfying
    \[
        |\alpha_*|+|v_*|+\|Q_*\|_{L^2(\sC_1)}
        \le C\|\varphi\|_{L^2(\Omega^\beta_1\cap\sC_1)}
    \]
    and, for every $N\ge1$, there are $C_N,r_N>0$, independent of
    $\varphi$, such that
    \[
        \|(\varphi-\upvarphi[\alpha_*,v_*,Q_*])
              (r\cdot,r^2\cdot)\|_
           {L^2(\Omega^\beta_N\cap\sC_1)}
        \le C_N\|\varphi\|_{L^2(\Omega^\beta_1\cap\sC_1)}r^{3.9}
        \qquad \forall\,r\in(0,r_N).
    \]
\end{proposition}
To prove the proposition, we need some preliminary results.
First, we estimate how far $\upvarphi[\alpha,v,Q]$ is from solving
\eqref{eq:linearized problem}.

\begin{lemma}\label{lem:pde linear ansatz}
    Suppose $|\alpha|+|v|+\|Q\|_{L^2(\sC_1)}\le1$. Then
    \[
        |\Heat\upvarphi[\alpha,v,Q]|
            \le C|x,t|_\pb^2\quad\text{ in }\{u>0\}\cap \Omega^\beta_1\cap\sC_1
    \]
    for some
    $C=C(n,\beta,u)$. Also,
$\upvarphi[\alpha,v,Q]$ vanishes inside $\{u=0\}$.
\end{lemma}
\begin{proof}
    First note that the functions $\alpha\partial_nu$ and
    $(v\cdot x)\partial_nu-x_n\partial_vu$ solve
    \eqref{eq:linearized problem}  (the second one is the derivative
    generated by a spatial rotation). So it remains to consider
    $Q\partial_nu$. In the positive phase,
    \[
        \Heat(Q\partial_nu)
        =\partial_nu\,\Heat Q
          +2\nabla\partial_nu\cdot\nabla Q
          +Q\Heat\partial_nu.
    \]
    Also, by \eqref{eq:linear background estimates},
    \[
        \partial_nu=x_n+O(|x,t|_\pb^2),\qquad
        \nabla\partial_nu=\bm e_n+O(|x,t|_\pb),\qquad
        \Heat\partial_nu=0.
    \]
    Since $x_n\Heat Q+2\partial_nQ=\Heat(x_nQ)=0$, the first assertion
    follows. 
    The second assertion is immediate, since $u=\nabla u=0$ on the contact set.
    \end{proof}

The next lemma controls the difference between the coefficients
of two Ansätze by their $L^2$ distance at each scale.
\begin{lemma}\label{lem:ansatz linear vs parameters}
    Set
    \[
        K\coloneqq B_{1/4}\times(-3/4,-1/2)\cap\{x_n>1/100\}.
    \]
    Given two triples $(\alpha_i,v_i,Q_i)$ as above,
    set $\upvarphi_i\coloneqq\upvarphi[\alpha_i,v_i,Q_i]$ for $i=1,2$, and put
    \[
       \mathcal N_r\coloneqq|\alpha_1-\alpha_2|+r|v_1-v_2|
          +r^2\|Q_1-Q_2\|_{L^2(\sC_1)}.
    \]
    There are $c,C,r_0>0$ such that
    \[
       cr\mathcal N_r\le
       \|(\upvarphi_1-\upvarphi_2)(r\cdot,r^2\cdot)\|_{L^2(K)}
       \le Cr\mathcal N_r
       \qquad \forall\,r\in(0,r_0).
    \]
\end{lemma}
\begin{proof}
    By \eqref{eq:cubic flatness}, $u(rx,r^2t)>0$ for every $(x,t)\in K$
    and all sufficiently small $r$. Moreover,
    \eqref{eq:linear background estimates} gives
    \[
        |\nabla u(rx,r^2t)-rx_n\bm e_n|\le Cr^2
        \qquad\text{in }K,
    \]
    therefore
    \[
      (\upvarphi_1-\upvarphi_2)(rx,r^2t)
       =rx_n\bigl((\alpha_1-\alpha_2)
          +rx\cdot(v_1-v_2)+r^2(Q_1-Q_2)\bigr)+E_r,
      \qquad \|E_r\|_{L^2(K)}\le Cr^2\mathcal N_r.
    \]
    Thus, finite-dimensional norm equivalence for polynomial norms, together with the fact that $\|E_r\|_{L^2(K)} \ll r \mathcal N_r$, proves both
    bounds.
\end{proof}

We next establish the energy estimate in $\Omega_R^\beta$
needed for compactness. For convenience, we define the parameter
\begin{equation}\label{eq:cusp scaling exponent}
    \gamma_\beta\coloneqq\frac{\beta}{2+\beta}.
\end{equation}
\begin{lemma}\label{lem:caccioppoli linear}
    There are $C,R_0>0$, depending only on $n$ and $\beta$, such that the
    following holds. Let $R\ge R_0$, and suppose that $u$ solves
    \eqref{eq:stefan} in $\Omega_R^\beta\cap\sC_1$. Let
    $\varphi\in H^1(\Omega_R^\beta\cap\sC_1)$ satisfy
    \[
        \Heat\varphi=f\quad\text{in }\{u>0\},
        \qquad \varphi=0\quad\text{in }\{u=0\},
    \]
    where $f\in L^\infty(\{u>0\}\cap\Omega_R^\beta\cap\sC_1)$. Then
    \[
        \int_{\Omega_{2^{-\gamma_\beta}R}^\beta\cap\sC_{1/2}}
            \left(|\nabla\varphi|^2
                 +|\partial_t\varphi|^2\right)
        \le C\bigg(
            \int_{\Omega_R^\beta\cap\sC_1}\varphi^2
              +\|f\|_\infty^2\bigg).
    \]
\end{lemma}
\begin{proof}
    Testing the weak equation with $\eta\varphi$ for a nonnegative
    compactly supported cutoff $\eta$ gives, in the weak sense,
        \begin{equation}\label{eq:energy1}
        \varphi\Heat\varphi\ge-\|f\|_\infty|\varphi|.
    \end{equation}
    Since $\partial_tu\ge0$ we have
    $\{u(\cdot,t-h)>0\}\subset\{u(\cdot,t)>0\}$ for $h>0$, so
    \[
        \eta(x,t)\frac{\varphi(x,t)-\varphi(x,t-h)}{h}
    \]
    is an admissible test function; integrating in time and letting
    $h\downarrow0$ gives, again in the weak sense,
        \begin{equation}\label{eq:energy2}
        \partial_t\varphi\Heat\varphi
        \ge-\|f\|_\infty|\partial_t\varphi|.
    \end{equation}
    Using \eqref{eq:energy1} and \eqref{eq:energy2}, we obtain the desired
    estimate by applying the proof of \cite[Lemma 13.7]{FRS}.
\end{proof}

We denote by $G$ the backward fundamental solution of the heat equation: 
\begin{equation}\label{eq:gaussian kernels}
    G(x,t)\coloneqq\frac1{(4\pi|t|)^{n/2}}
        \exp\bigg(\frac{|x|^2}{4t}\bigg),\qquad t<0.
\end{equation}

The following is the compactness lemma used in the iteration. If $b$ is
a background solution, $\upvarphi_b[\alpha,v,Q]$ denotes the expression in
\eqref{eq:ansatz linearized problem} with $u$ replaced by $b$.

\begin{lemma}\label{lem:improvement linear key lemma}
    Given $C_0,L>1$, there is an integer $\Lambda\ge8$, depending only
    on $n,\beta,C_0,L$, with the following property. Let
    $R\ge\Lambda$, and let $\bar u$ be a solution of
    \eqref{eq:stefan} in $ \Omega_{2^{\gamma_\beta\Lambda}R}^\beta\cap\sC_{2^\Lambda}
$
    such that
    \[
        \left\|\bar u-\tfrac12(x_n)_+^2\right\|_
        {L^2(B_\Lambda\times(-2,-1))}\le\Lambda^{-1}.
    \]
    Suppose that
    $w\in H^1_{\rm loc}
        \bigl(\Omega_{2^{\gamma_\beta\Lambda}R}^\beta
                    \cap\sC_{2^\Lambda}\bigr)$.
    Set $w_r(x,t)\coloneqq w(rx,r^2t)$ for $r>0$. Assume that, for every
    $m\in\{0,\ldots,\Lambda\}$,
    \begin{equation}\label{eq:assumption linear key lemma 1}
        \left(\int_{\Omega_R^\beta\cap\sC_1}
        \bigl(|w_{2^m}|^2+|\nabla w_{2^m}|^2
                  +|\partial_tw_{2^m}|^2\bigr)\right)^{1/2}
        \le C_0(2^m)^{3.9}.
    \end{equation}
    Assume also that
    \begin{equation}\label{eq:assumption linear key lemma 2}
        \begin{cases}
        |\Heat w|\le\Lambda^{-1}
          &\text{in }\Omega_R^\beta\cap
          (B_\Lambda\times(-\Lambda,-\Lambda^{-1}))
          \cap\{x_n\ge\Lambda^{-1}\},\\
        w=0
          &\text{in }\Omega_R^\beta\cap
          (B_\Lambda\times(-\Lambda,-\Lambda^{-1}))
          \cap\{x_n\le-\Lambda^{-1}\}.
        \end{cases}
    \end{equation}
    Finally, suppose that
    \begin{equation}\label{eq:assumption linear key lemma 3}
        \left|\int_{B_\Lambda\times(-2,-1)}
        w\,\upvarphi_{\bar u}[\widetilde\alpha,\widetilde v,\widetilde Q]
        G\right|\le\Lambda^{-1}
    \end{equation}
    whenever $\widetilde v\perp\bm e_n$, $\widetilde Q$ is
    parabolically $2$-homogeneous,
    $\Heat(x_n\widetilde Q)=0$, and
$
        |\widetilde\alpha|+|\widetilde v|
           +\|\widetilde Q\|_{L^2(\sC_1)}\le1.
$
    Then
    \[
        \|w_{1/2}\|_{L^2(\Omega_R^\beta\cap\sC_1)}\le L^{-1}.
    \]
\end{lemma}
\begin{proof}
    This is the compactness argument of \cite[Lemma 13.6]{FRS}: if the result were false, the energy bounds would give a limit $w_\infty$, caloric in $\{x_n>0\}$ and equal to zero in $\{x_n\le0\}$.
    Thus, the odd reflection of its restriction to $\{x_n>0\}$ is an ancient caloric function with growth strictly smaller than four, and hence a caloric polynomial of parabolic degree at most three.
    Moreover, $\bar u\to\tfrac12(x_n)_+^2$ locally in $\R^n\times(-2,-1)$ and
    \[
        \upvarphi_{\bar u}[\widetilde\alpha,\widetilde v,
          \widetilde Q]
        \longrightarrow
        (x_n)_+(\widetilde\alpha+x\cdot\widetilde v
          +\widetilde Q).
    \]
    Since the odd reflection is of the form $x_n(\widetilde\alpha+x\cdot\widetilde v+\widetilde Q)$ with $\widetilde v\perp\bm e_n$ and $\Heat(x_n\widetilde Q)=0$, \eqref{eq:assumption linear key lemma 3} implies $w_\infty=0$, contradicting the failure of the conclusion.
\end{proof}

We can now prove \Cref{prop:C3beta expansion linear}. We first introduce the notation
\begin{equation}\label{eq:parabolic scaling}
    S_r(x,t)\coloneqq(rx,r^2t).
\end{equation}

\begin{proof}[Proof of \Cref{prop:C3beta expansion linear}]
    Fix $N\ge1$. We first construct
    a triple $(\alpha_N,v_N,Q_N)$ and constants $C_N,r_N>0$ such that
    \begin{equation}\label{eq:linear expansion global}
        \|\bigl(\varphi-\upvarphi[\alpha_N,v_N,Q_N]\bigr)\circ S_r\|_
        {L^2(\Omega_N^\beta\cap\sC_1)}
        \le C_N\|\varphi\|_{L^2(\Omega_1^\beta\cap\sC_1)}r^{3.9}
        \qquad\forall\,r\in(0,r_N),
    \end{equation}
    and then prove that the triple is independent of $N$. We split the proof in four steps.

    \smallskip\noindent\textbf{Step 1.}
    We prove that approximations at finitely many consecutive dyadic
    scales yield an approximation at the next scale, together with
    an estimate for the change in the coefficients.

    Let $M$ and $R_*\ge N$ be large, to be fixed below.    We claim that there exists $\eta_*>0$ such that, for every $\eta\in(0,\eta_*)$, there is
    $r_\eta\in(0,1)$ with the following property. For $r\in(0,r_\eta)$,
    set $r_m\coloneqq2^{-m}r$, and let
    \[
        \upvarphi^{(m)}
          \coloneqq\upvarphi[\alpha^{(m)},v^{(m)},Q^{(m)}]
        \qquad\forall\,m\in\{0,\ldots,M\}.
    \]
    If
    \begin{equation}\label{eq:linear finite scale hypothesis}
        \|\bigl(\varphi-\upvarphi^{(m)}\bigr)\circ S_{r_m}\|_
        {L^2(\Omega_{R_*}^\beta\cap\sC_1)}
        \le\eta r_m^{3.9}
        \qquad\forall\,m\in\{0,\ldots,M\},
    \end{equation}
    and
    \[
        |\alpha^{(m)}|+|v^{(m)}|
          +\|Q^{(m)}\|_{L^2(\sC_1)}\le\tfrac12
        \qquad\forall\,m\in\{0,\ldots,M\},
    \]
    then there is an Ansatz
    $\upvarphi^{(M+1)}
      \coloneqq\upvarphi[\alpha^{(M+1)},v^{(M+1)},Q^{(M+1)}]$
    such that
    \[
        \|\bigl(\varphi-\upvarphi^{(M+1)}\bigr)\circ S_{r_{M+1}}\|_
        {L^2(\Omega_{R_*}^\beta\cap\sC_1)}
        \le\eta r_{M+1}^{3.9},
    \]
    and
    \begin{equation}\label{eq:linear parameter increment}
        |\alpha^{(M+1)}-\alpha^{(M)}|
          +r_M|v^{(M+1)}-v^{(M)}|
          +r_M^2\|Q^{(M+1)}-Q^{(M)}\|_{L^2(\sC_1)}
          \le C_*\eta r_M^{2.9},
    \end{equation}
    for some $C_*=C_*(n,\beta)>1$.
    
    To prove this fact we first note that, after decreasing $r_\eta$, we may assume that
    \[
        S_{r_M}\bigl(\Omega_{2^{\gamma_\beta M}R_*}^\beta
            \cap\sC_{2^M}\bigr)
        \subset S_r(\Omega_{R_*}^\beta\cap\sC_1)
        \subset\Omega_1^\beta\cap\sC_1.
    \]
    For $0\le m<M$, \eqref{eq:linear finite scale hypothesis} and
    \Cref{lem:ansatz linear vs parameters} give
    \begin{equation}\label{eq:parameter ansatz}
        |\alpha^{(m+1)}-\alpha^{(m)}|
          +r_m|v^{(m+1)}-v^{(m)}|
          +r_m^2\|Q^{(m+1)}-Q^{(m)}\|_{L^2(\sC_1)}
          \le C\eta r_m^{2.9},
    \end{equation}
    so summing \eqref{eq:parameter ansatz} and using
    \eqref{eq:linear background estimates} we obtain
    \[
        \|\bigl(\varphi-\upvarphi^{(M)}\bigr)\circ S_{r_m}\|_
        {L^2(\Omega_{R_*}^\beta\cap\sC_1)}
        \le C\eta r_m^{3.9}
        \qquad\forall\,m\in\{0,\ldots,M\}.
    \]
    Set $\Lambda\coloneqq M-1\ge8$, assume $R_*\ge\Lambda$, and write
    $D_\Lambda\coloneqq B_\Lambda\times(-2,-1)$.
    Decompose $D_\Lambda$ into its intersection with
    $B_2\times(-2,-1)$ and dyadic spatial annuli. On the first set
    use the preceding estimate at radius $2r_M$, and on the annulus
    $2^{j-1}<|x|\le2^j$ use it at radius $2^jr_M=r_{M-j}$.
    The rescaled sets lie in $\Omega_{R_*}^\beta\cap\sC_1$, since
    $R_*\ge\Lambda$. Thus, using the Gaussian decay,
    \[
        \|\bigl(\varphi-\upvarphi^{(M)}\bigr)\circ S_{r_M}
             \sqrt G\|_{L^2(D_\Lambda)}^2
       \le C\eta^2r_M^{7.8}
            \sum_{j\ge1}2^{j(n+2+7.8)}e^{-c4^j}\le C\eta^2r_M^{7.8}.
    \]
    Now choose the coefficients of $\upvarphi^{(M+1)}$ by minimizing
    \[
        \int_{B_\Lambda\times(-2,-1)}
        |(\varphi-\upvarphi[\alpha,v,Q])\circ S_{r_M}|^2G
    \]
    over all triples $(\alpha,v,Q)$. Since
    $S_{\sqrt2}(K)\subset D_\Lambda$ and $G$ is bounded below on
    $S_{\sqrt2}(K)$, the minimizing property gives
    \[
    \begin{split}
        \|\bigl(\upvarphi^{(M+1)}-\upvarphi^{(M)}\bigr)
             \circ S_{\sqrt2r_M}\|_{L^2(K)}
        &\le C\sum_{i=M,M+1}
             \|\bigl(\varphi-\upvarphi^{(i)}\bigr)
                  \circ S_{r_M}\sqrt G\|_{L^2(D_\Lambda)}\\
        &\le 2C\|\bigl(\varphi-\upvarphi^{(M)}\bigr)
                  \circ S_{r_M}\sqrt G\|_{L^2(D_\Lambda)}
             \le C\eta r_M^{3.9}.
    \end{split}
    \]
    Also, applying \Cref{lem:ansatz linear vs parameters} with
    $r=\sqrt2r_M$ gives \eqref{eq:linear parameter increment}, which
    combined with \eqref{eq:linear finite scale hypothesis} and
    \eqref{eq:parameter ansatz} yields
    \begin{equation}\label{eq:linear comparison new ansatz}
        \|\bigl(\varphi-\upvarphi^{(M+1)}\bigr)\circ S_{r_m}\|_
        {L^2(\Omega_{R_*}^\beta\cap\sC_1)}
        \le C_*\eta r_m^{3.9}
        \qquad\forall\,m\in\{0,\ldots,M\}
    \end{equation}
    and
    \[
        |\alpha^{(M+1)}|+|v^{(M+1)}|
          +\|Q^{(M+1)}\|_{L^2(\sC_1)}\le1.
    \]
    Now, set
    \[
        w\coloneqq\frac{(\varphi-\upvarphi^{(M+1)})\circ S_{r_M}}
                {C_*\eta r_M^{3.9}},
        \qquad \bar u\coloneqq u_{r_M}=r_M^{-2}u\circ S_{r_M}.
    \]
    By \eqref{eq:cubic flatness},
    \[
        \left\|\bar u-\tfrac12(x_n)_+^2\right\|_
          {L^2(B_\Lambda\times(-2,-1))}
        \le C_\Lambda r_M.
    \]
    Also, by the minimality of $\upvarphi^{(M+1)}$ and using the scaling identity
    \begin{equation}\label{eq:scaled linear ansatz}
        \upvarphi_u[\alpha,v,Q]\circ S_r
        =r\,\upvarphi_{u_r}[\alpha,rv,r^2Q],
        \qquad u_r\coloneqq r^{-2}u\circ S_r,
    \end{equation}
    we obtain
    \[
        \int_{B_\Lambda\times(-2,-1)}
        w\,\upvarphi_{\bar u}[\widetilde\alpha,
          \widetilde v,\widetilde Q]G=0
    \]
    for every $\widetilde\alpha\in\R$, $\widetilde v\perp\bm e_n$, and
    every parabolically $2$-homogeneous polynomial $\widetilde Q$ satisfying
    $\Heat(x_n\widetilde Q)=0$. Thus
    \eqref{eq:assumption linear key lemma 3} holds.
    
    Next, for every $m\in\{0,\ldots,M\}$, \Cref{lem:pde linear ansatz} and the bound
    $|S_{2^mr_M}(x,t)|_\pb\le \sqrt2\cdot2^mr_M$ in $\sC_1$ give
    \[
        |\Heat w_{2^m}|
        =\frac{(2^mr_M)^2}{C_*\eta r_M^{3.9}}
          |(\Heat\upvarphi^{(M+1)})\circ S_{2^mr_M}|
        \le \bar C\eta^{-1}r_M^{0.1}2^{4m}
        \quad\text{in }\{u_{2^mr_M}>0\}\cap\Omega_{R_*}^\beta\cap\sC_1,
    \]
    while \Cref{lem:smoothness close maximal stratum} gives
    \[
        \partial\{\bar u>0\}\cap\Omega_{R_*}^\beta
          \cap\bigl(B_\Lambda\times(-\Lambda,-\Lambda^{-1})\bigr)
        \subset\{|x_n|\le C_\Lambda r_M\}.
    \]
    Note that, after decreasing $r_\eta$, we can assume that
    \[
        C_\Lambda2^{-M}r_\eta\le(2\Lambda)^{-1},
        \qquad
        \bar C\eta^{-1}(2^{-M}r_\eta)^{0.1}2^{4M}
        \le\Lambda^{-1},
    \]
therefore
    \[
        \left\|\bar u-\tfrac12(x_n)_+^2\right\|_
          {L^2(B_\Lambda\times(-2,-1))}\le\Lambda^{-1},
        \qquad
        \partial\{\bar u>0\}\cap\Omega_{R_*}^\beta
          \cap\bigl(B_\Lambda\times(-\Lambda,-\Lambda^{-1})\bigr)
        \subset\{|x_n|<\Lambda^{-1}\}.
    \]
    Since $w=0$ in $\{\bar u=0\}$, rescaling the estimate for $m=M$ gives
    \eqref{eq:assumption linear key lemma 2}. Moreover, by construction,
    \[
        \|w_{2^m}\|_{L^2(\Omega_{R_*}^\beta\cap\sC_1)}
        \le(2^m)^{3.9}
        \qquad\forall\,m\in\{0,\ldots,M\}.
    \]
    Since $(w_{2^{m+1}})_{1/2}=w_{2^m}$ and
    $|\Heat w_{2^{m+1}}|\le\Lambda^{-1}\le1$,
    taking $R_*\ge R_0$ and applying \Cref{lem:caccioppoli linear}
    to $w_{2^{m+1}}$ (up to a change of variables) gives, for
    $m\in\{0,\ldots,\Lambda\}$,
    \[
        \left(\int_{\Omega_{R_*}^\beta\cap\sC_1}
        \bigl(|w_{2^m}|^2+|\nabla w_{2^m}|^2
          +|\partial_t w_{2^m}|^2\bigr)\right)^{1/2}
        \le C_0(2^m)^{3.9},
    \]
    with $C_0=C_0(n,\beta)>1$, which proves
    \eqref{eq:assumption linear key lemma 1}.
    We now choose $L\ge\max\{2,2^{3.9}C_*\}$ and fix $\Lambda$ as in
    \Cref{lem:improvement linear key lemma} for these $C_0$ and $L$,
    thereby fixing $M=\Lambda+1$.
    Taking $R_*\ge\max\{N,\Lambda,R_0\}$,
    \Cref{lem:improvement linear key lemma} gives
    \[
        \|\bigl(\varphi-\upvarphi^{(M+1)}\bigr)
              \circ S_{r_{M+1}}\|_{L^2(\Omega_{R_*}^\beta\cap\sC_1)}
        \le C_*L^{-1}\eta r_M^{3.9}
        \le\eta r_{M+1}^{3.9},
    \]
which proves the finite-scale implication.

    \smallskip\noindent\textbf{Step 2.}
    We normalize $\varphi$ and initialize the iteration.

    Choose $\eta\in(0,\eta_*)$ so that
    $3C_*\eta\sum_{k=M}^\infty(2^{-0.9})^k\le\tfrac12$, and fix
    $r_0\in(0,r_\eta)$ with
    $S_{r_0}(\Omega_{R_*}^\beta\cap\sC_1)\subset\Omega_1^\beta\cap\sC_1$.
    Set $r_m\coloneqq2^{-m}r_0$ for all $m\ge0$, fix $\sigma_1 \in (0,\eta r_M^{3.9+(n+2)/2})$, and 
define
    \[
        \widehat\varphi\coloneqq
        \frac{\sigma_1\varphi}
             {\|\varphi\|_{L^2(\Omega_1^\beta\cap\sC_1)}}.
    \]
    By linearity, $\widehat\varphi$ solves the same linearized problem
    with background $u$. Also, by the definition of $\sigma_1$ and a change of variables 
        \[
        \|\widehat\varphi\circ S_{r_m}\|_
        {L^2(\Omega_{R_*}^\beta\cap\sC_1)}
        \le\sigma_1r_m^{-(n+2)/2}\le\eta r_m^{3.9}
        \qquad\forall\,m\in\{0,\ldots,M\},
    \]
    (recall that $r_m\ge r_M$). Thus Step 1 applies to $\widehat\varphi$ with $\upvarphi^{(m)}=0$ for
    $0\le m\le M$.

    \smallskip\noindent\textbf{Step 3.}
    We now iterate Step 1 along consecutive sequences of length $M+1$ and find the desired expansion of $\varphi$.
    
    At the $j$-th application, since $r_j<r_\eta$ and $2^{-m}r_j=r_{j+m}$, we use the radius $r_j$ and the Ans\"atze
    indexed by $j,\ldots,j+M$, obtaining $\upvarphi^{(j+M+1)}$.
 Then, with $k=j+M$, \eqref{eq:linear parameter increment} gives
    \[
        |\alpha^{(k+1)}-\alpha^{(k)}|+
        r_k|v^{(k+1)}-v^{(k)}|+
        r_k^2 \|Q^{(k+1)}-Q^{(k)}\|_{L^2(\sC_1)}
            \leq  C_*\eta r_k^{2.9}.
    \]
    Since the initial triples vanish and $r_0<1$, summing yields
    \[
        |\alpha^{(\ell)}|+|v^{(\ell)}|
        +\|Q^{(\ell)}\|_{L^2(\sC_1)}
        \le3C_*\eta\sum_{k=M}^{\ell-1}(2^{-0.9})^k\le\tfrac12
        \qquad\forall\,\ell\ge M+1.
    \]
    Hence the coefficient bound required in Step 1 is preserved,
    so the iteration continues for every $j\ge0$ and 
the coefficients converge to a triple
    $(\widehat\alpha_N,\widehat v_N,\widehat Q_N)$. Also, summing the tails gives
    \[
        |\widehat\alpha_N-\alpha^{(j)}|
        +r_j|\widehat v_N-v^{(j)}|
        +r_j^2\|\widehat Q_N-Q^{(j)}\|_{L^2(\sC_1)}
        \le C\eta r_j^{2.9}.
    \]
    Combining this with the inductive approximation and
    \eqref{eq:linear background estimates} yields
    \[
        \|\bigl(\widehat\varphi
          -\upvarphi[\widehat\alpha_N,\widehat v_N,\widehat Q_N]\bigr)
          \circ S_{r_j}\|_{L^2(\Omega_{R_*}^\beta\cap\sC_1)}
        \le C\eta r_j^{3.9} \qquad \forall\,j \geq 0.
    \]
       Now set
    \[
        (\alpha_N,v_N,Q_N)\coloneqq
        \frac{\|\varphi\|_{L^2(\Omega_1^\beta\cap\sC_1)}}{\sigma_1}
        (\widehat\alpha_N,\widehat v_N,\widehat Q_N),
        \qquad r_N\coloneqq r_0.
    \]
    By linearity of the Ansatz in its parameters, multiplying the last
    estimate by $\|\varphi\|_{L^2(\Omega_1^\beta\cap\sC_1)}/\sigma_1$
    proves \eqref{eq:linear expansion global}. Also, the coefficient estimates above  give
    \[
        |\alpha_N|+|v_N|+\|Q_N\|_{L^2(\sC_1)}
        \le C_N\|\varphi\|_{L^2(\Omega_1^\beta\cap\sC_1)}.
    \]
    
    \smallskip\noindent\textbf{Step 4.}
    For each $N\ge1$, the preceding construction gives a triple
    $(\alpha_N,v_N,Q_N)$ satisfying \eqref{eq:linear expansion global},
    but which a priori depends on $N$. We now show that these triples coincide.
    Indeed, any two triples satisfying \eqref{eq:linear expansion global},
    possibly for different values of $N$, must coincide by
    \Cref{lem:ansatz linear vs parameters}, since both expansions hold
    on $\Omega_1^\beta\cap\sC_1$.
    Thus the triple is unique and independent of $N$; denoting it by
    $(\alpha_*,v_*,Q_*)$ and taking $N=1$ in the coefficient bound
    completes the proof.
\end{proof}

\subsection{Runge approximation}
Our goal now is to prove a Runge approximation result for the linearized problem
\eqref{eq:linearized problem}. For $t\in(0,T)$ set
\[
    O_t\coloneqq\{x\in\Omega:u(x,t)>0\}.
\]
Also, we fix a set $D\ssubset Q$ and assume that
\begin{equation}\label{assumption:a1}
    D=D'\times(a,b),\quad\text{where }0<a<b<T\text{ and }
    D'\ssubset\Omega\text{ is smooth}.
\end{equation}
Also, setting
\[
    K_t\coloneqq
    \begin{cases}
        \overline{D'},&t\in[a,b],\\
        \varnothing,&t\notin[a,b],
    \end{cases}
\]
we further assume that, for every $t\in(0,T)$ and every connected component
$\mathcal C$ of $O_t\setminus K_t$, 
\begin{equation}\label{assumption:a2}
        \text{the set }\partial\mathcal C\cap\partial\Omega
        \cap\{x\in\overline\Omega:u(x,t)>0\}
        \text{ has nonempty relative interior in }\partial\Omega.
\end{equation}

\begin{proposition}\label{prop:runge}
    Let $D\ssubset Q$ satisfy assumptions \eqref{assumption:a1} and
    \eqref{assumption:a2}. Given $\eps>0$ and $w\in H^1(D)$ solving
    \begin{equation}\label{eq:runge linearized}
        \begin{cases}
            \Heat w=0
              &\text{in }D\cap\{u>0\},\\
            w=0
              &\text{in }D\cap\{u=0\},
        \end{cases}
    \end{equation}
    there is
    $g\in C^\infty_c(\{u>0\}\cap
       (\partial\Omega\times(0,T)))$ such that
    \[
        \|\varphi_{g,u}-w\|_{L^2(D)}<\eps.
    \]
\end{proposition}
The proof adapts the classical duality argument for Runge approximation
(see \cite[Section~2]{runge}) to the set $\{u>0\}$.
However, this adaptation is nontrivial due to the singular part of the
free boundary, and we need to carefully exploit the monotonicity of $u$
and the zero capacity of the singular set.

For each $t$ we regard $H^1_0(O_t)$ as a closed subspace of
$H^1_0(\Omega)$ by zero extension. Since $\partial_tu\ge0$, these spaces are
increasing in $t$. We write
\[
    \mathbb V_u
       \coloneqq\overline{C^\infty_c(\{u>0\}\cap Q)}^{\,L^2(0,T;H^1_0(\Omega))},
\qquad
    \mathbb W_u
       \coloneqq\{\eta\in\mathbb V_u:
          \partial_t\eta\in L^2(0,T;H^{-1}(\Omega))\}.
\]
By \cite[Section~5.9.2]{Evans}, every
$\eta\in\mathbb W_u$ has a unique representative in $C([0,T];L^2(\Omega))$.
We always use this representative and set
\[
    \mathbb W_{u,0}
       \coloneqq\{\eta\in\mathbb W_u:\eta(0)=0\text{ in }L^2(\Omega)\}.
\]
Fix $\rho\in C_c^\infty((1,2))$ nonnegative and with integral one, and
define $\rho_h(s)\coloneqq h^{-1}\rho(s/h)$. If
$f\in L^2(0,T;X)$ where $X$ is either $H^1_0(\Omega)$ or $H^{-1}(\Omega)$, we
extend $f$ by zero to negative times without changing notation.
Then,
\begin{equation}\label{eq:delayed time convolution}
    f_h(t)\coloneqq\int_{\mathbb R}\rho_h(s)f(t-s)\,ds,
    \qquad
    \|f_h-f\|_{L^2(0,T;X)}
    \le\sup_{h<r<2h}
       \|f(\cdot-r)-f\|_{L^2(0,T;X)}
       \longrightarrow0.
\end{equation}
Also, if $f(t)\in H^1_0(O_t)$ for a.e. $t$, the monotonicity of these spaces gives
\begin{equation}\label{eq:delayed time convolution phase}
    f_h(t)\in H^1_0(O_{t-h})\subset H^1_0(O_t)\quad\text{for a.e. }t>h,
    \qquad f_h=0\quad\text{in }(0,h).
\end{equation}
We note that, thanks to \Cref{lem:zero capacity singular set},
\[
    O_t=\bigcup_{s<t}O_s,
    \qquad
    H^1_0(O_t)=\overline{\bigcup_{s<t}H^1_0(O_s)}^{H^1_0(\Omega)}.
\]
Consequently,\footnote{Indeed, for $v$ in the right-hand side, the nontrivial inclusion follows
by applying \eqref{eq:delayed time convolution} to $v$ and approximating
$v_h$ by finite sums $\sum_j\eta_j(t)w_j(x)$, where
$t_j\in(h,T)$, $\eta_j\in C_c^\infty(0,T)$,
$\supp\eta_j\Subset(t_j-h,T)$,
$v_h(t_j)\in H^1_0(O_{t_j-h})$, and each
$w_j\in C_c^\infty(O_{t_j-h})$ is chosen arbitrarily close to
$v_h(t_j)$ in $H^1_0(\Omega)$.
}
\begin{equation}\label{eq:slice characterization energy space}
    \mathbb V_u
    =\{v\in L^2(0,T;H^1_0(\Omega)):
         v(\cdot,t)\in H^1_0(O_t)\text{ for a.e. }t\}.
\end{equation}

We will need the following density result.

\begin{lemma}\label{lem:key density lemma}
    The vector space
    \[
        \mathcal V\coloneqq\{\eta\in\mathbb W_{u,0}
          \cap C^\infty((0,T);C^\infty_c(\Omega)):
          \eta=0\text{ in }\{u=0\},\
          \eta=0\text{ for }0<t<\delta_\eta
          \text{ for some }\delta_\eta>0\}
    \]
    is dense in $\mathbb W_{u,0}$ for the norm
    \[
        \|\eta\|_{\mathcal G}
        \coloneqq\|\eta\|_{L^2(0,T;H^1_0(\Omega))}
          +\|\partial_t\eta\|_{L^2(0,T;H^{-1}(\Omega))}.
    \]
\end{lemma}
\begin{proof}
    Fix $\eta\in\mathbb W_{u,0}$ and note that, by    \eqref{eq:slice characterization energy space},
    $\eta(t)\in H^1_0(O_t)$ for a.e. $t$.

    Extending $\eta$ by zero to negative times, since $\eta(0)=0$, its
    distributional time derivative is the zero extension of
    $\partial_t\eta$.
    Let $\eta_h$ be defined by \eqref{eq:delayed time convolution}. Then
    \[
        \partial_t\eta_h=\rho_h*\partial_t\eta
                   =\rho_h'*\eta.
    \]
    Applying \eqref{eq:delayed time convolution} to $\eta$ in $H^1_0(\Omega)$
    and to $\partial_t\eta$ in $H^{-1}(\Omega)$, and using
    \eqref{eq:delayed time convolution phase}, we obtain
    \[
        \eta_h\in H^1(0,T;H^1_0(\Omega)),\qquad
        \eta_h(t)\in H^1_0(O_t)\quad\text{for a.e. }t,\qquad
        \eta_h=0\quad\text{in }(0,h),\qquad
        \|\eta_h-\eta\|_{\mathcal G}\longrightarrow0.
    \]
    It is therefore enough to approximate $\eta\in H^1(Q)$ with $\eta = 0$ for $0<t<\delta$, for some $\delta>0$.
    
    Up to a truncation and a spatial cutoff,\footnote{The truncations
    $\eta_M:=\max\{-M,\min\{\eta,M\}\}$ converge to $\eta$ in $H^1(Q)$ as $M\to \infty$. Also, for $d(x)\coloneqq\operatorname{dist}(x,\partial\Omega)$,
    choose $\zeta_j\in C_c^\infty(\Omega)$ with $0\le\zeta_j\le1$,
    $\zeta_j=0$ for $d\le j^{-1}$, $\zeta_j=1$ for $d\ge2j^{-1}$,
    and $|\nabla\zeta_j|\le C/d$. Since
    $\eta\in L^2(0,T;H^1_0(\Omega))$, Hardy's inequality gives
    $\eta/d\in L^2(Q)$, hence
    \[
        \|\eta\nabla\zeta_j\|_{L^2(Q)}^2
        \le C\int_0^T\int_{\{d<2/j\}}\frac{|\eta|^2}{d^2}
        \longrightarrow0.
    \]
    Thus, $\zeta_j\eta\to\eta$ in $H^1(Q)$.}
    we may further suppose that
    \[
        \eta\in H^1(Q)\cap L^\infty(Q),\qquad \supp\eta\subset K\times[\delta,T]\quad\text{for some }K\ssubset\Omega\text{ and some }\delta>0,
    \]
    since both approximations converge in $H^1(Q)$ and therefore in
    $\|\cdot\|_{\mathcal G}$.
    Now, by \Cref{lem:zero capacity singular set} (more precisely, applied to an extension of $u$ to $(0,T+\delta)$) we can choose $c_j\in C^\infty_c(\mathbb R^{n+1})$, with $0\le c_j\le1$, equal to one near the singular free boundary meeting
    $\supp\eta$, such that
    \[
        \|c_j\|_{H^1(\mathbb R^{n+1})}\longrightarrow0,
        \qquad
        |\operatorname{supp}c_j|\longrightarrow0,
    \]
    and since $\eta\in H^1(Q)\cap L^\infty(Q)$ it follows that
        \[
        \|c_j\eta\|_{H^1(Q)}
        \le C\|\eta\|_\infty\|c_j\|_{H^1(\mathbb R^{n+1})}
          +C\|\eta\|_{H^1(Q\cap\supp c_j)}
        \longrightarrow0.
    \]
    Thus, up to replacing $\eta$ by $(1-c_j)\eta$, we may suppose that $\eta$ vanishes
    near the singular free boundary.

    Now, applying the first part of \Cref{lem:zero capacity singular set} (applied again to a time-extension of $u$), the part of the free boundary meeting $\supp\eta$ is covered by finitely many open sets $U_i$ in which, after a spatial rotation,
    \[
        \{u>0\}\cap Q\cap U_i
          =U_i\cap Q\cap\{x_n>\psi_i(x',t)\}.
    \]
    Adding finitely many open sets compactly contained inside $\{u>0\}$ to obtain a  finite covering of $\supp\eta$, and then choosing a partition of unity $\{\chi_\ell\}$ subordinate to this covering, it suffices to approximate
    $f_\ell\coloneqq\chi_\ell\eta$.

    If $\supp\chi_\ell$ is compactly contained in $\{u>0\}$, the existence of a smooth approximation follows by
    a standard space-time mollification with a kernel compactly supported inside
    $\mathbb R^n\times(0,\infty)$.
    
    If instead $\supp\chi_\ell\subset U_i$ for some $i=i(\ell)$, extend $f_\ell$ by zero outside $U_i\cap Q$, so that $f_\ell\in H^1(\mathbb R^n\times({-\infty},T))$.
    Write $z=(z',z_n)$, let $L$ be a Lipschitz constant for $\psi_i$
    in $(x',t)$, choose $0\le\varrho\in
    C_c^\infty(\mathbb R^{n+1})$ with integral one and
    \[
        \supp\varrho\Subset
        \{(z,s):s>0,\ z_n>L(|z'|+s)\},
    \]
    and define the smooth functions
    \[
        f_{\ell,\varepsilon}(x,t)
        \coloneqq\int_{\mathbb R^{n+1}}\varrho_\varepsilon(z,s)
             f_\ell(x-z,t-s)\,dz\,ds.
    \]
    By construction $f_{\ell,\varepsilon}\in\mathcal V$ for $\varepsilon$ small, and
    \[
        \|f_{\ell,\varepsilon}-f_\ell\|_{\mathcal G}
        \le C\|f_{\ell,\varepsilon}-f_\ell\|_
             {H^1(\mathbb R^n\times(0,T))}
        \longrightarrow0.
    \]
    This concludes the proof  of the lemma.
\end{proof}

Next, we discuss the well-posedness of the dual problem
\[
    \partial_t\psi+\Delta\psi=F\quad\text{in }\{u>0\}\cap Q,
    \qquad \psi(T)=0,
\]
with zero lateral boundary data and $\psi=0$ in $\{u=0\}\cap Q$.
For
$\psi\in L^\infty(0,T;L^2(\Omega))$, we interpret the terminal condition
$\psi(T)=0$ as
\begin{equation}\label{eq:dual terminal condition}
    \lim_{t\uparrow T}\operatorname*{ess\,sup}_{s\in(t,T)}
    \|\psi(s)\|_{L^2(\Omega)}=0.
\end{equation}

\begin{lemma}\label{lem:dual problem}
    Given $F\in L^2(0,T;H^{-1}(\Omega))$, there is a unique
    \[
        \psi\in L^\infty(0,T;L^2(\Omega))\cap\mathbb V_u
    \]
    satisfying \eqref{eq:dual terminal condition} and
    \begin{equation}\label{eq:dual transposition}
        \int_0^T\left(
          \int_\Omega\nabla\psi\cdot\nabla\eta
          +\langle\partial_t\eta,\psi\rangle\right)\,dt
        =-\int_0^T\langle F,\eta\rangle\,dt
        \qquad \forall\,\eta\in\mathbb W_{u,0}.
    \end{equation}
    Moreover, for every $t\in[0,T)$,
    \begin{equation}\label{eq:caccioppoli dual}
        \operatorname*{ess\,sup}_{s\in(t,T)}
          \|\psi(s)\|_{L^2(\Omega)}^2
        +\int_t^T\|\nabla\psi(s)\|_{L^2(\Omega)}^2\,ds
        \le C_\Omega\int_t^T\|F(s)\|_{H^{-1}(\Omega)}^2\,ds.    \end{equation}
\end{lemma}
\begin{proof}
    The proof is classical and uses a backward Euler discretization. However, since the setup is not standard (as the spaces $H^1_0(O_t)$ vary with time), we prove the argument in detail. 
        
     For
    $h\coloneqq T/N$ and $t_i\coloneqq ih$, put $\psi_N\coloneqq0$ and, successively for
    $i=N-1,\ldots,0$, let $\psi_i\in H^1_0(O_{t_i})$ solve
    \begin{equation}\label{eq:dual backward Euler}
        (\psi_i-\psi_{i+1},v)_{L^2(\Omega)}
        +h\int_\Omega\nabla\psi_i\cdot\nabla v
        =-\int_{t_i}^{t_{i+1}}\langle F,v\rangle\,dt
        \qquad \forall\,v\in H^1_0(O_{t_i})
    \end{equation}
    (these functions exist uniquely by Lax--Milgram). Testing with
    $v=\psi_i$ and summing in $i$ gives, for every
    $k\in\{0,\ldots,N-1\}$,
    \begin{equation}\label{eq:dual discrete energy}
        \max_{k\le j<N}\|\psi_j\|_{L^2(\Omega)}^2
        +\sum_{i=k}^{N-1}\|\psi_i-\psi_{i+1}\|_{L^2(\Omega)}^2
        +h\sum_{i=k}^{N-1}\|\nabla\psi_i\|_{L^2(\Omega)}^2
       \le
        C\int_{t_k}^T\|F(t)\|_{H^{-1}(\Omega)}^2\,dt.
    \end{equation}
Now, for $t\in(t_i,t_{i+1})$, define
    \[
        \psi_h^-(t)\coloneqq\psi_i,\qquad \psi_h^+(t)\coloneqq\psi_{i+1}.
    \]
    Since $H^1_0(O_{t_i})\subset H^1_0(O_t)$ for $t>t_i$,
    \eqref{eq:slice characterization energy space} gives
    $\psi_h^-\in\mathbb V_u$. Thus, by \eqref{eq:dual discrete energy}, after
    passing to a subsequence we get
    \[
        \psi_h^-\rightharpoonup\psi,
        \qquad \psi_h^+\rightharpoonup\psi
        \quad\text{in }L^2(0,T;H^1_0(\Omega)),
        \qquad
        \psi_h^-\overset{*}{\rightharpoonup}\psi
        \quad\text{in }L^\infty(0,T;L^2(\Omega)),
    \]
    where the two weak limits agree because
    \begin{equation}\label{eq:dual mesh shift}
        \|\psi_h^+-\psi_h^-\|_{L^2(Q)}^2
        =h\sum_{i=0}^{N-1}\|\psi_i-\psi_{i+1}\|_{L^2(\Omega)}^2
        \longrightarrow0.
    \end{equation}
    Since $\mathbb V_u$ is weakly closed, this implies that
    $\psi\in\mathbb V_u$.

    For $\eta\in\mathcal V$, put
    \[
        \eta_0\coloneqq0,\qquad
        \eta_i\coloneqq h^{-1}\int_{t_{i-1}}^{t_i}\eta(s)\,ds
        \quad\text{for }i\in\{1,\ldots,N\},
    \]
    and set, on $(t_i,t_{i+1})$,
    \[
        \eta_h^-\coloneqq\eta_i,
        \qquad D_h\eta\coloneqq\frac{\eta_{i+1}-\eta_i}{h}.
    \]
    Since the spaces  
    $H^1_0(O_t)$ are increasing in time, it follows that
 $\eta_i\in H^1_0(O_{t_i})$, so these are
    admissible tests in \eqref{eq:dual backward Euler}. Thus, since
    $\eta_0=\psi_N=0$ we have
        \[
        \sum_{i=0}^{N-1}(\psi_i-\psi_{i+1},\eta_i)_{L^2(\Omega)}
        =\sum_{i=0}^{N-1}
          (\psi_{i+1},\eta_{i+1}-\eta_i)_{L^2(\Omega)},
    \]
    and therefore
    \[
        \int_0^T\langle D_h\eta,\psi_h^+\rangle\,dt
        +\int_0^T\int_\Omega
             \nabla\psi_h^-\cdot\nabla\eta_h^-\,dt
        =-\int_0^T\langle F,\eta_h^-\rangle\,dt.
    \]
    Thus, since \[
        \eta_h^-\longrightarrow\eta
          \quad\text{in }L^2(0,T;H^1_0(\Omega)),
        \qquad
        D_h\eta\longrightarrow\partial_t\eta
          \quad\text{in }L^2(0,T;H^{-1}(\Omega)),
    \]
    we may pass to the limit in \eqref{eq:dual backward Euler}, obtaining
    \eqref{eq:dual transposition} first for $\eta\in\mathcal V$ and then
    for $\eta\in\mathbb W_{u,0}$ by \Cref{lem:key density lemma}.

    Given $t<T$, choose $k=k(h)$ with $t_k\le t<t_{k+1}$; by weak
    lower semicontinuity, \eqref{eq:dual discrete energy} gives
    \eqref{eq:caccioppoli dual}, from which
    \eqref{eq:dual terminal condition} follows as $t\uparrow T$.

    To prove uniqueness, suppose that $F=0$ and define $\eta(t,\cdot)\coloneqq\int_0^t\psi(s,\cdot)\,ds$.
    The monotonicity of $H^1_0(O_t)$ and \eqref{eq:slice characterization energy space} give $\eta\in\mathbb W_{u,0}$, with $\partial_t\eta=\psi$.
    Testing \eqref{eq:dual transposition} with $\eta$ yields
    \[
        0 = \int_0^T\|\psi(t)\|_{L^2(\Omega)}^2\,dt + \frac12\left\|\int_0^T\nabla\psi(t)\,dt\right\|_{L^2(\Omega)}^2,
    \]
    thus $\psi=0$, as desired.
\end{proof}
We can now prove \Cref{prop:runge}.

\begin{proof}[Proof of \Cref{prop:runge}]
    Set
    \[
        \mathcal X\coloneqq C^\infty_c\bigl(\{u>0\}\cap
           (\partial\Omega\times(0,T))\bigr),
    \]
    and let $\mathcal Y$ be the class of functions in $H^1(D)$ satisfying \eqref{eq:runge linearized}. Define
    \[
        A\colon\mathcal X\to \mathcal Y\subset L^2(D),\qquad
        Ag\coloneqq\varphi_{g,u}|_D.
    \]
    We want to prove that $A\mathcal X$ is dense in $\mathcal Y$. To this aim, it suffices to prove that, for every $f\in L^2(D)$,
    \begin{equation}\label{eq:proof runge orthogonal to image}
        f\perp_{L^2(D)} A\mathcal X\quad\Longrightarrow\quad f\perp_{L^2(D)}\mathcal Y.
    \end{equation}
    So, fix $f\perp A\mathcal X$.
    Since the elements of $A\mathcal X$ and $\mathcal Y$
    vanish in the contact set, we may replace $f$ by
    $f\mathbf1_{D\cap\{u>0\}}$, and we denote by $F$ its extension by zero to
    $Q$. 
    
    \smallskip\noindent\textbf{Step 1.}
    Let $\psi$ be given by \Cref{lem:dual problem}, so that
        \begin{equation}\label{eq:runge transposition}
        \int_Q\nabla\psi\cdot\nabla\eta
        +\int_0^T\langle\partial_t\eta,\psi\rangle\,dt
        +\int_QF\eta=0
        \qquad\forall\,\eta\in\mathbb W_{u,0}.
    \end{equation}
    Taking compactly supported test functions in $\{u>0\}\cap Q$, this gives
    \begin{equation}\label{eq:runge adjoint equation}
        \partial_t\psi+\Delta\psi=F
        \quad\text{in }\{u>0\}\cap Q.
    \end{equation}
    Now, fix $g:\partial\Omega\times(0,T)\to \R$ with $g\in\mathcal X$, choose an extension $G\in C^\infty(\overline Q)$ such that
    \[
        G=g\quad\text{on }\partial\Omega\times(0,T),\qquad
        \supp G\Subset
        \bigl(\{u>0\}\cap(\overline\Omega\times(0,T))\bigr)
        \setminus\overline D,
    \]
    and define $\eta_g\coloneqq\varphi_{g,u}-G$. Since
    $\eta_g\in H^1(Q)$ has zero lateral trace and vanishes a.e. in
    $\{u=0\}$, the smoothness of the regular free boundary,
    \Cref{lem:zero capacity singular set},
    \eqref{eq:slice characterization energy space}, and $\eta_g(\cdot,0)=0$
  give 
    $\eta_g\in\mathbb W_{u,0}$.

    Next, approximating $\psi$ in $\mathbb V_u$ by functions in
    $C^\infty_c(\{u>0\}\cap Q)$ and using the equation for
    $\varphi_{g,u}$ gives
    \begin{equation}\label{eq:runge forward identity}
        \int_Q\bigl(\nabla\psi\cdot\nabla\varphi_{g,u}
                    +\psi\,\partial_t\varphi_{g,u}\bigr)=0.
    \end{equation}
    Also, since $\supp F\subset\overline D$ and
    $\supp G\cap\overline D=\varnothing$, we have $FG=0$. Thus, taking
    $\eta=\eta_g$ in
    \eqref{eq:runge transposition} and subtracting
    \eqref{eq:runge forward identity} gives
    \begin{equation}\label{eq:runge conormal identity}
        \int_Q\bigl(\nabla\psi\cdot\nabla G+\psi\partial_tG\bigr)
        =\int_QF\varphi_{g,u}
        =\int_Df\,Ag=0.
    \end{equation}
    Note now that, on a neighborhood of $\supp G$, we have $F=0$ and
    $(\partial_t+\Delta)\psi=0$ (by \eqref{eq:runge adjoint equation}).
    Thus $G,\psi$ are smooth there and an integration by parts gives
    \[
        \int_{\partial\Omega\times(0,T)}
            \partial_{\nu_\Omega}\psi\,g
        =\int_Q\bigl(\nabla\psi\cdot\nabla G+\psi\partial_tG\bigr)=0.
    \]
    Since $g\in\mathcal X$ is arbitrary, this proves that
    $\partial_{\nu_\Omega}\psi=0$ on
    $\{u>0\}\cap(\partial\Omega\times(0,T))$. Thus, recalling that
    $\psi\in\mathbb V_u$ has zero trace there, we obtain
    \begin{equation}\label{eq:runge lateral Cauchy data}
        \psi=\partial_{\nu_\Omega}\psi=0
        \quad\text{on }
        \{u>0\}\cap(\partial\Omega\times(0,T)).
    \end{equation}

    \smallskip\noindent\textbf{Step 2.}
    We claim that
    \begin{equation}\label{eq:runge exterior vanishing}
        \psi=0\quad\text{a.e. in }\{u>0\}\cap(Q\setminus\overline D).
    \end{equation}
    Indeed, fix $t\in(0,T)\setminus\{a,b\}$ and a connected component
    $\mathscr C$ of $O_t\setminus K_t$. Assumption
    \eqref{assumption:a2} and the continuity of $u$ give a relatively open
    set $\Gamma\subset\partial\Omega$ and an interval $I$ containing $t$
    such that $\Gamma\subset\partial\mathscr C$ and
    $\Gamma\times I\subset\{u>0\}$, and by
    \eqref{eq:runge lateral Cauchy data} we have
    $\psi=\partial_{\nu_\Omega}\psi=0$ on $\Gamma\times I$. Hence, after
    reversing time, boundary unique continuation
    \cite[Theorem~3.5.1]{VessellaParabolicUC} gives
    \[
        \psi(\cdot,t)=0\quad\text{in a nonempty open set }
        V\Subset\mathscr C.
    \]
    Now, given $x\in\mathscr C$, join $V$ to $x$ by a compact path in
    $\mathscr C$. Since $t\notin\{a,b\}$, a connected neighborhood of
    this path remains in $\{u>0\}\cap(Q\setminus\overline D)$ for all times
    sufficiently close to $t$.
    Thus, the spacelike unique-continuation theorem for the heat equation
    \cite{uniqueContinuationHeat} (again applied after reversing time) gives
    $\psi(x,t)=0$. Since $x$ and $\mathscr C$ are arbitrary,
    this proves \eqref{eq:runge exterior vanishing}.

    Finally, since $\psi\in\mathbb V_u$,
    \eqref{eq:slice characterization energy space},
    \eqref{eq:runge exterior vanishing}, and the smoothness of $D'$ give
    \begin{equation}\label{eq:runge slice test space}
        \psi(\cdot,t)\in H^1_0(D'\cap O_t)
        \qquad\text{for a.e. }t\in(a,b).
    \end{equation}
    
    \smallskip\noindent\textbf{Step 3.}
    The vector field $Z\coloneqq(\nabla_x\psi,\psi) \in \R^n\times \R$ satisfies $\operatorname{div}_{x,t}Z=F\in L^2(\{u>0\}\cap Q)$ by \eqref{eq:runge adjoint equation}, so its weak normal trace belongs to $H^{-1/2}$ on every Lipschitz boundary compactly contained in $\{u>0\}\cap Q$;
    using \eqref{eq:runge exterior vanishing}, we obtain (here $N_D$ denotes the outward unit normal of $D$)
    \begin{equation}\label{eq:runge product flux}
        Z\cdot N_D=0
        \quad\text{on }\partial D\cap\{u>0\}.
    \end{equation}
    Let
    \[
        E\coloneqq D\cap\{u>0\},\qquad
        \Gamma\coloneqq(\Sigma\cap\overline D)
          \cup\bigl(\partial D'\times\{a,b\}\bigr).
    \]
    By \Cref{lem:zero capacity singular set}, and since
    $\partial D'\times\{a,b\}$ is a smooth set of codimension two,
    $\Gamma$ has zero $W^{1,2}(\mathbb R^{n+1})$-capacity. Hence there are
    $\zeta_j\in C_c^\infty(\mathbb R^{n+1})$, equal to one near
    $\Gamma$, such that
    \[
        0\le\zeta_j\le1,\qquad
        \|\zeta_j\|_{H^1(\mathbb R^{n+1})}\longrightarrow0,
        \qquad |\supp\zeta_j|\longrightarrow0.
    \]
    Also, away from $\Gamma$, the free boundary is smooth and meets the product
    faces of $D$ transversely, so $E$ is locally Lipschitz. 
    
    Given $w\in H^1(D)$ satisfying \eqref{eq:runge linearized}, we aim to show that $f\perp w$.
    Using that $w$ vanishes on the regular boundary and \eqref{eq:runge product flux}, setting $w_M:=\max\{-M,\min\{w,M\}\}$,
    \[
        0=\int_E\left[
          (1-\zeta_j)
          \bigl(Fw_M+Z\cdot\nabla_{x,t}w_M\bigr)
          -w_MZ\cdot\nabla_{x,t}\zeta_j
        \right],
    \]
   so letting $j\to\infty$ yields
   \[
     \int_Efw_M+
        \int_E\bigl(\nabla\psi\cdot\nabla w_M+\psi\partial_tw_M\bigr)=0,
   \]
   and letting $M\to+\infty$ gives
   \begin{equation}\label{eq:runge gauss green}
        \int_Efw+
        \int_E\bigl(\nabla\psi\cdot\nabla w+\psi\partial_tw\bigr)=0.
    \end{equation}
On the other hand,
    \eqref{eq:runge linearized} implies that
    \[
        \Delta_xw(\cdot,t)=\partial_tw(\cdot,t)
        \quad\text{in }D'\cap O_t,\quad \text{for a.e. }t \in (a,b).
    \]
    Also, by \eqref{eq:runge slice test space}, $\psi(\cdot,t)$ is an
    admissible test function. Hence, integrating in time implies that
    \[
        \int_E\bigl(\nabla w\cdot\nabla\psi+\partial_tw\,\psi\bigr)=0,
    \]
    which combined with \eqref{eq:runge gauss green} gives
    \[
        \int_Dfw=\int_Efw=0.
    \]
    This proves \eqref{eq:proof runge orthogonal to image}, concluding the proof.
    \end{proof}

\section{\texorpdfstring{$C^1$}{} regularity of the cubic expansion at singular points under perturbation}
\label{section:improvement of flatness}

Given a singular point $(0,0)$ in the maximal stratum, as shown in
\cite{FRS,genericStefan4D} we can write
\[
    u(x,t) = p_2(x) + p_3(x,t) + o(|x,t|_\pb^3)\qquad \text{for $t\leq0$,}
\]
where $p_2$ and $p_3$ are parabolically $2$- and $3$-homogeneous,
respectively. The goal of this section is to show that  the variation of these coefficients along a family of solutions is given
by the cubic expansion of the solution of the linearized problem. The proof should be compared with the one of \cite[Theorem 13.5]{FRS}, where the main difference is that in our case we approximate any solution, not only $\tfrac12(x_n)_+^2$. This requires introducing nonlinear approximations as in \Cref{definition:ansatz perturbation}.

Fix $\beta\in(0,1/2)$, let $\gamma_\beta$ be as in
\eqref{eq:cusp scaling exponent}, and set
$\vartheta\coloneqq2^{-\gamma_\beta}$. We consider a solution
$u\colon\Omega^{2\beta}_1\cap\sC_1\to[0,+\infty)$ of
\eqref{eq:stefan}, with $(0,0)\in\partial\{u>0\}$, satisfying
\begin{equation}\label{eq:flatness assumption}
    \|u-\tfrac12(x_n)_+^2\|_{L^\infty(\Omega^{2\beta}_1\cap\sC_r)}
    \le r^{2.5}\qquad\forall\,r\in(0,r_0)
\end{equation}
for some $r_0>0$.
Then, we consider
$\varphi\in H^1(\Omega^\beta_1\cap\sC_1)$ satisfying
\begin{equation}\label{eq:linearized problem transversality}
    \begin{cases}
        \Heat \varphi = 0   &\text{in }\{u>0\}\cap\Omega^\beta_1\cap\sC_1,\\
        \varphi = 0         &\text{in }\{u=0\}\cap\Omega^\beta_1\cap\sC_1,
    \end{cases}
\end{equation}
and assume that 
\begin{equation}\label{eq:expansion perturbation}
    \varphi(rx,r^2t) = (\alpha +rx\cdot v+ r^2Q(x,t))\partial_nu(rx,r^2t) - rx_n(v\cdot\nabla u(rx,r^2t)) + \mathcal R(rx,r^2t)\quad \text{in }\Omega^\beta_1\cap\sC_1,
\end{equation}
where $\alpha\in\R$, $v\in\{x_n=0\}$, $Q$ is a parabolically
$2$-homogeneous polynomial satisfying $\Heat(x_nQ)\equiv0$, and the
remainder $\mathcal R$ satisfies the following estimate: for every fixed
$R\ge1$, there is $r_R>0$ such that
\begin{equation}\label{eq:apply expansion linear}
    \|\mathcal R(rx,r^2t)\|_{L^2(\Omega^\beta_R\cap\sC_1)}
    \le r^{3.8}\qquad\forall\,r\in(0,r_R).
\end{equation}

We first introduce nonlinear Ansätze that encode the normal translation,
rotation, and quadratic deformation corresponding to the coefficients
$\alpha$, $v$, and $Q$.
\begin{definition}[Ansatz for a perturbed solution]\label{definition:ansatz perturbation}
    Given a function $w$, $\alpha\in\R$, $v\in \bm{e_n}^\perp$, and 
     $Q$ a parabolically $2$-homogeneous polynomial satisfying
    $\Heat(x_nQ)\equiv0$, we define
    \[
        \sU[w,0,0,Q](x,t)\coloneqq w(x+Q(x,t)\bm{e_n},t)
    \]
    and
    \[
        \sU[w,\alpha,v,Q](x,t)
        \coloneqq\sU[w,0,0,Q]\left(\mathsf{R}_{\frac{\bm{e_n}+v}{|\bm{e_n}+v|}}x
        +\alpha\bm{e_n},t\right),
    \]
    where, for $e\in\bS^{n-1}$ with $e\cdot\bm e_n>0$, $\mathsf{R}_e\in SO(n)$ is the unique rotation
    satisfying $\mathsf{R}_e e=\bm e_n$ and acting as the identity on
    $\operatorname{span}\{e,\bm e_n\}^\perp$.
\end{definition}
The next theorem relates these nonlinear approximations to the cubic
expansion of $\varphi$, with estimates valid down to scales comparable
to the displacement of the base point.
\begin{theorem}\label{theorem:improvement of flatness}
    Let $u$ satisfy \eqref{eq:flatness assumption}, and let $u_k$ solve
    \eqref{eq:stefan} in $\Omega^{2\beta}_1(x_k,t_k)$, where
    $(x_k,t_k)\to(0,0)$. Suppose that, for
    some rotations $\mathsf{R}_k\to Id$, the functions
    \[
        u_k^\sharp(x,t)\coloneqq u_k(\mathsf{R}_kx+x_k,t+t_k)
    \]
    satisfy \eqref{eq:flatness assumption} with the same $r_0$. Suppose in
    addition that there are $\eps_k\to0^+$
    such that
    \begin{equation}\label{eq:hp thm flatness: lin}
        \left\|\frac{u_k-u}{\eps_k} - \varphi\right\|_
        {L^2(\Omega^\beta_1\cap\Omega^\beta_1(x_k,t_k)\cap\sC_1)}
        \to0
    \end{equation}
    where $\varphi$ solves \eqref{eq:linearized problem transversality} and satisfies \eqref{eq:expansion perturbation}--\eqref{eq:apply expansion linear} for suitable $\alpha,v,Q$.

    Then there is $C_\circ>0$ such that, for every $\delta>0$, there are
    $\rho_\delta,k_\delta>0$ and Ansätze
    $\sU_k\coloneqq\sU[u,\eps_k\alpha_k,\eps_kv_k,\eps_kQ_k]$,
    $k>k_\delta$, satisfying
    \begin{equation}\label{eq:thm flatness closeness conclusion}
        |\alpha_k-\alpha|+|v_k-v|+\|Q_k-Q\|_{L^2(\sC_1)} \leq C_\circ\delta
    \end{equation}
    and, for all $r\in\left(C_\circ|x_k,t_k|_\pb,\rho_\delta\right)$,     \begin{equation}\label{eq:thm flatness expansion conclusion}
        \|(u_k - \sU_k)(r\cdot,r^2\cdot)\|_{L^2(B_{3/2}\times(-3/2,-1/4))} \le C_\circ\delta|\eps_k| r^{3.5},
    \end{equation}
and
\begin{equation}\label{eq:thm flatness expansion conclusion2}
        \|(u_k-\sU_k)(x_k+r\mathsf{R}_k\cdot,t_k+r^2\cdot)\|_{L^2(B_1\times(-1,-1/2))}
        \le C_\circ\delta|\eps_k|r^{3.5}.
    \end{equation}
\end{theorem}

To prove this theorem, we shall need several preliminary results presented in the next subsection.

\subsection{Properties of the Ansätze}
In the proof, we shall apply \Cref{definition:ansatz perturbation}
with $w=\bar u$, where $\bar u$ is a suitable extension of $u$.
Since $u$ is initially defined only in a cusp, we first extend it
to a full parabolic neighbourhood of $(0,0)$, so that the
translations in the Ansatz are well defined.
\begin{lemma}\label{lem:extend smooth solution}
Let $u$ solve \eqref{eq:stefan} in
    $\Omega^{2\beta}_2\cap\sC_1$ and satisfy
    \eqref{eq:flatness assumption}.
    For each integer $J\ge3$ there are $r_J>0$ and a function $\bar u\colon\sC_{r_J}\to[0,+\infty)$ such that
    \begin{equation}\label{eq:pde extended solution}
        \bar u\ge0,\qquad \partial_t\bar u\ge c\partial_n\bar u\ge0,\qquad
        \Heat\bar u
        =\bigl(1+O(r^{J-2})\chi_{(\Omega_1^\beta)^c}\bigr)
        \chi_{\{\bar u>0\}}\quad\text{in }\sC_r,
    \end{equation}
    and
    \begin{equation}\label{eq:pde derivative extended solution}
        \partial_n\bar u=x_n+O(r^2),\qquad
        \nabla\partial_n\bar u=\bm{e_n}+O(r)
        \quad\text{in }\{\bar u>0\}\cap\sC_r,
    \end{equation}
    for every $r\in(0,r_J)$. Moreover,
    \[
        \bar u\equiv u\quad\text{in }\Omega_1^\beta\cap\sC_{r_J}.
    \]
\end{lemma}
\begin{proof}
    By the proof of \cite[Theorem~13.5]{FRS},
    for every $k$ and fixed $N$ there is a polynomial Ansatz
    $\sP=\tfrac12(\sA_\sP)_+^2$ of order $k$ such that
    \[
        \|(u-\sP)\circ S_\rho\|_{L^2(\Omega_N^{2\beta}\cap\sC_1)}
        \le C_k\rho^{k+\alpha}
    \]
    for some $\alpha>0$ and all sufficiently small $\rho>0$.
    To obtain a pointwise estimate in
    $\Omega_{3/2}^\beta\cap(\sC_{2r}\setminus\sC_{r/2})$, take
    $\rho=r^{1/4}$ and observe that backward cylinders of radius $cr$
    around its points lie in $S_\rho(\Omega_N^{2\beta}\cap\sC_1)$, since
    \[
        \frac{|x|^{2+2\beta}}{-t}\le Cr^\beta
        =o(\rho^{2\beta}).
    \]
    Since $\Heat\sP=(1+O(|(x,t)|_\pb^k))\chi_{\{\sP>0\}}$
    we have $\Heat|u-\sP|\ge-C_k|(x,t)|_\pb^k$, so the local
    $L^2$--$L^\infty$ estimate with bounded source yields
    \[
        |u-\sP|
        \le C_kr^{-(n+2)/2}\rho^{k+\alpha+(n+2)/2}
             +C_kr^{k+2}
        \le C_Lr^L
    \]
    for any prescribed $L$, upon choosing $k=k(L)$ sufficiently large.

    Next, by \Cref{lem:smoothness close maximal stratum} with $M=3$, applied
    after parabolic rescaling, we may decrease $r_J$ so that
    \[
        \{u>0\}=\{x_n>f(x',t)\}
        \quad\text{in }\Omega_{3/2}^\beta\cap\sC_{2r_J},
    \]
    where $f\in C^{6,3}$. Hence, boundary Schauder estimates and
    \[
        u(x',f+s,t)
        =s^2\int_0^1(1-\tau)\partial_{nn}u(x',f+\tau s,t)\,d\tau,
        \qquad
        \partial_{nn}u(x',f,t)=\frac1{1+|\nabla f|^2},
    \]
    give a signed extension $\sA_u$ of $\sqrt{2u}$ for which
    $r^{-1}\sA_u\circ S_r$ has uniformly bounded $C^{4,2}$ norms
    on these cylinders, with $\partial_n\sA_u\ge c>0$.
    Also, the estimate $|u-\sP|\le C_Lr^L$ and the quadratic nondegeneracy
    estimate in \eqref{eq:optimal regularity} give an
    $O(r^{L/2})$ bound for the distance between the two graphs and for
    $\sA_u-\sA_\sP$. Thus, after extending this difference across the graphs, we define
    $
        h_r\coloneqq r^{-1}(\sA_u-\sA_\sP)\circ S_r,
  $
    and on the rescaled cylinders introduced above, the preceding estimates give
    \[
        \|h_r\|_{L^\infty}\le Cr^{L/2-1},
        \qquad \|h_r\|_{C^{4,2}}\le C.
    \]
    By interpolation,    \[
        \|h_r\|_{C^{2,1}}
        \le C\|h_r\|_{L^\infty}^{1/2}\|h_r\|_{C^{4,2}}^{1/2}
        \le Cr^{L/4-1/2},
    \]
    and scaling back we obtain
    \[
        |\sA_u-\sA_\sP|
        +r|\nabla(\sA_u-\sA_\sP)|
        +r^2\bigl(|D^2(\sA_u-\sA_\sP)|
             +|\partial_t(\sA_u-\sA_\sP)|\bigr)
        \le Cr^{L/4+1/2}.
    \]
    Taking $L\ge4J-2$, we obtain
    \begin{equation}\label{eq:extension defining function derivatives}
        |\sA_u-\sA_\sP|+r|\nabla(\sA_u-\sA_\sP)|
        +r^2\bigl(|D^2(\sA_u-\sA_\sP)|
             +|\partial_t(\sA_u-\sA_\sP)|\bigr)
        \le C_Jr^J
    \end{equation}
    in $\Omega_{3/2}^\beta\cap(\sC_{2r}\setminus\sC_{r/2})$
    for every $r\in(0,r_J)$.
    In particular, in
    $(B_r\cap\{x_n>r/4\})\times(-r^2,-r^2/100)$ we have
    $\sA_\sP=x_n+O(r^2)$ and, by
    \eqref{eq:extension defining function derivatives},
    $\sA_u=\sA_\sP+O(r^J)$. Thus, both functions are bounded
    below by $r/8$ for $r$ sufficiently small and therefore
    \begin{equation}\label{eq:defining function difference of squares}
        \sA_u-\sA_\sP
        =\frac{2(u-\sP)}{\sA_u+\sA_\sP}=O(r^J) \qquad \text{in }(B_r\cap\{x_n>r/4\})\times(-r^2,-r^2/100).
    \end{equation}
    We now prove $\partial_t\sA_\sP(0,0)>0$ by adapting a barrier argument from \cite[Lemma~8.4]{FRS}.
    Since $f(0,0)=0$ and $\nabla_{x'}f(0,0)=0$, Taylor's formula gives $f(x',t)\le C(|x'|^2-t)$ for $t\le0$, therefore
    \[
        \{q\coloneqq x_n-C(|x'|^2-t)>0\}
        \cap\Omega_{3/2}^\beta\cap\sC_{r_J}\subset\{u>0\}
    \]
    for $C$ sufficiently large. Set $s\coloneqq|x'|^2+q^2$ and
    $R(t)\coloneqq c_0(-t)^{1/(2+\beta)}$. For $a_0>0$ to be chosen
    below, let $a(t)$ solve the ODE
    \[
        a'(t)=-\frac\kappa{R(t)^2}a(t)\quad\text{for }-T<t<0,
        \qquad a(-T)=a_0,
    \]
    and define
    \[
        v\coloneqq a(t)(e^{Aq}-1)\left(1-\frac{s}{R(t)^2}\right).
    \]
    Choosing $c_0$ small, $A,\kappa$ large, and $T>0$ small, we have
    \[
        \Heat v\ge
        \frac{a(t)(e^{Aq}-1)}{R(t)^2}
        \left[\kappa\left(1-\frac{s}{R(t)^2}\right)
             +\frac{2s}{(2+\beta)|t|}-C\right]\ge0
    \]
    in $\{q>0,\ s<R(t)^2,\ -T<t<0\}\subset\Omega_{3/2}^\beta$,
    where the last inequality follows by taking $\kappa$ sufficiently
    large when $s\le R(t)^2/2$ and $T$ sufficiently small when
    $s\ge R(t)^2/2$, since $R(t)^2/|t|\to\infty$ as $t\uparrow0$.
    Since $\partial_tu>0$ in $\{u>0\}$, we may choose $a_0>0$
    sufficiently small that $v\le\partial_tu$ on
    $\{q>0,\ s<R(-T)^2,\ t=-T\}$, and then, since $v=0$ on the
    lateral boundary, the comparison principle gives $\partial_tu\ge v$.
    Moreover, $\beta>0$ implies $\int_{-T}^0R(t)^{-2}\,dt<\infty$, so
    $a(t)$ has a positive limit $a(0)$ as $t\uparrow0$, and therefore,
    for all sufficiently small $r>0$, we obtain
    \[
        \partial_tu(r\bm e_n,-r^2)\ge cr.
    \]
    Note that $\partial_tu=\sA_u\partial_t\sA_u$ in the positive phase and $\sA_u(r\bm e_n,-r^2)=r+O(r^2)$, hence
    $\partial_t\sA_u(r\bm e_n,-r^2)\ge c$, while
    \eqref{eq:extension defining function derivatives} gives
    $\partial_t(\sA_u-\sA_\sP)=O(r^{J-2})\to0$ as $r\downarrow0$
    because $J\ge3$. Letting $r\downarrow0$ therefore yields
    $\partial_t\sA_\sP(0,0)>0$, and by continuity and the same derivative
    estimate we obtain, after decreasing $r_J$,
    \[
            \partial_t\sA_u\ge c\partial_n\sA_u>0\quad
                \text{in }\Omega_{3/2}^\beta\cap\sC_{2r_J},\qquad            \partial_t\sA_\sP\ge c\partial_n\sA_\sP>0\quad
                \text{in }\sC_{2r_J}.
    \]
    Thus, the implicit function theorem together with
    \eqref{eq:extension defining function derivatives} gives a
    function $\Phi\colon\Omega_{3/2}^\beta\cap\sC_{2r_J}\to\R$ such that
    \[
        \sA_\sP(x+\Phi(x,t)\bm e_n,t)=\sA_u(x,t),\qquad
        |\Phi|+r|\nabla\Phi|
        +r^2\bigl(|D^2\Phi|+|\partial_t\Phi|\bigr)\le C_Jr^J,
    \]
    where the estimate holds in
    $\Omega_{3/2}^\beta\cap(\sC_{2r}\setminus\sC_{r/2})$
    for every $r\in(0,r_J)$. Moreover,
    \begin{equation}\label{eq:monotonicity polynomial ansatz}
        \partial_t\sP\ge c\partial_n\sP\ge0
        \qquad\text{in }\sC_{2r_J}.
    \end{equation}
   Now, let $\eta(x,t)\coloneqq\xi\left(\frac{|x|}{(-t)^{1/(2+\beta)}}\right)$ where $\xi$ is smooth, equals one on $(0,1)$, and vanishes on $(3/2,+\infty)$, and set $\Psi\coloneqq\eta\Phi$ in $\Omega_{3/2}^\beta\cap\sC_{2r_J}$, extended by zero to $\sC_{2r_J}$.
    Then $\Psi=\Phi$ in $\Omega_1^\beta\cap\sC_{2r_J}$, and since
    $|x|\asymp r$ and $-t\asymp r^{2+\beta}$ on
    $\left(\Omega_{3/2}^\beta\setminus\Omega_1^\beta\right)
    \cap\left(\sC_{2r}\setminus\sC_{r/2}\right)$, we have
    \[
        |\nabla\eta|\le Cr^{-1},\qquad
        |D^2\eta|\le Cr^{-2},\qquad
        |\partial_t\eta|\le Cr^{-2-\beta},
    \]
    and therefore the bounds on $\Phi$ give, for every $r\in(0,r_J)$,
      \[
        |\Psi|+r|\nabla\Psi|+r^2|D^2\Psi|
        +r^{2+\beta}|\partial_t\Psi|\le C_Jr^J\qquad \text{in } \left(\Omega_{3/2}^\beta\setminus\Omega_1^\beta\right)
       \cap\left(\sC_{2r}\setminus\sC_{r/2}\right).
    \]
   Now, set
    \[
        \bar u\coloneqq\sP(x+\Psi(x,t)\bm{e_n},t)
        \qquad\text{in }\sC_{r_J}.
    \]
    Then $\bar u=u$ in $\Omega_1^\beta\cap\sC_{r_J}$ and $\bar u=\sP$ in
    $\sC_{r_J}\setminus\Omega_{3/2}^\beta$. Also, for every $r\in(0,r_J)$,
    inside $\{\bar u>0\}\cap\sC_r$ and with the derivatives of
    $\sP$ evaluated at $(x+\Psi\bm{e_n},t)$, we have
    \[
        \partial_t\bar u=\partial_t\sP+\partial_n\sP\,\partial_t\Psi,
        \qquad
        \partial_n\bar u=\partial_n\sP(1+\partial_n\Psi),
    \]
    and
    \[
        \Heat\bar u=\Heat\sP+\partial_{nn}\sP|\nabla\Psi|^2
        +2\nabla\partial_n\sP\cdot\nabla\Psi+\partial_n\sP\Heat\Psi.
    \]
    Since
    \[
        \partial_n\sP=x_n+O(r^2),\qquad
        \nabla\partial_n\sP=\bm e_n+O(r),
    \]
    the bounds on $\Psi$ give \eqref{eq:pde derivative extended solution}.
    Also, by \eqref{eq:monotonicity polynomial ansatz},
    \[
        \partial_t\bar u
        \ge\bigl(c-C_Jr^{J-2-\beta}\bigr)\partial_n\sP
        \ge c\partial_n\bar u\ge0.
    \]
    Finally, in $\{\bar u>0\}\cap(\sC_r\setminus\Omega_1^\beta)$ we have $\Heat\sP=1+O(r^{J-1})$.
    Hence, using $\partial_n\sP=O(r)$ and $\beta<1$,
    \[
        |\Heat\bar u-1|
        \le C_J\bigl(r^{J-1}+r^{2J-2}+r^{J-1-\beta}\bigr)
        \le C_Jr^{J-2}\qquad \text{in }\{\bar u>0\}\cap(\sC_r\setminus\Omega_1^\beta).
    \]
     Since $\bar u$ and its first
    derivatives vanish on its free boundary, this proves \eqref{eq:pde extended solution}.
    \end{proof}

\begin{definition}\label{definition:extended ansatz}
    Let $\bar u$ be given by \Cref{lem:extend smooth solution} with $J=5$, and recall \Cref{definition:ansatz perturbation}.
    We call $\sU[\bar u,\alpha,v,Q]$ an extended $\delta$-Ansatz if
    \[
        \max\{|\alpha|,|v|,\|Q\|_{L^2(\sC_1)}\} \le\delta.
    \]
\end{definition}

The next two lemmas show that the extended Ans\"atze satisfy the Stefan
equation up to a controlled error and that their first variation is the
linearized Ansatz in \eqref{eq:ansatz linearized problem}.

\begin{lemma}\label{lem:properties extended ansatz: pde}
    Let $\bar u,r_5$ be given by \Cref{lem:extend smooth solution} with $J=5$ and let
    $\sU\coloneqq\sU[\bar u,\alpha,v,Q]$ be an extended $|\eps|$-Ansatz for some $\eps\in \R$. If
    $|\eps|$ is sufficiently small, then
    \[
        \sU\ge0,\qquad \partial_t\sU\ge0
        \quad\text{in }\sC_{r_5/2},
    \]
    and
    \begin{equation}\label{eq:pde ansatz}
        |\Heat\sU-\chi_{\{\sU>0\}}|
        \le C\bigl(|\eps|\,|(x,t)|_\pb^2+|\alpha|^2(|\alpha|+|\eps|)\bigr)
        \chi_{\{\sU>0\}}
        \quad\text{in }\Omega^\beta_{1/2}\cap\sC_{r_5/2}.
    \end{equation}
\end{lemma}
\begin{proof}
    Nonnegativity of $\sU$ follows from the definition. Also, since spatial rotations preserve the heat operator, we may assume that $v=0$. Set
    \[
        y\coloneqq x+\alpha\bm{e_n},\qquad Z\coloneqq y+Q(y,t)\bm{e_n}.
    \]
    Inside $\{\sU>0\}$, with all the derivatives of $\bar u$ evaluated at $(Z,t)$,
    \[
        \partial_t\sU=\partial_t\bar u+\partial_n\bar u\,\partial_tQ,
        \qquad
        \Heat\sU=\Heat\bar u+\partial_{nn}\bar u|\nabla Q|^2
        +\partial_n\bar u\Heat Q+2\nabla\partial_n\bar u\cdot\nabla Q.
    \]
    Since $Q$ is parabolically $2$-homogeneous,
    \[
        |\partial_tQ|\le C\|Q\|_{L^2(\sC_1)}\le C|\eps|,
    \]
   and therefore, by \eqref{eq:pde extended solution}, for $|\eps|$ sufficiently small we have
    \[
        \partial_t\sU\ge(c-C|\eps|)\partial_n\bar u\ge0.
    \]
    Also, by homogeneity, \eqref{eq:pde derivative extended solution}, and
    \[
        y_n\Heat Q+2\partial_nQ=\Heat(y_nQ)=0,
    \]
    the last three terms of $\Heat\sU$ are bounded by
    $
        C|\eps|\bigl(|(x,t)|_\pb^2+|\alpha|^2\bigr).
    $
    
    To estimate the error coming from \eqref{eq:pde extended solution}, set
    $s\coloneqq(-t)^{1/(2+\beta)}$. If $(x,t)\in\Omega^\beta_{1/2}$ and
    $(Z,t)\notin\Omega^\beta_1$, then, by homogeneity,
    \[
        s\le |Z|
        \le \frac{s}{2}+|\alpha|+|Q(y,t)|
        \le \frac{s}{2}+|\alpha|+C|\eps|(s+|\alpha|)^2.
    \]
    For $|\eps|$ and $r_5$ sufficiently small, this gives
    $s\le C|\alpha|$, and hence $|(Z,t)|_\pb\le C|\alpha|$.
    Since $J=5$, the error in \eqref{eq:pde extended solution} is therefore
    bounded by
    \[
        C|\alpha|^3
        \le C|\alpha|^2(|\alpha|+|\eps|).
    \]
    Since $\bar u$ and its first derivatives vanish on the free boundary, this proves 
\eqref{eq:pde ansatz}.
\end{proof}

\begin{lemma}\label{lem:linearized vs perturbed ansatz}
    Let $\bar u, r_5$ be given by \Cref{lem:extend smooth solution} with $J=5$.
    Given $\alpha\in\R$,
    $v\in \bm{e_n}^\perp$, and a parabolically 2-homogeneous polynomial
    $Q$ satisfying $\Heat(x_nQ)\equiv0$, for
    $\eps\neq 0$ set
    \[
        \sU_\eps \coloneqq \sU[\bar u,\eps\alpha,\eps v,\eps Q].
    \]
    Then
    \[
        \lim_{\eps\to0}\tfrac1\eps(\sU_\eps-u)
        = (\alpha+x\cdot v+Q)\partial_nu-x_n(v\cdot\nabla u)
        \qquad\text{in }L^2_{\rm loc}(\Omega_1^\beta\cap\sC_{r_5}).
    \]
\end{lemma}
\begin{proof}
Noticing that
    \[
        \left.\frac d{d\eps}\mathsf{R}_{\frac{\bm{e_n}+\eps v}{|\bm{e_n}+\eps v|}} x\right|_{\eps=0}=(x\cdot v)\bm{e_n}-x_nv,
    \]
we get
    \[
        \lim_{\eps\to0}\tfrac1\eps(\sU_\eps-\bar u)
        =(\alpha+x\cdot v+Q)\partial_n\bar u
        -x_n(v\cdot\nabla\bar u)
        \qquad\text{in }L^\infty_{\rm loc}(\sC_{r_5}).
    \]
   The conclusion follows since $\bar u=u$ in $\Omega_1^\beta\cap\sC_{r_5}$.
\end{proof}
To compare Ans\"atze at different scales, we first compute their variation
with respect to the parameters and then estimate the parameter differences
in terms of the Ans\"atze.
\begin{lemma}\label{lem:tangent plane}
    Let $\bar u$ be given by \Cref{lem:extend smooth solution} with $J=5$.
    There exists $\eps_0\in(0,1)$ such that, for $\max\{|\alpha|,|v|,\|Q\|_{L^2(\sC_1)}\}<\eps_0$, the parameter map
    \[
       (\alpha,v,Q)\mapsto \sU_{\bar u}(\alpha,v,Q)\coloneqq\sU[\bar u,\alpha,v,Q] \in L^2(\sC_{r_5/2})
    \]
    is well defined, of class $C^1$, and   
    \[
        D\sU_{\bar u}(0,0,0)[\tilde\alpha,\tilde v,\tilde Q]
        =(\tilde\alpha+x\cdot\tilde v+\tilde Q)\partial_n\bar u
        -x_n(\tilde v\cdot\nabla\bar u).
    \]
    Moreover, let $r>0$ and suppose that
    $\sU_{\bar u}(r\alpha,v,r^{-1}Q)$ is an extended $\eps_0$-Ansatz.
    Given a parameter triple
    $(\tilde\alpha,\tilde v,\tilde Q)$ satisfying
    \begin{equation} \label{eq:norm tilde 1}
        |\tilde\alpha|+|\tilde v|+\|\tilde Q\|_{L^2(\sC_1)}\le1,
    \end{equation}
    set, for $|\tau|$ small,
    \[
        \widetilde{\sU}_\tau(x,t)\coloneqq r^{-2}\sU_{\bar u}\bigl(
        r(\alpha+\tau\tilde\alpha),
        v+\tau\tilde v,
        r^{-1}(Q+\tau\tilde Q)\bigr)(rx,r^2t).
    \]
  Then, as
$    r+|\alpha|+|v|+\|Q\|_{L^2(\sC_1)}\to 0$
we have
\[
    \left.\frac d{d\tau}\widetilde{\sU}_\tau\right|_{\tau=0}
    \longrightarrow
    (x_n)_+(\tilde\alpha+x\cdot\tilde v+\tilde Q)
    \qquad\text{in }L^2_{\rm loc}(\R^n\times(-\infty,0]),
\]
uniformly with respect to
$(\tilde\alpha,\tilde v,\tilde Q)$.
\end{lemma}
\begin{proof}
Since $\bar u\in C^1$, the parameter map is $C^1$ and its
    derivative at zero is given by
    \Cref{lem:linearized vs perturbed ansatz}.
    Differentiating the coordinate map in
    \Cref{definition:ansatz perturbation} and using the homogeneity of $Q$
    and \eqref{eq:pde derivative extended solution}, we obtain
    \[
        \left.\frac d{d\tau}\widetilde{\sU}_\tau\right|_{\tau=0}
        =\bigl((x_n)_+\bm{e_n}+o(1)\bigr)\cdot
        \bigl((\tilde\alpha+x\cdot\tilde v+\tilde Q)\bm{e_n}
        -x_n\tilde v+o(1)\bigr)
    \]
    in $L^2_{\rm loc}(\R^n\times(-\infty,0])$, where both $o(1)$-terms
    converge to zero as
    $r+|\alpha|+|v|+\|Q\|_{L^2(\sC_1)}\to0$, uniformly with respect to
    $(\tilde\alpha,\tilde v,\tilde Q)$.
\end{proof}

\begin{lemma}\label{lem:ansatz vs parameters}
    Let $K\subset\R^n\times(-\infty,0]$ be compact and suppose that
    $K\cap\{x_n>0\}$ has nonempty interior. There are
    $C,\eta,r_*>0$, depending on $K$, such that the following holds. Let
    \[
        \sU^i\coloneqq\sU[\bar u,\alpha^i,v^i,Q^i],\qquad i=1,2,
    \]
    be extended $\eta$-Ansätze, let $r\in(0,r_*)$, and suppose that
    \[
        \frac{|\alpha^i|}{r}+|v^i|
        +r\|Q^i\|_{L^2(\sC_1)}\le\eta,
        \qquad i=1,2.
    \]
    Set
    \[
        \mathcal N_r\coloneqq|\alpha^1-\alpha^2|+r|v^1-v^2|
        +r^2\|Q^1-Q^2\|_{L^2(\sC_1)}.
    \]
    Then
    \[
        C^{-1}r\mathcal N_r\le
        \|(\sU^1-\sU^2)(r\cdot,r^2\cdot)\|_{L^2(K)}
        \le Cr\mathcal N_r.
    \]
\end{lemma}

\begin{proof}
Apply the fundamental theorem of calculus along the segment joining the
rescaled parameter triples $(r^{-1}\alpha^i,v^i,rQ^i)$, $i=1,2$.
By \Cref{lem:tangent plane}, the estimates then follow from
finite-dimensional norm equivalence, as in the proof of
\Cref{lem:ansatz linear vs parameters}.
The condition on \(K\) is needed only for the lower bound, since it guarantees that the map
\[
(\alpha,v,Q)\longmapsto (x_n)_+(\alpha+x\cdot v+Q)\big|_K
\]
is injective.
\end{proof}

\subsection{Proof of \texorpdfstring{\Cref{theorem:improvement of flatness}}{}}
We will use two preliminary results, adapted from \cite[Lemmas 13.6 and 13.7]{FRS}.
We use the reverse heat kernel \(G\) defined in \eqref{eq:gaussian kernels}.

The first lemma is a key compactness result, while the second is a Caccioppoli-type estimate.
\begin{lemma}\label{lem:improvement key lemma}
    Given $\bar C,L>1$ and an integer $\Lambda_0\ge8$, there is an
    integer $\Lambda\ge\Lambda_0$ such that the following holds. Let
    $R\ge\Lambda$, $\sigma\in[0,\Lambda^{-1})$, and suppose that
    \[
        w\in H^1_{\rm loc}\left(
        \Omega^\beta_{R\vartheta^{-\Lambda}}
        \cap\sC_{2^\Lambda}\cap\{t<-\sigma^2\}\right).
    \]
    Set $w_r(x,t)\coloneqq w(rx,r^2t)$ for $r>0$, and assume that
    \begin{equation}\label{eq:assumption key lemma 1}
        \left(\int_{\Omega^\beta_{R}\cap\big(B_1\times(-1,-(2^{-m}\sigma)^2)\big)}
        |w_{2^m}|^2+|\nabla w_{2^m}|^2+|\partial_t w_{2^m}|^2\,dx\,dt\right)^{1/2}
        \le \bar C(2^m)^{3.5}
        \qquad\forall\,m\in\{0,\ldots,\Lambda\},
    \end{equation}
and
    \begin{equation}\label{eq:assumption key lemma 2}\begin{cases}
        |\Heat w|\le \frac1{\Lambda}    &\text{in }\Omega^\beta_{R}\cap(B_{\Lambda}\times(-\Lambda,-\frac1{\Lambda}))\cap\{x_n\ge \frac1{\Lambda}\},\\
        w=0                         &\text{in }\Omega^\beta_{R}\cap(B_{\Lambda}\times(-\Lambda,-\frac1{\Lambda}))\cap\{x_n\le -\frac1{\Lambda}\}.
    \end{cases}\end{equation}
    Suppose moreover that
    \begin{equation}\label{eq:assumption key lemma 3}
        \left|\int_{B_{\Lambda}\times(-1,-1/2)}
        w(x_n)_+(\tilde\alpha+x\cdot\tilde v+\tilde Q)G(x,t)\,dx\,dt\right|<\frac1{\Lambda}
    \end{equation}
    for every $\tilde\alpha\in\R$, $\tilde v\perp e_n$, and every parabolically 2-homogeneous polynomial $\tilde Q$ satisfying
 $
        \Heat(x_n\tilde Q)=0,$
with $        |\tilde\alpha|+|\tilde v|+\|\tilde Q\|_{L^2(\sC_1)}\le1.
$
    Then
    \[
        \left(\int_{\Omega^\beta_R\cap(B_1\times(-1,-(2\sigma)^2))}w^2(2^{-1}x,4^{-1}t)dx\,dt\right)^{1/2}\le\frac1L.
    \]
\end{lemma}
\begin{proof}
    We follow the proof of \Cref{lem:improvement linear key lemma}, based on
    \cite[Lemma 13.6]{FRS}. To account for the terminal time
    $-\sigma^2$, we also need to control the $L^2$ mass near $t=0$.
    The estimate
    \eqref{eq:assumption key lemma 1} with $m=0$, together with the monotonicity
    of the time sections of $\Omega_R^\beta$, gives
    \[
        \sup_{t\in(-1/8,-\sigma^2)}
        \int_{\{x\in B_{1/2}:(x,t)\in\Omega_R^\beta\}}
        |w(x,t)|^2\,dx\le C.
    \]
    Hence the contribution of $(-\eta,-\sigma^2)$ to the squared $L^2$-norm
    is bounded by $C\eta$, uniformly in $\sigma$, whenever
    $\sigma^2<\eta<1/8$.
    This mass therefore tends to zero as $\eta\downarrow0$, uniformly in
    $\sigma$, which allows the $L^2$ convergence in the cited proof to
    extend to our current situation.
\end{proof}

\begin{lemma}\label{lem:improvement caccioppoli}
    Fix $\beta\in(0,1)$ and let $u_i\colon\Omega^\beta_R\cap\sC_2\cap\{t<-s^2\}\to[0,+\infty)$, $i=1,2$, be solutions of
    \[
    \begin{cases}
        \Heat u_i = (1+e_i)\chi_{\{u_i>0\}},\\
        \partial_tu_i\ge0
    \end{cases}
    \]
    where the functions $e_i$ satisfy $|e_i|\le\bar\eps<1/100$, and assume
    that $0\le s<1/100$. Set $w\coloneqq u_1-u_2$. Then there are constants
    $C,R_\circ$ depending only on $n,\beta$ such that, for $R\ge R_\circ$,
    \[
\int_{\Omega^\beta_{\vartheta R}\cap(B_1\times(-1,-s^2))}
        \bigl(|\nabla w|^2+|\partial_tw|^2\bigr)\,dx\,dt \le C\left(\int_{\Omega^\beta_R\cap(B_2\times(-4,-s^2))}w^2\,dx\,dt+\bar\eps^2\right).
    \]
\end{lemma}
\begin{proof}
    This follows from \cite[Lemma 13.7]{FRS}, applied with
    $\theta=\sqrt\vartheta$ (note that the same proof applies with terminal time
    $-s^2$, since the boundary terms are nonnegative).
\end{proof}

We now combine these two results with Gaussian projection to obtain the
inductive step for \Cref{theorem:improvement of flatness}. The relation between $\tilde u$ and the background in the lemma is expressed by
\eqref{eq:prop improvement of flatness base hp1}: $\tilde u$ is approximated by
extended Ans\"atze built from $\bar u$ at $M_1+1$ consecutive dyadic
scales, and the lemma constructs an approximation at the next scale while controlling the change in the parameters.

In the application
below, $u$ is the fixed background solution and $\tilde u=u_k$ is one of the
perturbed solutions, with $\frac{u_k-u}{\eps_k}\to\varphi$ (recall \eqref{eq:hp thm flatness: lin}).

\begin{lemma}\label{lem:inductive step improvement of flatness}
There are an integer $M_1\ge9$ and a constant $N>0$, depending only on
$n$ and $\beta$, with the following property. Let $u$ satisfy the hypotheses
of \Cref{lem:extend smooth solution}, let $\bar u$ be its
extension with $J=5$, and set $R_*\coloneqq2M_1$.
Then there are constants $\eta_1,C_1>0$, which may also depend on $u$,
such that, for every $\delta\in(0,\eta_1)$, there exists $r_\delta>0$
for which the following holds.

Suppose that $\tilde u$ solves \eqref{eq:stefan} in $\Omega^\beta_N\cap\big(B_{r/\eta_1}\times(-r^2/\eta_1,-s^2)\big)$, where $0<r<r_\delta$ and $0\le s<\eta_1r_{M_1}$, and set $r_m\coloneqq2^{-m}r$. Assume moreover that there is a family of extended $|\eps|$-Ansätze
\[
    \sU^{(m)}\coloneqq\sU[\bar u,\eps\alpha^{(m)},\eps v^{(m)},\eps Q^{(m)}],
    \qquad 0\le m\le M_1,
\]
with $0<|\eps|<\eta_1$, such that
\begin{equation}\label{eq:prop improvement of flatness base hp1}
\|(\tilde u-\sU^{(m)})(r_m\cdot,r_m^2\cdot)\|_{L^2(\Omega^\beta_{R_*}\cap(B_1\times(-1,-(s/r_m)^2)))}\le \delta|\eps| r_m^{3.5}.
\end{equation}
Suppose also that
\begin{equation}\label{eq:prop improvement of flatness base hp2}
    \begin{cases}
        \tilde u(r_{M_1}\cdot,r_{M_1}^2\cdot)=0    &\text{in }\Big(B_{1/\eta_1}\times(-1/\eta_1,-\eta_1)\Big) \cap\{x_n\le-\eta_1\},\\
        \tilde u(r_{M_1}\cdot,r_{M_1}^2\cdot)>0    &\text{in }\Big(B_{1/\eta_1}\times(-1/\eta_1,-\eta_1)\Big) \cap\{x_n\ge\eta_1\},
    \end{cases}
\end{equation}
and for $0\le m\le M_1$
\begin{equation}\label{eq:improvement of flatness flatness assumption}
    \|r_m^{-2}\tilde u(r_m\cdot,r_m^2\cdot)-\tfrac12(x_n)_+^2\|_{L^2(B_1\times(-1,-1/2))}\le\eta_1.
\end{equation}
Then there is an extended $2|\eps|$-Ansatz
\[
    \sU^{(M_1+1)}\coloneqq\sU[\bar u,\eps\alpha^{(M_1+1)},
        \eps v^{(M_1+1)},\eps Q^{(M_1+1)}]
\]
such that
\[
    \|(\tilde u-\sU^{(M_1+1)})(r_{M_1+1}\cdot,r_{M_1+1}^2\cdot)\|_{L^2(\Omega^\beta_{R_*}\cap(B_1\times(-1,-(s/r_{M_1+1})^2)))}\le \delta|\eps| r_{M_1+1}^{3.5}
\]
and
\begin{equation}\label{eq:proof improvement closeness new ansatz}
|\alpha^{(M_1+1)}-\alpha^{(M_1)}|+r_{M_1}|v^{(M_1+1)}-v^{(M_1)}|
+r_{M_1}^{2}\|Q^{(M_1+1)}-Q^{(M_1)}\|_{L^2(\sC_1)}
    \le C_1\delta(r_{M_1})^{2.5}.
\end{equation}
\end{lemma}

\begin{proof}[Proof of \Cref{lem:inductive step improvement of flatness}]
    We carry out Steps~1--3 for an arbitrary $M_1\ge9$, and choose
    $M_1$ in Step~4.
        
    \smallskip\noindent\textbf{Step 1.}
    Set
    \[
        K\coloneqq\overline{B_{1/4}\times(-3/4,-1/2)
        \cap\{x_n>1/100\}},
    \]
    let $\eta,r_*>0$ be small enough for
    \Cref{lem:ansatz vs parameters} to apply both on $K$ and on
    $\overline{B_1\times(-1,0)}$, and consider constants 
    $\eta_1<\eta/2$ and $r_\delta<r_*$ to be fixed.
    For $m\in\{0,\ldots,M_1-1\}$, rescaling
    \eqref{eq:prop improvement of flatness base hp1} to $r_{m+1}=r_m/2$
    and using \eqref{eq:improvement of flatness flatness assumption} gives
    \[
        \left\|r_{m+1}^{-2}\sU^{(i)}
        (r_{m+1}\cdot,r_{m+1}^2\cdot)-\tfrac12(x_n)_+^2\right\|_
        {L^2(B_1\times(-1,-1/2))}
        \le\eta_1+C\delta|\eps|r_m^{1.5}
        \qquad\text{for }\,i\in\{m,m+1\}.
    \]
    Since the rescaled parameters are
    $(\eps\alpha^{(i)}/r_{m+1},\eps v^{(i)},r_{m+1}\eps Q^{(i)})$,
    a compactness argument shows that, after decreasing $\eta_1$ and $r_\delta$, 
    \[
        \frac{|\eps\alpha^{(i)}|}{r_{m+1}}+|\eps v^{(i)}|
        +r_{m+1}\|\eps Q^{(i)}\|_{L^2(\sC_1)}\le\eta,
        \qquad i\in\{m,m+1\}.
    \]
    The same rescaling of \eqref{eq:prop improvement of flatness base hp1}
    gives
    \[
        \|(\sU^{(m+1)}-\sU^{(m)})
        (r_{m+1}\cdot,r_{m+1}^2\cdot)\|_{L^2(K)}
        \le C\delta|\eps|r_m^{3.5}.
    \]
    This allows us to apply  \Cref{lem:ansatz vs parameters},
with $r=r_{m+1}$ to get 
    \[        |\alpha^{(m+1)}-\alpha^{(m)}|
        +r_m|v^{(m+1)}-v^{(m)}|\\
+r_m^2\|Q^{(m+1)}-Q^{(m)}\|_{L^2(\sC_1)}
        \le C\delta r_m^{2.5}
        \qquad\forall\,m\in\{0,\ldots,M_1-1\},
\]
    and summing this bound over subsequent indices we obtain
        \begin{equation}\label{eq:summed closeness sequence ansatz}
|\alpha^{(\ell)}-\alpha^{(m)}|
        +r_m|v^{(\ell)}-v^{(m)}|+r_m^2\|Q^{(\ell)}-Q^{(m)}\|_{L^2(\sC_1)}
        \le C\delta r_m^{2.5}
        \qquad\forall\,0\le m\le\ell\le M_1.
    \end{equation}
    Thanks to this bound, after further decreasing $\eta_1$ and $r_\delta$ if needed, we deduce that
    \begin{equation}\label{eq:normalized parameters sequence}
        \frac{|\eps\alpha^{(\ell)}|}{r_m}+|\eps v^{(\ell)}|
        +r_m\|\eps Q^{(\ell)}\|_{L^2(\sC_1)}\le\eta
        \qquad\forall\,0\le m\le\ell\le M_1.
    \end{equation}
    
    \smallskip\noindent\textbf{Step 2.}
    Let $\eps_0$ be as in \Cref{lem:tangent plane}, and further require $2\eta_1<\eps_0$. We now choose the parameters $(\alpha^{(M_1+1)},v^{(M_1+1)},Q^{(M_1+1)})$ as the minimizers of
    \[
       (a,\zeta,S)\mapsto  \left(\int_{B_{M_1}\times(-1,-1/2)}
        \left|\tilde u(r_{M_1}x,r_{M_1}^2t)
        -\sU[\bar u,\eps a,\eps\zeta,\eps S]
        (r_{M_1}x,r_{M_1}^2t)\right|^2G(x,t)\,dx\,dt\right)^{1/2}
    \]
    among $a\in\R$, $\zeta\perp e_n$, and parabolically
    $2$-homogeneous polynomials $S$ satisfying $\Heat(x_nS)=0$, with the constraint 
$ \max\{|a|,|\zeta|,\|S\|_{L^2(\sC_1)}\}\le2,$ and we define 
    \[
        \sU^{(M_1+1)}\coloneqq
        \sU[\bar u,\eps\alpha^{(M_1+1)},
        \eps v^{(M_1+1)},\eps Q^{(M_1+1)}].
    \]
    Then, by \eqref{eq:summed closeness sequence ansatz},
    \eqref{eq:normalized parameters sequence}, the upper estimate in
    \Cref{lem:ansatz vs parameters}, and
    \eqref{eq:prop improvement of flatness base hp1}, we have
    \begin{equation}\label{eq:proof improvement step 2a}
        \|(\tilde u-\sU^{(M_1)})(r_m\cdot,r_m^2\cdot)\|_
        {L^2(\Omega^\beta_{R_*}\cap
        (B_1\times(-1,-(s/r_m)^2)))}
        \le C\delta|\eps|r_m^{3.5}
        \qquad\forall\,m\in\{0,\ldots,M_1\}.
    \end{equation}
    Summing \eqref{eq:proof improvement step 2a} over the dyadic annuli and
    using the Gaussian decay gives
    \begin{equation}\label{eq:proof improvement step 2}
\left(\int_{B_{M_1}\times(-1,-1/2)}
        |(\tilde u-\sU^{(M_1)})(r_{M_1}x,r_{M_1}^2t)|^2
        G(x,t)\,dx\,dt\right)^{1/2}\le C_0\delta|\eps|r_{M_1}^{3.5},
    \end{equation}
    where $C_0\ge1$ depends only on $n$ and $\beta$, so by minimality
    \begin{equation}\label{eq:minimizing ansatz gaussian}
        \left(\int_{B_{M_1}\times(-1,-1/2)}
        |(\tilde u-\sU^{(M_1+1)})(r_{M_1}x,r_{M_1}^2t)|^2
        G(x,t)\,dx\,dt\right)^{1/2}
        \le C_0\delta|\eps|r_{M_1}^{3.5}.
    \end{equation}
    Since the Gaussian kernel $G$ is bounded from below on $K$,
    \eqref{eq:minimizing ansatz gaussian} and
    \eqref{eq:improvement of flatness flatness assumption} give
    \[
        \left\|r_{M_1}^{-2}\sU^{(M_1+1)}
        (r_{M_1}\cdot,r_{M_1}^2\cdot)-\tfrac12(x_n)_+^2
        \right\|_{L^2(K)}
        \le C\bigl(\eta_1+\delta|\eps|r_{M_1}^{1.5}\bigr).
    \]
    As in Step 1, after decreasing $\eta_1$ and
    $r_\delta$ if needed, this implies that
    \begin{equation}\label{eq:minimizing ansatz centered}
        \frac{|\eps\alpha^{(M_1+1)}|}{r_{M_1}}
        +|\eps v^{(M_1+1)}|
        +r_{M_1}\|\eps Q^{(M_1+1)}\|_{L^2(\sC_1)}
        \le\eta.
    \end{equation}
    Thus, by the triangle inequality,
    \eqref{eq:proof improvement step 2},
    \eqref{eq:minimizing ansatz gaussian}, and 
    \Cref{lem:ansatz vs parameters}, we obtain 
    \eqref{eq:proof improvement closeness new ansatz} with a constant $C_1$
    independent of $M_1$ and $\delta$. Since $\sU^{(M_1)}$ is an extended
    $|\eps|$-Ansatz,
    \eqref{eq:proof improvement closeness new ansatz} also shows that the
    minimizer lies in the interior of the ball of radius $2$.

    Finally,  \Cref{lem:ansatz vs parameters} applied
    on $\overline{B_1\times(-1,0)}$, together with
    \eqref{eq:proof improvement step 2a} and
    \eqref{eq:proof improvement closeness new ansatz}, gives (after
    increasing $C_0$, if needed)
    \begin{equation}\label{eq:proof improvement growth estimate}
\|(\tilde u-\sU^{(M_1+1)})(r_m\cdot,r_m^2\cdot)\|_
        {L^2(\Omega^\beta_{R_*}\cap
        (B_1\times(-1,-(s/r_m)^2)))}\le C_0\delta|\eps|r_m^{3.5}
        \qquad\forall\,m\in\{0,\ldots,M_1\}.
    \end{equation}

    \smallskip\noindent\textbf{Step 3.}
    For $m\in\{1,\ldots,M_1\}$, set
    \[
        \tilde u_m\coloneqq r_m^{-2}\tilde u(r_m\cdot,r_m^2\cdot),
        \qquad
        U_m^*\coloneqq r_m^{-2}\sU^{(M_1+1)}(r_m\cdot,r_m^2\cdot),
    \]
   and decrease $r_\delta$ so that
    \[
        S_{r_m}\left(\Omega^\beta_{R_*/\vartheta}
        \cap(B_{M_1}\times(-M_1,0))\right)
        \subset\Omega^\beta_{1/2}\cap\sC_{r_5/2}
        \qquad\forall\,m\in\{1,\ldots,M_1\}.
    \]
    By \eqref{eq:minimizing ansatz centered} and the fact that
    $\sU^{(M_1+1)}$ is an extended $2|\eps|$-Ansatz,
    \[
        |\eps\alpha^{(M_1+1)}|\le Cr_{M_1},
        \qquad |\eps\alpha^{(M_1+1)}|\le2|\eps|,
    \]
    therefore, since $r_{M_1}\le r_m$, 
    \[
        |\eps\alpha^{(M_1+1)}|^2
        \bigl(|\eps\alpha^{(M_1+1)}|+|\eps|\bigr)
        \le C|\eps|r_m^2.
    \]
    Hence \eqref{eq:pde ansatz} gives, for a constant $C_{M_1}>0$,
    \[
        |\Heat U_m^*-\chi_{\{U_m^*>0\}}|
        \le C_{M_1}|\eps|r_m^2\chi_{\{U_m^*>0\}}
        \quad\text{in }\Omega^\beta_{R_*/\vartheta}
        \cap(B_{M_1}\times(-M_1,0)).
    \]
    Thus, decreasing further $r_\delta$ so that
    $C_{M_1}r_\delta^{1/4}\le\delta$, it follows that
    \begin{equation}\label{eq:normalized ansatz residual improvement}
        |\Heat U_m^*-\chi_{\{U_m^*>0\}}|
        \le\delta|\eps|r_m^{7/4}\chi_{\{U_m^*>0\}}
        \quad\text{in }\Omega^\beta_{R_*/\vartheta}
        \cap(B_{M_1}\times(-M_1,0))
        \qquad\forall\,m\in\{1,\ldots,M_1\}.
    \end{equation}
    Moreover,
    \[
        S_{1/2}(\Omega^\beta_{R_*/\vartheta})=\Omega^\beta_{R_*}.
    \]
    Hence, if $R_*/\vartheta\ge R_\circ$ and $\eta_1<1/100$, then
    \eqref{eq:proof improvement growth estimate} at the scales $r_m$ and
    $r_{m-1}=2r_m$, \eqref{eq:normalized ansatz residual improvement}, and
    \Cref{lem:properties extended ansatz: pde,lem:improvement caccioppoli}
    give, for a constant $\bar C>1$ depending only on $n$ and $\beta$,
    \begin{multline}\label{eq:normalized improvement energy}
        \left(\int_{\Omega^\beta_{R_*}\cap(B_1\times(-1,-(s/r_m)^2))}
        \bigl(|\tilde u_m-U_m^*|^2+|\nabla(\tilde u_m-U_m^*)|^2
        +|\partial_t(\tilde u_m-U_m^*)|^2\bigr)\,dx\,dt\right)^{1/2}\\
\le \bar C\delta|\eps|r_m^{1.5}
        \qquad\forall\,m\in\{1,\ldots,M_1\}.
    \end{multline}
    Also, by \eqref{eq:pde derivative extended solution}, the definition of
    $U_{M_1}^*$, and \eqref{eq:minimizing ansatz centered},
    \[
        \partial\{U_{M_1}^*>0\}\cap\Omega^\beta_{R_*}
        \cap(B_{M_1}\times(-M_1,-M_1^{-1}))
        \subset\{|x_n|\le C_{M_1}(r_{M_1}+\eta)\}.
    \]
    Since $M_1$ is fixed, we can first decrease $\eta$, and then $\eta_1$ and
    $r_\delta$, so that
    \[
        C_{M_1}(2^{-M_1}r_\delta+\eta)<M_1^{-1}
    \]
 and (recall 
    \eqref{eq:prop improvement of flatness base hp2})
    \begin{equation}\label{eq:minimizing ansatz phase separation}
    \begin{cases}
        \tilde u_{M_1}=U_{M_1}^*=0
        &\text{if }x_n\le-M_1^{-1}\\
        \tilde u_{M_1}>0\text{ and }U_{M_1}^*>0
        &\text{if }x_n\ge M_1^{-1}
    \end{cases}
    \qquad \text{inside }\Omega^\beta_{R_*}\cap
    (B_{M_1}\times(-M_1,-M_1^{-1})).
    \end{equation}
    
    \smallskip\noindent\textbf{Step 4.}
    Differentiating the coordinate map in
    \Cref{definition:ansatz perturbation} and using
    \eqref{eq:pde derivative extended solution}, as in the proof of
    \Cref{lem:tangent plane}, gives a bound $C(1+|x|)^3$ for the first
    variation of a normalized Ansatz on $B_{M_1}\times(-1,-1/2)$. Since
    $G(x,t)\le Ce^{-c|x|^2}$ on this time interval, the Gaussian
    $L^2$ norm of this variation on
    $(B_{M_1}\setminus B_\ell)\times(-1,-1/2)$ is bounded by
    $C(1+\ell)^de^{-c\ell^2}$, for fixed constants $C,c>0$ and $d\ge0$
    depending only on $n$ and $\beta$.
    
    Now, choose an integer $\Lambda_0\ge8$ so large that
    \[
        \frac{2(\Lambda_0+1)}{\vartheta}\ge R_\circ,
        \qquad
        C(1+\ell)^d e^{-c\ell^2}<\frac1{2\ell}
        \qquad\forall\,\ell\ge\Lambda_0,
    \]
    and apply \Cref{lem:improvement key lemma} with $\bar C$,
    $L=2^{3.5}C_0$, and $\Lambda_0$. With $\Lambda$ being the resulting
    integer, set
    \[
        M_1\coloneqq\Lambda+1,\qquad R_*\coloneqq2M_1,
        \qquad N\coloneqq R_*\vartheta^{-(M_1+1)}.
    \]
    After decreasing $\eta$ if needed, choose $\eta_1$ so that the restrictions in Steps~1--3 hold and
    $\eta_1<\min\{\Lambda^{-1},\eps_0/2,1/100\}$. Next, for every
    $\delta\in(0,\eta_1)$, choose $r_\delta$ so that the restrictions in
    Steps~1--3 hold and
$C_0^{-1}r_\delta^{1/4}\le\Lambda^{-1}.$

    Now, set $\sigma\coloneqq s/r_{M_1}$ and
    \[
        w(x,t)\coloneqq\frac{(\tilde u-\sU^{(M_1+1)})(r_{M_1}x,r_{M_1}^2t)}{C_0\delta|\eps|r_{M_1}^{3.5}}.
    \]
    We have $0\le\sigma<\Lambda^{-1}$. Also, since
    $R_*=2(\Lambda+1)\ge\Lambda$ and $2^\Lambda r_{M_1}=r/2$,
    \eqref{eq:cusp scaling exponent} and the choice of $N$ give
    \[
        w\in H^1_{\rm loc}\left(
        \Omega^\beta_{R_*\vartheta^{-\Lambda}}
        \cap\sC_{2^\Lambda}\cap\{t<-\sigma^2\}\right).
    \]
 Note that, if
    $j\in\{0,\ldots,\Lambda\}$ and $m=M_1-j$, then
    \[
        w_{2^j}
        =\frac{r_m^2}{C_0\delta|\eps|r_{M_1}^{3.5}}
        (\tilde u_m-U_m^*).
    \]
    Thus \eqref{eq:normalized improvement energy} gives
    \eqref{eq:assumption key lemma 1}. Also, since $M_1=\Lambda+1$,
    \eqref{eq:normalized ansatz residual improvement} and
    \eqref{eq:minimizing ansatz phase separation} imply
    \[
    \begin{cases}
        |\Heat w|\le C_0^{-1}r_{M_1}^{1/4}\le\Lambda^{-1}
        &\text{in }\Omega^\beta_{R_*}\cap
        (B_\Lambda\times(-\Lambda,-\Lambda^{-1}))
        \cap\{x_n\ge\Lambda^{-1}\},\\
        w=0
        &\text{in }\Omega^\beta_{R_*}\cap
        (B_\Lambda\times(-\Lambda,-\Lambda^{-1}))
        \cap\{x_n\le-\Lambda^{-1}\},
    \end{cases}
    \]
    thus \eqref{eq:assumption key lemma 2} holds.

    Let $(\tilde\alpha,\tilde v,\tilde Q)$ satisfy
    \eqref{eq:norm tilde 1}, with $\tilde v\perp e_n$ and
    $\Heat(x_n\tilde Q)=0$, and set
    \[
        \widetilde{\sU}_\tau(x,t)
        \coloneqq r_{M_1}^{-2}\sU_{\bar u}\bigl(
        \eps\alpha^{(M_1+1)}+\tau r_{M_1}\tilde\alpha,
        \eps v^{(M_1+1)}+\tau\tilde v,
        \eps Q^{(M_1+1)}+\tau r_{M_1}^{-1}\tilde Q\bigr)
        (r_{M_1}x,r_{M_1}^2t),
    \]
    and $\psi\coloneqq
        \left.\frac d{d\tau}\widetilde{\sU}_\tau\right|_{\tau=0}.$
    Since the minimizing triple $(\alpha^{(M_1+1)},v^{(M_1+1)},Q^{(M_1+1)})$ is an interior point of the set $ \max\{|a|,|\zeta|,\|S\|_{L^2(\sC_1)}\}\le2$ (see Step 2), minimality gives
    \[
        0=\int_{B_{M_1}\times(-1,-1/2)}w\psi G\,dx\,dt,
        \qquad
        \left(\int_{B_{M_1}\times(-1,-1/2)}w^2G\,dx\,dt\right)^{1/2}
        \le1.
    \]
    Also, by \eqref{eq:minimizing ansatz centered} and
    \Cref{lem:tangent plane}, 
        \[
        \left(\int_{B_\Lambda\times(-1,-1/2)}
        \left|\psi-(x_n)_+(\tilde\alpha+x\cdot\tilde v+\tilde Q)\right|^2
        G\,dx\,dt\right)^{1/2}<\frac1{2\Lambda}.
    \]
  Furthermore, the cubic bound at the beginning of Step~4 and the Gaussian decay yield
  \[
        \left(\int_{(B_{M_1}\setminus B_\Lambda)\times(-1,-1/2)}
        |\psi|^2G\,dx\,dt\right)^{1/2}
        \le C(1+\Lambda)^de^{-c\Lambda^2}<\frac1{2\Lambda},
  \]
    so \eqref{eq:assumption key lemma 3} follows. Thus, we can apply 
    \Cref{lem:improvement key lemma} with $R=R_*$, and recalling that
    $r_{M_1+1}=r_{M_1}/2$ we obtain
    \[
        \|(\tilde u-\sU^{(M_1+1)})(r_{M_1+1}\cdot,r_{M_1+1}^2\cdot)\|_
        {L^2(\Omega^\beta_{R_*}\cap
        (B_1\times(-1,-(s/r_{M_1+1})^2)))}
        \le\delta|\eps|r_{M_1+1}^{3.5}.\qedhere
    \]
\end{proof}

\begin{proof}[Proof of \texorpdfstring{\Cref{theorem:improvement of flatness}}{}]
Let $M_1$ and $N$ be the constants from
\Cref{lem:inductive step improvement of flatness}. Recalling
\eqref{eq:cusp scaling exponent} and \eqref{eq:parabolic scaling}, we see that
$S_\lambda^{-1}(\Omega^\beta_1)
    =\Omega^\beta_{\lambda^{-\gamma_\beta}},$
so applying this identity also with $2\beta$ in place of $\beta$,
after a fixed common parabolic dilation we may assume that $u$ and $u_k$
are defined in $\Omega^{2\beta}_{2N}\cap\sC_1$ and
$\Omega^{2\beta}_{2N}(x_k,t_k)$, respectively. Also, after this dilation,
\eqref{eq:flatness assumption} holds for $u$ and $u_k^\sharp$, and
\eqref{eq:hp thm flatness: lin} holds on their common domain. 

Let $\bar u$ be given by \Cref{lem:extend smooth solution} with $J=5$, and
let $\eta,r_*>0$ be given by \Cref{lem:ansatz vs parameters} for
$\overline{B_1\times(-1,0)}$. Choose $\eta_1,C_1$ as in
\Cref{lem:inductive step improvement of flatness}, with $\eta_1<\eta/2$,
and set
$R_*\coloneqq2M_1$. By homogeneity in $(\eps_k,\varphi)$, we may assume
\[
    |\alpha|+|v|+\|Q\|_{L^2(\sC_1)}\le\frac14\qquad \text{and}\qquad \delta<\eta_1/2.
\]
To prove the theorem we shall first construct the Ansätze at the initial scales, then apply
\Cref{lem:inductive step improvement of flatness} down to scales comparable
to $|(x_k,t_k)|_\pb$, and finally take the limit of their coefficients.
We divide the proof into three steps.

\smallskip\noindent\textbf{Step 1.}
Let $r_\delta$ be given by
\Cref{lem:inductive step improvement of flatness}, set $d_k\coloneqq|(x_k,t_k)|_\pb$ and $s_k\coloneqq2d_k$,
and choose $\rho_0$ small, to be fixed later, with $\rho_m\coloneqq2^{-m}\rho_0$. For
$m\in\{0,\ldots,M_1\}$, set
\[
    \sU_k^m\coloneqq\sU[\bar u,\eps_k\alpha,\eps_kv,\eps_kQ],
    \qquad
    (\alpha_k^m,v_k^m,Q_k^m)\coloneqq(\alpha,v,Q).
\]
By \eqref{eq:hp thm flatness: lin},
\eqref{eq:apply expansion linear} with $R=R_*$, and
\Cref{lem:linearized vs perturbed ansatz}, we can decrease
$\rho_0$ and then choose $k_\delta$ so that, for every $k>k_\delta$,
\[
    |\eps_k|<\eta_1,
    \qquad s_k<\eta_1\rho_{M_1},
\]
\begin{equation}\label{eq:iteration ansatz estimates}
    \|(u_k-\sU_k^m)(\rho_m\cdot,\rho_m^2\cdot)\|_{L^2(\Omega^\beta_{R_*}\cap(B_1\times(-1,-(s_k/\rho_m)^2)))}
    \le\delta|\eps_k|\rho_m^{3.5}
    \qquad\forall\,m\in\{0,\ldots,M_1\}.
\end{equation}

\smallskip\noindent\textbf{Step 2.}
After further increasing $k_\delta$, for every $k>k_\delta$ we have
\begin{equation}
\label{eq:inclusion}
    \Omega_N^\beta\cap\{-1<t<-s_k^2\}
    \subset\Omega_{2N}^\beta(x_k,t_k).
\end{equation}
By \eqref{eq:flatness assumption} applied to $u_k^\sharp$, the convergence
$\mathsf{R}_k\to Id$, and \Cref{theorem:smooth expansion}, after decreasing $\rho_0$ if
necessary, we have
\begin{equation}\label{eq:fixed frame flatness iteration}
    \|r^{-2}u_k(r\cdot,r^2\cdot)-\tfrac12(x_n)_+^2\|_
    {L^2(B_1\times(-1,-1/2))}\le\eta_1
    \qquad\text{if }0<r\le\rho_0\text{ and }s_k<2\eta_1r,
\end{equation}
while the two conditions in
\eqref{eq:prop improvement of flatness base hp2} hold with $\tilde u=u_k$ and
$r_{M_1}=r$ whenever $s_k<\eta_1r$.

We now proceed by induction on $j$, constructing $(M_1+1)$-tuples
$(\sU_k^j,\ldots,\sU_k^{j+M_1})$ of extended $|\eps_k|$-Ansätze
satisfying the estimate in \eqref{eq:iteration ansatz estimates}
at the corresponding scales. 
Step~1 provides the initial $(M_1+1)$-tuple for $j=0$, and for the inductive
step we suppose that the $(M_1+1)$-tuple indexed by $j$ has been constructed
and that
\[
    s_k<\eta_1\rho_{j+M_1}.
\]
Then the induction hypothesis, \eqref{eq:inclusion}, and
\eqref{eq:fixed frame flatness iteration} prove 
\eqref{eq:prop improvement of flatness base hp1},
\eqref{eq:prop improvement of flatness base hp2}, and
\eqref{eq:improvement of flatness flatness assumption}. Hence, applying
\Cref{lem:inductive step improvement of flatness} with
$\tilde u=u_k$, $r=\rho_j$, $s=s_k$, and $\sU^{(m)}=\sU_k^{j+m}$, gives an
Ansatz $\sU_k^{j+M_1+1}$ satisfying the same approximation estimate at the
scale $\rho_{j+M_1+1}$ and
\begin{equation}\label{eq:iteration ansatz increment}
    |\alpha_k^{j+M_1+1}-\alpha_k^{j+M_1}|
    +\rho_{j+M_1}|v_k^{j+M_1+1}-v_k^{j+M_1}|
    +\rho_{j+M_1}^2\|Q_k^{j+M_1+1}-Q_k^{j+M_1}\|_{L^2(\sC_1)}
    \le C_1\delta\rho_{j+M_1}^{2.5}.
\end{equation}
After decreasing $\rho_0$ once more, summing
\eqref{eq:iteration ansatz increment} gives
\[
    |\alpha_k^{j+M_1+1}-\alpha|+|v_k^{j+M_1+1}-v|
    +\|Q_k^{j+M_1+1}-Q\|_{L^2(\sC_1)}
    \le C\delta\rho_0^{1/2}\le\frac34.
\]
Together with the initial coefficient bound, this shows that
$\sU_k^{j+M_1+1}$ is an extended $|\eps_k|$-Ansatz, and the induction
continues.

Set
\[
    J\coloneqq M_1+1+
    \sup\{j\in\mathbb N\cup\{0\}:s_k<\eta_1\rho_{j+M_1}\}
    \in\mathbb N\cup\{\infty\}.
\]
By construction, \eqref{eq:iteration ansatz estimates} holds for every
integer $i\le J$. Also, if $J<\infty$, the definition of $J$ implies
\begin{equation}\label{eq:terminal dyadic scale}
    \frac{s_k}{2\eta_1}<\rho_J\le\frac{s_k}{\eta_1}.
\end{equation}
When $J<\infty$, we extend the sequence of coefficient triples constantly for
$i\ge J$.

\smallskip\noindent\textbf{Step 3.}
Define
\[
    (\alpha_k,v_k,Q_k)
    \coloneqq\lim_{i\to\infty}(\alpha_k^i,v_k^i,Q_k^i),
    \qquad
    \overline{\sU}_k\coloneqq\sU[\bar u,\eps_k\alpha_k,\eps_kv_k,\eps_kQ_k].
\]
Note that the limit exists by \eqref{eq:iteration ansatz increment}, and summing that
estimate gives, for $j \leq J$,
\begin{equation}\label{eq:final ansatz coefficient comparison}
    |\alpha_k^j-\alpha_k|+\rho_j|v_k^j-v_k|
    +\rho_j^2\|Q_k^j-Q_k\|_{L^2(\sC_1)}
    \le C\delta\rho_j^{2.5}.
\end{equation}
In particular, taking $j=0$ gives \eqref{eq:thm flatness closeness conclusion}.

By \eqref{eq:iteration ansatz estimates} and
\eqref{eq:fixed frame flatness iteration}, arguing as in
\eqref{eq:minimizing ansatz centered} and then using
\eqref{eq:final ansatz coefficient comparison}, we obtain
\[
    \frac{|\eps_k\alpha_k^j|}{\rho_j}+|\eps_kv_k^j|
      +\rho_j\|\eps_kQ_k^j\|_{L^2(\sC_1)}\le\frac\eta4,\qquad
    \frac{|\eps_k\alpha_k|}{\rho_j}+|\eps_kv_k|
      +\rho_j\|\eps_kQ_k\|_{L^2(\sC_1)}\le\frac\eta2
    \quad\forall\,j\le J,
\]
after choosing $\eta_1$ small enough.
Hence, \Cref{lem:ansatz vs parameters} together with
\eqref{eq:iteration ansatz estimates} yields
\begin{equation}\label{eq:final ansatz dyadic estimate}
    \|(u_k-\overline{\sU}_k)(\rho_j\cdot,\rho_j^2\cdot)\|_{L^2(\Omega^\beta_{R_*}\cap(B_1\times(-1,-(s_k/\rho_j)^2)))}
    \le C\delta|\eps_k|\rho_j^{3.5}
    \qquad\forall\,j\le J.
\end{equation}
Now, choose $C_\circ\ge8/\eta_1$ sufficiently large and for $C_\circ d_k\le r<\rho_0/4$ choose $j\le J$ such that
$r/\rho_j\in(1/4,1/2]$ (the existence of such $j$ when $J<\infty$ follows from 
\eqref{eq:terminal dyadic scale}). Then
\[
    S_{r/\rho_j}\bigl(B_{3/2}\times(-3/2,-1/4)\bigr)
    \subset\Omega^\beta_{R_*}\cap
    \bigl(B_1\times(-1,-(s_k/\rho_j)^2)\bigr),
\]
so \eqref{eq:final ansatz dyadic estimate} and a change of variables give
\[
    \|(u_k-\overline{\sU}_k)(r\cdot,r^2\cdot)\|_
    {L^2(B_{3/2}\times(-3/2,-1/4))}
    \le C\delta|\eps_k|r^{3.5}.
\]
Also, since $r/\rho_j\in(1/4,1/2]$, these parameter bounds imply
\[
    \frac{|\eps_k\alpha_k|}{r}+|\eps_kv_k|
    +r\|\eps_kQ_k\|_{L^2(\sC_1)}\le C\eta.
\]
Recalling \Cref{definition:ansatz perturbation} we note that, in
$\overline{\sU}_k(rx,r^2t)$, the function $\bar u$ is evaluated at points
$(y,r^2t)$ with $|y|\le Cr$, where $C$ is independent of $k$ and $r$.
Thus, we can choose $\rho_\delta\in(0,\rho_0/4]$ so that these
points belong to $\Omega_1^\beta\cap\sC_{r_5}$ whenever
$C_\circ d_k\le r<\rho_\delta$ and
$(x,t)\in B_{3/2}\times(-3/2,-1/4)$. Since $\bar u=u$ on
$\Omega_1^\beta\cap\sC_{r_5}$, we obtain
\[
    \overline{\sU}_k(rx,r^2t)
    =\sU[u,\eps_k\alpha_k,\eps_kv_k,\eps_kQ_k](rx,r^2t)
    =\sU_k(rx,r^2t)\qquad \forall\,(x,t)\in B_{3/2}\times(-3/2,-1/4).
\]
This proves \eqref{eq:thm flatness expansion conclusion}. Finally, if $r\ge C_\circ|(x_k,t_k)|_\pb$ and $C_\circ$ is large enough,
\eqref{eq:thm flatness expansion conclusion2} follows from \eqref{eq:thm flatness expansion conclusion} and 
the change of variables
    \[
      (x,t)\mapsto  (y,\tau)\coloneqq\left(\mathsf{R}_kx+\frac{x_k}{r},t+\frac{t_k}{r^2}\right).\qedhere
    \]
\end{proof}

\section{Transversal perturbations: definition and cleaning}
\label{section:cleaning}

In this section we use the expansions of solutions and their linearized perturbations at the maximal stratum to define transversality. We then apply the results of \Cref{section:improvement of flatness} to prove the crucial cleaning estimate \eqref{eq:introduction explain cleaning}.

Fix an integer $m\in\{1,\dots,n-1\}$ and let $u$ solve \eqref{eq:stefan} in $B_1\times(-1,1)$ with $(0,0)\in\Sigma_{n-1}^m(u)$ and expansion
\begin{equation}\label{eq:expansion cubic points transversality}
    u(x,t)=\tfrac12x_n^2+|x_n|\bigg(6at+(a+b)x_n^2
    -3\sum_{i=m+1}^{n-1}b_ix_i^2\bigg)+p_3^\odd(x,t)
    +o(|x,t|_\pb^3)\qquad \text{for }t \leq 0,
\end{equation}
where $p_3^\odd=x_nP_2$, $P_2$ is parabolically $2$-homogeneous with
$\Heat(x_nP_2)=0$, and $a,b_i>0$, $b\coloneqq\sum_{i=m+1}^{n-1}b_i$.
Set
\[
    q_2^\even\coloneqq6at+(a+b)x_n^2
        -3\sum_{i=m+1}^{n-1}b_ix_i^2,\qquad
    q_2^+\coloneqq q_2^\even+P_2,\qquad
    q_2^-\coloneqq P_2-q_2^\even,
\]
so that
\begin{equation}\label{eq:u q2}
    u=\tfrac12x_n^2+(x_n)_+q_2^+-(x_n)_-q_2^-+o(|x,t|_\pb^3),
    \qquad q_2^\even=\tfrac12(q_2^+-q_2^-).
\end{equation}
By \Cref{theorem:splitting at singular points}, after decreasing
$\beta\in(0,1/2)$ and a parabolic rescaling, we may write
\[
    u=u^++u^-\qquad\text{in }\Omega^{2\beta}_1\cap\sC_1.
\]
With $I_-$ as in \eqref{eq:lower branch ansatz dictionary}, set
\begin{equation}\label{eq:reflected branches}
    I_+\coloneqq Id,\qquad
    u_\sharp^\sigma\coloneqq u^\sigma\circ I_\sigma,
    \quad \sigma\in\{+,-\}.
\end{equation}
Applying \Cref{lem:extend smooth solution} to $u_\sharp^\pm$ gives
\begin{equation}\label{eq:two-sided strict branch speed}
    \lambda_+\coloneqq6a+\partial_tP_2>0,\qquad
    \lambda_-\coloneqq6a-\partial_tP_2>0,\qquad
    \lambda_++\lambda_-=12a.
\end{equation}
Accordingly, given $\varphi$ solving \eqref{eq:linearized problem}, we write
$\varphi=\varphi^++\varphi^-$ with the same splitting and set
$\varphi_\sharp^+\coloneqq\varphi^+$, $\varphi_\sharp^-\coloneqq\varphi^-\circ I_-$.
By \Cref{prop:C3beta expansion linear}, there are unique triples
$(\widetilde\alpha^\pm,\widetilde v^\pm,\widetilde Q^\pm)$ such that, for
each fixed radius $N$,
\begin{equation}\label{eq:assumption perturbation}
\begin{split}
 &\big\|\big(\varphi_\sharp^\pm-
  [(\widetilde\alpha^\pm+x\cdot\widetilde v^\pm+\widetilde Q^\pm)
       \partial_nu_\sharp^\pm
       -x_n(\widetilde v^\pm\cdot\nabla u_\sharp^\pm)]\big)
       (r\cdot,r^2\cdot)\big\|_{L^2(\Omega_N^\beta\cap\sC_1)}
 \le C_Nr^{3.9},
\end{split}
\end{equation}
where $\widetilde v^\pm\perp\bm e_n$ and
$\Heat(x_n\widetilde Q^\pm)=0$.  We write the lower triple in the original
coordinates as
\[
 (\alpha^+,v^+,Q^+)\coloneqq(\widetilde\alpha^+,\widetilde v^+,\widetilde Q^+),
 \qquad
 (\alpha^-,v^-,Q^-)\coloneqq
 (-\widetilde\alpha^-,-\widetilde v^-,-\widetilde Q^-\circ I_-).
\]
The constant and linear terms in \eqref{eq:assumption perturbation}
determine the terms of orders one and two in $\varphi^\pm$, so we define
\[
    J_{\le2}[\varphi^\pm]\coloneqq\alpha^\pm+v^\pm\cdot x,
\]
and
\begin{equation}\label{eq:jets of perturbation}
    J_{\le2}[\varphi]^\even
       \coloneqq\tfrac12\big(J_{\le2}[\varphi^+]-J_{\le2}[\varphi^-]\big)
       =\alpha_\varphi^\even+x\cdot v_\varphi^\even.
\end{equation}
We denote by $J_1[\varphi^\pm]$ and $J_2[\varphi^\pm]$ the constant
and linear parts of $J_{\le2}[\varphi^\pm]$.

Since infinitesimal time and spatial translations $\partial_tu$ and $v\cdot\nabla u$ (for $v\in\R^n$) solve
\eqref{eq:linearized problem}, the
preceding construction applies to them.
The following definition compares the even jet of a perturbation $\varphi$ with
those induced by infinitesimal time translations and translations in the
nondegenerate tangential directions.

\begin{definition}\label{definition:transversality}
    Let $u$ satisfy \eqref{eq:expansion cubic points transversality}, and
    let $\varphi$ solve \eqref{eq:linearized problem} and satisfy
    \eqref{eq:assumption perturbation}. Put
    \[
        \mathsf H\coloneqq D_x^2q_2^\even|_{\bm e_n^\perp},\qquad
        E\coloneqq\Ker\mathsf H,\qquad
        F\coloneqq E^\perp\cap\bm e_n^\perp,
    \]
    so that $\dim E=m$. We say that $\varphi$ is \emph{transversal}
    at $(0,0)$ if there are no $\theta\in\R$ and $v\in F$ such that
    \[
        J_{\le2}[\varphi]^\even
        =J_{\le2}[\theta\partial_tu+v\cdot\nabla u]^\even.
    \]
\end{definition}

For later use, set
\begin{equation}\label{eq:quantitative transversality gaps}
    \mu(\varphi)\coloneqq
    \inf_{\substack{\theta\in\R\\v\in F}}
    \big\|J_{\le2}[\varphi-\theta\partial_tu-v\cdot\nabla u]^\even\big\|,
\end{equation}
so that $\varphi$ is transversal if and only if $\mu(\varphi)>0$.
Note that, for $c,\theta\in\R$ and $v\in F$,
\[
    \mu(c\varphi)=|c|\mu(\varphi),\qquad
    \mu(\varphi+\theta\partial_tu+v\cdot\nabla u)=\mu(\varphi).
\]
If $m=n-1$, then $F=\{0\}$ and there is no
$v$-variable.

A recurrent idea is that we can always find unique $\theta\in\R, v\in F$ such that
\[
  J_{\le2}[\varphi+\theta\partial_tu+v\cdot\nabla u]^\even = x\cdot w\quad\text{for some }w\in E,
\]
so transversality is equivalent to $w\neq0$.

We first show that transversality is an open condition.
\begin{lemma}\label{lem:transversality open conditions}
    Fix an integer $1\le m\le n-1$. Let $u_k\to u$ in
    $L^2(B_1\times(-1,1))$ be solutions of \eqref{eq:stefan}, and let
    $z_k\to(0,0)$ with
    \[
        z_k\in\Sigma_{n-1}^m(u_k),\qquad
        (0,0)\in\Sigma_{n-1}^m(u).
    \]
    Let $\varphi_k\to\varphi$ in $L^2(B_1\times(-1,1))$, where
    $\varphi_k$ and $\varphi$ solve \eqref{eq:linearized problem} for
    $u_k$ and $u$, respectively. If $\varphi$ is transversal to $u$ at $(0,0)$, then, for all sufficiently large
    $k$, $\varphi_k$ is transversal to
    $u_k$ at $z_k$.
\end{lemma}
\begin{proof}
    Choose the normal at $z_k$ with positive scalar product with $\bm e_n$.
    The uniform bounds from \Cref{corollary:uniform smooth expansion maximal stratum}
    make the estimates in \Cref{prop:C3beta expansion linear} uniform in $k$.
    Hence, by compactness and uniqueness of the expansions, the jets of $u_k$
    and $\varphi_k$ at $z_k$ converge to those of $u$ and $\varphi$ at the origin.
    In particular, denoting by
    $\mathsf H_k$ the Hessian in \Cref{definition:transversality} for $u_k$
    at $z_k$, its kernel $E_k$ and orthogonal complement $F_k$ converge
    to $E$ and $F$.

    Suppose that transversality fails along a subsequence.
    Then there are $\theta_k\in\R$ and $v_k\in F_k$ such that
    \[
        J_{\le2}[\varphi_k]^\even
        =J_{\le2}[\theta_k\partial_tu_k+v_k\cdot\nabla u_k]^\even.
    \]
    Comparing the constant coefficients and then the $F_k$ components of
    the linear coefficients shows that $(\theta_k,v_k)$ is bounded, since
    $a_k\to a>0$ and $\mathsf H_k|_{F_k}$ converges to the invertible
    map $\mathsf H|_F$. Passing to a subsequence and then to the limit
    contradicts the transversality of $\varphi$.
\end{proof}

We next look at the variations induced by a time translation, a
spatial translation, and a rotation of the normal: for $\theta\in\R$, $V\in\R^n$, and
$w\in\bm e_n^\perp$, set
\begin{equation}\label{eq:infinitesimal rototranslation}
    \Psi_{\theta,V,w}
    \coloneqq\theta\partial_tu+V\cdot\nabla u
      +(w\cdot x)\partial_nu-x_n(w\cdot\nabla u).
\end{equation}
As the next lemma shows, the rotational term has zero even jet,
but it contributes to the full first variation and is therefore
needed when comparing solutions in different coordinate systems.
\begin{lemma}\label{lem:jets rototranslations}
    Let $u$ satisfy \eqref{eq:expansion cubic points transversality},
    consider $\eps_k\to0^+$, $\theta_k\to\theta\in\R$,
    $V_k\to V\in\R^n$,
    $w_k\to w\in\bm e_n^\perp$, and set
    \[
        \mathsf{R}_k\coloneqq \mathsf{R}_{\frac{\bm e_n+\eps_kw_k}
                              {|\bm e_n+\eps_kw_k|}} \in SO(n).
    \]
    Then
    \[
        \frac{u(\mathsf{R}_k\cdot+\eps_kV_k,\cdot+\eps_k\theta_k)-u}{\eps_k}
        \longrightarrow\Psi_{\theta,V,w}
    \]
    in $L^2_{\rm loc}(B_1\times(-1,1))$. Moreover, recalling \eqref{eq:u q2},
    \[
        J_1[\partial_tu^\pm]=\partial_tq_2^\pm,
        \qquad
        J_{\le2}[V\cdot\nabla u^\pm]
        =V\cdot\bm e_n+x\cdot D_x^2q_2^\pm
          \bigl(V-(V\cdot\bm e_n)\bm e_n\bigr),
    \]
    and
    \[
        J_{\le2}[(w\cdot x)\partial_nu-x_n(w\cdot\nabla u)]^\even=0.
    \]
\end{lemma}
\begin{proof}
    The first assertion follows from the regularity of $u$ and the identity
    \[
        \left.\frac d{d\eps}
        \mathsf{R}_{\frac{\bm e_n+\eps w}{|\bm e_n+\eps w|}}x
        \right|_{\eps=0}
        =(w\cdot x)\bm e_n-x_nw.
    \]
    Next, set $\sP^\pm\coloneqq\tfrac12(x_n+q_2^\pm)_\pm^2$ and note that, by
    \Cref{corollary:uniform smooth expansion maximal stratum} and \Cref{lem:caccioppoli},
    \[
    \begin{split}
        \|\bigl(\partial_tu^\pm-\partial_t\sP^\pm\bigr)
            (r\cdot,r^2\cdot)\|_{L^2(B_1\times(-1,-1/2))}&=o(r),\\
        \|\bigl(\nabla u^\pm-\nabla\sP^\pm\bigr)
            (r\cdot,r^2\cdot)\|_{L^2(B_1\times(-1,-1/2))}&=o(r^2).
    \end{split}
    \]
    Differentiating $\sP^\pm$ gives the formulas for $J_1$ and $J_{\leq 2}$.
    Finally, recalling \eqref{eq:ansatz linearized problem}, we have
\[
    (w\cdot x)\partial_nu^\pm-x_n(w\cdot\nabla u^\pm)
    =\upvarphi_{u^\pm}[0,w,0],
\]
so the coefficient triples in the original coordinates are
$(0,w,0)$ for both sides of the free boundary, and therefore their jets are both
$x\cdot w$ and cancel in the even jet.
\end{proof}

We now study how rotations and translations affect transversality.
\begin{lemma}\label{lem:transversality vs rototranslations}
    With the notation of \Cref{definition:transversality} and \eqref{eq:infinitesimal rototranslation},
    \[
        \mu(\varphi+\Psi_{\theta,V,w})=\mu(\varphi).
    \]
    Suppose in addition that $u_k\to u$ in
    $L^2(B_1\times(-1,1))$ are solutions of \eqref{eq:stefan} and that,
    for some $\eps_k\to0^+$,
    \begin{equation}\label{eq:transversality full quotient}
        \frac{u_k-u}{\eps_k}\to\varphi
        \qquad\text{strongly in }L^2(B_1\times(-1,1)),
    \end{equation}
    and let $z_k=(x_k,t_k)\in\Sigma_{n-1}^m(u_k)$ with $z_k\to(0,0).$
    If $\varphi$ is transversal, then
    \[
        z_k\ne0\quad\text{for all sufficiently large }k,\qquad
        \eps_k=o(|z_k|_\pb).
    \]
\end{lemma}
\begin{proof}
    Let $V_F$ be the orthogonal projection of $V$ onto $F$; by
    \Cref{lem:jets rototranslations},
    \[
        J_{\le2}[V\cdot\nabla u]^\even
        =x\cdot\mathsf HV_F
        =J_{\le2}[V_F\cdot\nabla u]^\even,\qquad
        J_{\le2}[(w\cdot x)\partial_nu-x_n(w\cdot\nabla u)]^\even=0,
    \]
    so the invariance follows from \eqref{eq:infinitesimal rototranslation}
    and \eqref{eq:quantitative transversality gaps}.
    Before proving the remaining assertions, we introduce some notation.
    After a fixed parabolic rescaling by $r_*>0$ and decreasing
    $\beta$, we may assume that the functions are defined in
    $B_2\times(-2,2)$ and that
    \Cref{corollary:uniform smooth expansion maximal stratum}, together
    with \eqref{eq:cubic flatness}, gives
    \eqref{eq:flatness assumption} for $u^\pm$ and $u_k^\pm$ in their
    inward coordinates, with one common $r_0>0$.
    
    By \Cref{theorem:splitting at singular points}, choose the splitting of
    $u_k$ so that the inward normal of $u_k^+$ at $z_k$ has positive scalar
    product with $\bm e_n$, and with $I_\sigma$ as in
    \eqref{eq:reflected branches} define
    \begin{equation}\label{eq:transversality branch domains}
        u_k=u_k^++u_k^-,\qquad
        u_{k,\sharp}^\sigma\coloneqq u_k^\sigma\circ I_\sigma,
        \qquad x_k^\sigma\coloneqq I_\sigma x_k,\qquad \mathscr D_k^\sigma\coloneqq
        \Omega_1^\beta\cap\Omega_1^\beta(x_k^\sigma,t_k)
        \cap\sC_1,\qquad \sigma\in\{+,-\}.
    \end{equation}

    \smallskip\noindent\textbf{Step 1.}
    We prove that
    \begin{equation}\label{eq:transversality branchwise quotient}
        \left\|\frac{u_{k,\sharp}^\sigma-u_\sharp^\sigma}{\eps_k}
        -\varphi_\sharp^\sigma\right\|_{L^2(\mathscr D_k^\sigma)}
        \longrightarrow0,
        \qquad \sigma\in\{+,-\}.
    \end{equation}
    By \Cref{lem:caccioppoli} with $\bar\eps=0$ and \eqref{eq:transversality full quotient}, we obtain
    \begin{equation}\label{eq:transversality local Linfty}
        \|u_k-u\|_{L^\infty(B_{3/2}\times(-3/2,3/2))}
        \le C\|u_k-u\|_{L^2(B_2\times(-2,2))}
        \le C\eps_k.
    \end{equation}
    By \Cref{lem:smoothness close maximal stratum} the free boundaries
    are smooth graphs, and the construction in the proof of
    \Cref{lem:extend smooth solution} gives smooth 
    functions $\sA^\sigma$ and $\sA_k^\sigma$ such that
    \begin{equation}\label{eq:transversality defining functions}
    \begin{gathered}
        u_\sharp^\sigma=\tfrac12(\sA^\sigma)_+^2,
        \qquad
        u_{k,\sharp}^\sigma=\tfrac12(\sA_k^\sigma)_+^2,\\
        \{u_\sharp^\sigma>0\}=\{\sA^\sigma>0\},
        \qquad
        \{u_{k,\sharp}^\sigma>0\}=\{\sA_k^\sigma>0\},\\
        \tfrac12\le\partial_n\sA^\sigma\le\tfrac32,
        \qquad
        \tfrac12\le\partial_n\sA_k^\sigma\le\tfrac32,
        \qquad
        |\nabla\sA^\sigma|+|\nabla\sA_k^\sigma|\le C
        \quad\text{in }\mathscr D_k^\sigma.
    \end{gathered}
    \end{equation}
    Define
    \[
        \mathcal E_k^\sigma\coloneqq\mathscr D_k^\sigma\cap
        \bigl(\{u_{k,\sharp}^\sigma>0\}
        \mathbin\triangle\{u_\sharp^\sigma>0\}\bigr).
    \]
    Equations \eqref{eq:transversality local Linfty} and
    \eqref{eq:transversality defining functions}, applied on each line
    parallel to $\bm e_n$, give
    \begin{equation}\label{eq:transversality mismatch bounds}
            \mathcal E_k^\sigma\subset
        \{|\sA^\sigma|\le C\eps_k^{1/2}\},
        \qquad |\mathcal E_k^\sigma|\le C\eps_k^{1/2},\qquad
        u_{k,\sharp}^\sigma+u_\sharp^\sigma\le C\eps_k
        \qquad\text{on }\mathcal E_k^\sigma,
    \end{equation}
    therefore
    \begin{equation}\label{eq:transversality mismatch vanishing}
        \eps_k^{-2}\int_{\mathcal E_k^\sigma}
        (u_{k,\sharp}^\sigma+u_\sharp^\sigma)^2\,dx\,dt
        \le C\eps_k^{1/2}\longrightarrow0.
    \end{equation}
    Fix $\eta>0$. By \eqref{eq:transversality mismatch bounds}, for all
    sufficiently large $k$,
    \[
    \begin{aligned}
        \frac{u_{k,\sharp}^\sigma-u_\sharp^\sigma}{\eps_k}
        &=\left(\frac{u_k-u}{\eps_k}\right)\circ I_\sigma,
        &\varphi_\sharp^\sigma&=\varphi\circ I_\sigma
        &&\text{ in }\mathscr D_k^\sigma\cap\{\sA^\sigma>\eta\},\\
        u_{k,\sharp}^\sigma&=u_\sharp^\sigma=0,
        &\varphi_\sharp^\sigma&=0
        &&\text{ in }\mathscr D_k^\sigma\cap\{\sA^\sigma<-\eta\},
    \end{aligned}
    \]
    and combining these identities with \eqref{eq:transversality full quotient}
    and \eqref{eq:transversality mismatch vanishing} gives
    \[
    \begin{split}
        \limsup_{k\to\infty}
        \left\|\frac{u_{k,\sharp}^\sigma-u_\sharp^\sigma}{\eps_k}
        -\varphi_\sharp^\sigma\right\|_{L^2(\mathscr D_k^\sigma)}
        \le C\|\varphi_\sharp^\sigma\|_
        {L^2((\Omega_1^\beta\cap\sC_1)
        \cap\{|\sA^\sigma|<2\eta\})}.
    \end{split}
    \]
    Since $\varphi_\sharp^\sigma=0$ in
    $\{u_\sharp^\sigma=0\}$, the right-hand side tends to zero as
    $\eta\downarrow0$, proving \eqref{eq:transversality branchwise quotient}.

    \smallskip\noindent\textbf{Step 2.}
    Suppose by contradiction that, along a subsequence,
    \begin{equation}\label{eq:transversality bad scale}
        d_k\coloneqq|z_k|_\pb\le C\eps_k,
    \end{equation}
    and set $\rho_k\coloneqq\eps_k^{1/2}$. Since
    $|t_k|\le d_k^2$, \eqref{eq:transversality bad scale} gives
    \begin{equation}\label{eq:transversality intermediate scale}
        \frac{d_k}{\rho_k}\longrightarrow0,
        \qquad \frac{\eps_k}{\rho_k}\longrightarrow0,
        \qquad \frac{t_k}{\eps_k}\longrightarrow0.
    \end{equation}
    Let $\nu_k$ be the inward unit normal of $u_k^+$ at $z_k$, and
    define the following matrices in
    $SO(n)$:
    \[
        O_k\coloneqq \mathsf{R}_{\nu_k}^{\top},\qquad
        O_k^+\coloneqq O_k,\qquad O_k^-\coloneqq I_-O_kI_-.
    \]
    For each $\delta>0$, let $k_\delta$ and $\rho_\delta$ be given by
    \Cref{theorem:improvement of flatness}. After passing to a subsequence,
    choose $\delta_k\downarrow0$ slow enough so that
    \begin{equation}\label{eq:transversality diagonal scales}
        k\ge k_{\delta_k},\qquad
        C_\circ d_k<\rho_k<\rho_{\delta_k}.
    \end{equation}
    By the choice of $r_*$ in Step~1 and
    \eqref{eq:transversality branchwise quotient},
    \Cref{theorem:improvement of flatness} yields triples
    $(\widetilde\alpha_k^\sigma,\widetilde v_k^\sigma,
      \widetilde Q_k^\sigma)$ such that
    \begin{equation}\label{eq:transversality parameters}
    \begin{split}
        &|\widetilde\alpha_k^\sigma-\widetilde\alpha^\sigma|
        +|\widetilde v_k^\sigma-\widetilde v^\sigma|
        +\|\widetilde Q_k^\sigma-\widetilde Q^\sigma\|_{L^2(\sC_1)}
        \le C\delta_k,
    \end{split}
    \end{equation}
    and, setting $\sU_k^\sigma\coloneqq
        \sU[u_\sharp^\sigma,\eps_k\widetilde\alpha_k^\sigma,
              \eps_k\widetilde v_k^\sigma,
              \eps_k\widetilde Q_k^\sigma],$
    \begin{equation}\label{eq:transversality moving estimates}
        \big\|(u_{k,\sharp}^\sigma-\sU_k^\sigma)
        (x_k^\sigma+\rho_kO_k^\sigma\,\cdot,
         t_k+\rho_k^2\,\cdot)\big\|_{L^2(B_1\times(-1,-1/2))}
        \le C\delta_k\eps_k\rho_k^{7/2}.
    \end{equation}
    Let
$K\coloneqq B_{1/8}(\tfrac12\bm e_n)
        \times(-\tfrac34,-\tfrac23)
        \Subset B_1\times(-1,-1/2)$ and define
    \begin{equation}\label{eq:transversality ansatz defining function}
        \widehat{\sA}_k^\sigma(x,t)
        \coloneqq\sA^\sigma\big(S_k^\sigma x
        +\eps_k\widetilde\alpha_k^\sigma\bm e_n
        +\eps_k\widetilde Q_k^\sigma
          (S_k^\sigma x+\eps_k\widetilde\alpha_k^\sigma\bm e_n,t)
          \bm e_n,t\big),\qquad S_k^\sigma\coloneqq
        \mathsf{R}_{\frac{\bm e_n+\eps_k\tilde v_k^\sigma}
        {|\bm e_n+\eps_k\tilde v_k^\sigma|}}.
    \end{equation}
    Then $\sU_k^\sigma=\tfrac12(\widehat{\sA}_k^\sigma)_+^2$. Also, \eqref{eq:transversality intermediate scale} and
    \eqref{eq:transversality parameters} together with
    \Cref{corollary:uniform smooth expansion maximal stratum} give
    \begin{equation}\label{eq:transversality defining functions intermediate scale}
    \begin{split}
        \rho_k^{-1}\sA_k^\sigma
        (x_k^\sigma+\rho_kO_k^\sigma y,t_k+\rho_k^2s)
        &=y_n+O(\rho_k),\\
        \rho_k^{-1}\widehat{\sA}_k^\sigma
        (x_k^\sigma+\rho_kO_k^\sigma y,t_k+\rho_k^2s)
        &=y_n+o(1),
     \end{split}
    \end{equation}
    uniformly for $(y,s)\in K$.
    Thus both functions $\sA_k^\sigma$ and $\widehat{\sA}_k^\sigma$
    are bounded below by $c\rho_k$ on $K$, so
    \eqref{eq:defining function difference of squares} and
    \eqref{eq:transversality moving estimates} give
    \begin{equation}\label{eq:transversality defining function comparison}
        \big\|(\sA_k^\sigma-\widehat{\sA}_k^\sigma)
        (x_k^\sigma+\rho_kO_k^\sigma\,\cdot,
         t_k+\rho_k^2\,\cdot)\big\|_{L^2(K)}
        \le C\delta_k\eps_k\rho_k^{5/2}.
    \end{equation}
    Next, write the order-five expansions on $K$ as
    \[
    \begin{split}
        \sA_k^\sigma
        (x_k^\sigma+\rho_kO_k^\sigma y,t_k+\rho_k^2s)
        &=\sum_{j=0}^4\rho_k^jA_{j,k}^\sigma(y,s)+R_k^\sigma,\\
        \widehat{\sA}_k^\sigma
        (x_k^\sigma+\rho_kO_k^\sigma y,t_k+\rho_k^2s)
        &=\sum_{j=0}^4\rho_k^j\widehat A_{j,k}^\sigma(y,s)
          +\widehat R_k^\sigma,
    \end{split}
    \]
    where the coefficients are parabolically homogeneous and
    \[
        \|R_k^\sigma\|_{L^2(K)}
        +\|\widehat R_k^\sigma\|_{L^2(K)}\le C\rho_k^{9/2},
        \qquad
        A_{0,k}^\sigma=0,\qquad A_{1,k}^\sigma=y_n.
    \]
   Note that polynomial norm equivalence and
    \eqref{eq:transversality defining function comparison} imply
    \begin{equation}\label{eq:transversality coefficient comparison}
        \rho_k^j\|A_{j,k}^\sigma-\widehat A_{j,k}^\sigma\|
        \le C\bigl(\delta_k\eps_k\rho_k^{5/2}+\rho_k^{9/2}\bigr),
        \qquad j\in\{0,1\}.
    \end{equation}
   In particular, since $\eps_k=\rho_k^2$, the right-hand side divided by
    $\eps_k\rho_k^j$ tends to zero for $j\in\{0,1\}$, and
the linear part of $\widehat{\sA}_k^\sigma$ in these coordinates is
    $
        \nabla\widehat{\sA}_k^\sigma(x_k^\sigma,t_k)
        \cdot O_k^\sigma y.
    $
    By \eqref{eq:transversality bad scale},
    \eqref{eq:transversality parameters}, and the uniform $C^2$ bounds for
    $\sA^\sigma$, its gradient at $(x_k^\sigma,t_k)$ differs from
    $\bm e_n$ by $O(\eps_k)$. Hence, taking $j=1$ in
    \eqref{eq:transversality coefficient comparison} and recalling the
    definition of $\mathsf{R}_e$, we obtain $|O_k-Id|\le C\eps_k$. Thus, passing to
    a subsequence,
    \begin{equation}\label{eq:transversality geometric limits}
        \frac{x_k}{\eps_k}\longrightarrow V,
        \qquad
        \frac{O_k-Id}{\eps_k}\longrightarrow
        \bm e_n\otimes w-w\otimes\bm e_n
    \end{equation}
    for some $V\in\R^n$ and $w\in\bm e_n^\perp$. 
    
    Set $\Phi\coloneqq\varphi+\Psi_{0,V,w}$, and write
    \[
        \sA^\sigma(y,s)=y_n+\widehat q_2^\sigma(y,s)
        +O(|(y,s)|_\pb^3).
    \]
    Let $\widetilde\alpha_\Phi^\sigma,\widetilde v_\Phi^\sigma$
    be the first two coefficients of $\Phi_\sharp^\sigma$ in
    \eqref{eq:assumption perturbation}. Differentiating
    \eqref{eq:transversality ansatz defining function} and using
    \eqref{eq:transversality parameters},
    \eqref{eq:transversality geometric limits}, and
    \Cref{lem:jets rototranslations}, we obtain
    \[
        \widehat A_{0,k}^\sigma
            =\eps_k\widetilde\alpha_\Phi^\sigma+o(\eps_k),\qquad
        \widehat A_{1,k}^\sigma
            =y_n+\eps_k\bigl(y'\cdot\widetilde v_\Phi^\sigma
              +\widetilde\alpha_\Phi^\sigma\partial_n\widehat q_2^\sigma\bigr)
              +o(\eps_k).
    \]
    Since $A_{0,k}^\sigma=0$ and $A_{1,k}^\sigma=y_n$,
    \eqref{eq:transversality coefficient comparison} gives
    $\widetilde\alpha_\Phi^\sigma=0$ and
    $\widetilde v_\Phi^\sigma=0$. Thus $J_{\le2}[\Phi]^\even=0$,
    hence $\mu(\Phi)=0$, contradicting $\mu(\Phi)=\mu(\varphi)>0$ (recall that, by assumption, 
    $\varphi$ is transversal).
\end{proof}

We now show that, under a transversal perturbation, a singular point can persist in the same stratum only if the perturbation size is bounded by any prescribed power of the parabolic displacement of the point. This is the ``cleaning estimate'' \eqref{eq:introduction explain cleaning}.

We use the coordinates of \eqref{eq:expansion cubic points transversality} on both sides of $\{y_n=0\}$, with the convention in \eqref{eq:lower branch ansatz dictionary}, and set
\[
   K^+\coloneqq\left \{|y'|<\tfrac18,\ \tfrac14<y_n<\tfrac38,
                   -\tfrac78<s<-\tfrac58\right\},\qquad
   K^-\coloneqq I_-K^+.
\]
For later use, we define
\begin{equation}\label{eq:ansatz change of variables}
    \mathcal T_{\alpha,v,Q}(x,t)\coloneqq\bigg(
    \mathsf{R}_{\frac{\bm e_n+v}{|\bm e_n+v|}}x+\alpha\bm e_n
    +Q\left(\mathsf{R}_{\frac{\bm e_n+v}{|\bm e_n+v|}}x
             +\alpha\bm e_n,t\right)\bm e_n,t\bigg),
\end{equation}
so that \Cref{definition:ansatz perturbation} reads
$\sU[w,\alpha,v,Q]=w\circ\mathcal T_{\alpha,v,Q}$.

\begin{proposition}\label{prop:cleaning with transversality}
    Let $u$ and $u_k$ solve \eqref{eq:stefan} in
    $B_1\times(-1,1)$, suppose that
    \[
        (0,0)\in\Sigma_{n-1}^m(u),\qquad
        z_k=(x_k,t_k)\in\Sigma_{n-1}^m(u_k),\qquad z_k\to0,
    \]
    and assume that $u$ satisfies
    \eqref{eq:expansion cubic points transversality}. Suppose moreover that,
    for some $\eps_k\to0^+$,
    \[
        \frac{u_k-u}{\eps_k}\longrightarrow\varphi
        \qquad\text{strongly in }L^2(B_1\times(-1,1)),
    \]
    where $\varphi$ solves \eqref{eq:linearized problem} and satisfies
    \eqref{eq:assumption perturbation}. Set
    $d_k\coloneqq|z_k|_\pb$ and assume that $\varphi$ is transversal
    at $(0,0)$.
    Then, for every $N\ge1$, there exists $C_N>0$ such that,
    for all sufficiently large $k$,
    \begin{equation}\label{eq:cleaning epsilon distance}
        \eps_k\le C_Nd_k^N.
    \end{equation}
\end{proposition}
\begin{proof}
    Since $d_k\to0$, it suffices to prove the estimate for every
    integer $N\ge3$, so we fix such an $N$ and argue by contradiction.
    We divide the proof into four steps, comparing the defining
    polynomials after one common rototranslation and applying the
    \L ojasiewicz inequality.

    \smallskip\noindent\textbf{Step 1.}
    Suppose by contradiction that
    \begin{equation}\label{eq:cleaning contradiction scale}
        d_k^N=o(\eps_k)
    \end{equation}
   and recall that, thanks to \Cref{lem:transversality vs rototranslations},
    $\eps_k=o(d_k)$.

    Let $C_\circ$ be given by \Cref{theorem:improvement of flatness}, chosen
    sufficiently large. Set $r_k\coloneqq2C_\circ d_k$ and let
    $O_k\coloneqq \mathsf{R}_{\nu_k}^{\top}$, where $\nu_k\to\bm e_n$ is the
    common normal at $z_k$. By
    \eqref{eq:transversality branchwise quotient} and
    \Cref{theorem:improvement of flatness}, a diagonal choice gives
    coefficients $(\alpha_k^\sigma,v_k^\sigma,Q_k^\sigma)$ converging
    to those of $\varphi^\sigma$, with the convention
    \eqref{eq:lower branch ansatz dictionary}.
    Let $\sA^\sigma$ and $\sA_{k,z}^\sigma$ define the order-$(N+3)$
    expansions of $u$ at $0$ and $u_k$ at $z_k$, respectively,
    the latter in the coordinates determined by $O_k$.
    Since
    $r_k^{N+7/2}=o(\eps_kr_k^3)$, arguing as in
    \eqref{eq:defining function difference of squares} and using
    \eqref{eq:thm flatness expansion conclusion2} and
    \Cref{corollary:uniform smooth expansion maximal stratum}, we obtain, for $\sigma=\pm$,
    \begin{equation}\label{eq:cleaning first defining function comparison}
        \Big\|
        (\sA^\sigma\circ
        \mathcal T_{\eps_k\alpha_k^\sigma,\eps_kv_k^\sigma,\eps_kQ_k^\sigma})
        (x_k+r_kO_k\,\cdot,t_k+r_k^2\,\cdot)
        -\sA_{k,z}^\sigma(r_k\,\cdot,r_k^2\,\cdot)
        \Big\|_{L^2(K^\sigma)}
        =o(\eps_kr_k^2)
    \end{equation}
    (to prove the bound above, we used that the two polynomials in
    \eqref{eq:cleaning first defining function comparison} have the
    same sign and absolute value comparable to $r_k$ on $K^\sigma$).
    Hence,
    by polynomial norm equivalence in
    \eqref{eq:cleaning first defining function comparison} and a Taylor
    expansion of \eqref{eq:ansatz change of variables}, we obtain
    \[
        \sA^\sigma(z_k)+\eps_k\alpha_k^\sigma=O(\eps_kr_k),
        \qquad
        \nabla_x\sA^\sigma(z_k)-\nu_k=O(\eps_k).
    \]
    Now, set $\nu_{\sA}\coloneqq\nabla_x\sA^\odd/|\nabla_x\sA^\odd|$.
Adding and subtracting the gradient estimates,
    and using that $\alpha_k^\sigma\to\alpha^\sigma$, yields
    \begin{equation}\label{eq:cleaning low coefficients}
        |\sA^+(z_k)|+|\sA^-(z_k)|
        +|\nu_k-\nu_{\sA}(z_k)|
        +|\pi_F\nabla_x\sA^\even(z_k)|\le C\eps_k,\qquad \eps_k^{-1}\sA^\sigma(z_k)\longrightarrow-\alpha^\sigma.
    \end{equation}

    \smallskip\noindent\textbf{Step 2.}
    We want to find a bounded sequence of parameters
    $\zeta_k,\theta_k\in\R$, $V_{F,k}\in F$, and
    $w_k\in\bm e_n^\perp$ such that, writing
    \[
        \mathsf{R}_k\coloneqq
        \mathsf{R}_{\frac{\bm e_n+\eps_kw_k}{|\bm e_n+\eps_kw_k|}},
        \qquad
        \widehat z_k\coloneqq
        (\mathsf{R}_kx_k+\eps_k(\zeta_k\bm e_n+V_{F,k}),t_k+\eps_k\theta_k),
        \qquad \widehat\nu_k\coloneqq \mathsf{R}_k\nu_k,
    \]
    we have
    \begin{equation}\label{eq:cleaning common gauge}
        \sA^\pm(\widehat z_k)=0,\qquad
        \pi_F\nabla_x\sA^\even(\widehat z_k)=0,\qquad
        \widehat\nu_k=\nu_{\sA}(\widehat z_k).
    \end{equation}
    Note that, for large $k$, the unit vectors $\widehat\nu_k$ and
    $\nu_{\sA}(\widehat z_k)$ have positive $\bm e_n$-component, so the last
    equation in \eqref{eq:cleaning common gauge} is
    equivalent to
    $\pi_{\bm e_n^\perp}(\widehat\nu_k-\nu_{\sA}(\widehat z_k))=0$.
    After writing it in this form, divide all the left-hand sides in
    \eqref{eq:cleaning common gauge} by $\eps_k$ and note that, if we choose $(\zeta_k,\theta_k,V_{F,k},w_k)=(0,0,0,0)$, then we obtain the vector
    \[
        \frac1{\eps_k}
        \begin{pmatrix}
            \sA^+(z_k)\\
            \sA^-(z_k)\\
            \pi_F\nabla_x\sA^\even(z_k)\\
            \pi_{\bm e_n^\perp}(\nu_k-\nu_{\sA}(z_k))
        \end{pmatrix}.
    \]
    Thanks to \eqref{eq:cleaning low coefficients}, this vector is bounded and, after passing to a further subsequence, it converges to
    $-(\alpha^+,\alpha^-,\eta_F,\eta_\nu)^T$, where
    $\eta_F\in F$ and $\eta_\nu\in\bm e_n^\perp$.

    To compute the limit as $\eps_k\to0$ of \eqref{eq:cleaning common gauge}, we observe that
    \[
        \mathsf{R}_k=Id+\eps_k(\bm e_n\otimes w_k-w_k\otimes\bm e_n)
        +O(\eps_k^2).
    \]
    Thus, since $x_k\to0$ and $\nu_k\to\bm e_n$,
    \[
        \widehat x_k=x_k+\eps_k(\zeta_k\bm e_n+V_{F,k})+o(\eps_k),
        \qquad
        \widehat\nu_k=\nu_k-\eps_kw_k+o(\eps_k).
    \]
  Moreover, since $\partial_{in}q_2^\pm=0$ for $i<n$,
    \[
        \nabla_x\sA^\pm(0,0)=\bm e_n,\qquad
        \partial_t\sA^\pm(0,0)=\pm\lambda_\pm,
    \qquad
        \nabla_F\partial_n\sA^\even(0,0)=0,\qquad
        \partial_n\nu_{\sA}(0,0)=0.
    \]
    Hence, by a Taylor expansion, the rescaled equations converge in $C^1$ to the affine system
    \begin{equation}\label{eq:cleaning limiting affine system}
        \begin{pmatrix}
            1&\lambda_+&0&0\\
            1&-\lambda_-&0&0\\
            0&\partial_t\nabla_F\sA^\even(0,0)&\mathsf H|_F&0\\
            0&-\partial_t\nu_{\sA}(0,0)&-D_x\nu_{\sA}(0,0)|_F
                &-Id_{\bm e_n^\perp}
        \end{pmatrix}
        \begin{pmatrix}\zeta\\\theta\\V_F\\w\end{pmatrix}
        =\begin{pmatrix}
            \alpha^+\\
            \alpha^-\\
            \eta_F\\
            \eta_\nu
        \end{pmatrix}.
    \end{equation}
    The first diagonal block has determinant
    $-(\lambda_++\lambda_-)=-12a<0$ by
    \eqref{eq:expansion cubic points transversality}.
    Also, $\mathsf H|_F$ is invertible by
    the definition of $F$. Hence the Jacobian is invertible, and
    \eqref{eq:cleaning limiting affine system} has a unique solution
    $(\zeta,\theta,V_F,w)$.
    Hence, by the implicit function theorem, there exist parameters
    $(\zeta_k,\theta_k,V_{F,k},w_k)\to(\zeta,\theta,V_F,w)$
    satisfying \eqref{eq:cleaning common gauge}.

    \smallskip\noindent\textbf{Step 3.}
    With the notation from Step 2, define
    \[
        \widehat u_k(x,t)\coloneqq
        u_k\bigl(\mathsf{R}_k^T[x-\eps_k(\zeta_k\bm e_n+V_{F,k})],
        t-\eps_k\theta_k\bigr).
    \]
    Then $\widehat z_k\in\Sigma_{n-1}^m(\widehat u_k)$ and, by
    \Cref{lem:jets rototranslations,lem:transversality vs rototranslations},
    \[
        \frac{\widehat u_k-u}{\eps_k}\longrightarrow
        \varphi_*\coloneqq
        \varphi+\Psi_{-\theta,-\zeta\bm e_n-V_F,-w}
        \qquad\text{in }L^2_{\rm loc},
    \]
    and $\varphi_*$ is transversal.

    Set $\hat r_k\coloneqq2C_\circ\max\{d_k,|\widehat z_k|_\pb\}$. Then, since $\eps_k=o(d_k)$, it follows by \eqref{eq:cleaning contradiction scale}, the boundedness of $\theta_k$, and $N>2$, that
    \[
        \hat r_k^N\le
        C\bigl(d_k^N+(\eps_k|\theta_k|)^{N/2}\bigr)
        =o(\eps_k),
    \]
    thus
    \begin{equation}\label{eq:cleaning post gauge scale}
        \eps_k=o(\hat r_k),\qquad \hat r_k^N=o(\eps_k).
    \end{equation}
    Let $\widehat O_k\coloneqq \mathsf{R}_{\hat\nu_k}^{\top}$ and let
    $\widehat\sA_{k,z}^\sigma$ define the order-$(N+3)$ expansion of
    $\widehat u_k$ at $\widehat z_k$ in these coordinates.
    Thanks to \eqref{eq:cleaning post gauge scale}, we can argue as in
    \eqref{eq:cleaning first defining function comparison} and obtain
    coefficients $(\bar\alpha_k^\sigma,\bar v_k^\sigma,\bar Q_k^\sigma)$
    converging to those of $\varphi_*^\sigma$ and satisfying, for $\sigma=\pm$,
    \begin{equation}\label{eq:cleaning defining function comparison}
        \Big\|
        (\sA^\sigma\circ
        \mathcal T_{\eps_k\bar\alpha_k^\sigma,\eps_k\bar v_k^\sigma,
        \eps_k\bar Q_k^\sigma})
        (\widehat x_k+\hat r_k\widehat O_k\,\cdot,\widehat t_k+\hat r_k^2\,\cdot)\\
        -\widehat\sA_{k,z}^\sigma(\hat r_k\,\cdot,\hat r_k^2\,\cdot)
        \Big\|_{L^2(K^\sigma)}
        =o(\eps_k\hat r_k^2).
    \end{equation}
    By polynomial norm equivalence in
    \eqref{eq:cleaning defining function comparison}, a Taylor expansion
    of \eqref{eq:ansatz change of variables}, and the fact that
    $\sA^\sigma(\widehat z_k)=0$ (see \eqref{eq:cleaning common gauge}),
    the constant coefficients give
    \[
        \bar\alpha_k^\sigma+\bar v_k^\sigma\cdot\widehat x_k
        =O(\hat r_k|\bar\alpha_k^\sigma|+\hat r_k^2+\eps_k).
    \]
   In particular, since $\eps_k=o(\hat r_k)$ and $\bar v_k^\sigma\cdot\widehat x_k = O(\hat r_k)$, we obtain
    $|\bar\alpha_k^\sigma|\le C\hat r_k$.
    
   Next, expanding the linear coefficients gives
    \[
        \widehat O_k^T
        \bigl(\nabla_x\sA^\sigma(\widehat z_k)+\eps_k\bar v_k^\sigma\bigr)
        =\bm e_n+o(\eps_k).
    \]
    Subtracting the identities for $\sigma=+$ and $\sigma=-$,
    dividing by $2$, and using that
    $\tfrac12(\bar v_k^+-\bar v_k^-)\to v_{\varphi_*}^\even$, we obtain
    \begin{equation}\label{eq:cleaning low jet}
        |\bar\alpha_k^\sigma|\le C\hat r_k,\qquad
        \nabla_x\sA^\even(\widehat z_k)
        +\eps_kv_{\varphi_*}^\even=o(\eps_k),
    \end{equation}
 so by \eqref{eq:cleaning common gauge} and \eqref{eq:cleaning low jet} and recalling that $\bar\alpha_k^\sigma\to\alpha_{\varphi_*}^\sigma$,
    \begin{equation}\label{eq:cleaning limiting low jet}
        J_{\le2}[\varphi_*]^\even=x\cdot v_{\varphi_*}^\even,
        \qquad v_{\varphi_*}^\even\in E.
    \end{equation}
    
 Comparing now the quadratic spatial coefficients in
    \eqref{eq:cleaning defining function comparison} gives
    \begin{equation}\label{eq:cleaning full Hessian comparison}
        \widehat O_k^TD_x^2\sA^\even(\widehat z_k)\widehat O_k
        -D_x^2\widehat\sA_{k,z}^\even(0,0)=O(\eps_k).
    \end{equation}
    Set $\mathsf H_k\coloneqq D_{y'}^2\widehat\sA_{k,z}^\even(0,0)$.
    We now look for functions $(y,t)\mapsto \Xi(y,t)\in F$ and $(y,t)\mapsto f(y,t)\in\R$, with $y\in E$,
    satisfying
    \begin{equation}\label{eq:cleaning equations for Xi and f}
        (\nabla_F\sA^\even,\sA^\odd)
        (y+\Xi(y,t)+f(y,t)\bm e_n,t)=0.
    \end{equation}
    The Jacobian with respect to $(\Xi,f)$ at
    $(y,t,\Xi,f)=(0,0,0,0)$ is
    \[
        \begin{pmatrix}\mathsf H|_F&0\\0&1\end{pmatrix},
    \]
    since $\nabla_F\partial_n\sA^\even(0,0)=0$ and
    $\nabla_F\sA^\odd(0,0)=0$.
    As $\mathsf H|_F$ is invertible, the analytic implicit function
    theorem gives unique analytic functions $\Xi$ and $f$ near $(0,0)$
    satisfying \eqref{eq:cleaning equations for Xi and f}, with
    $\Xi(0,0)=f(0,0)=0$. Differentiating
    \eqref{eq:cleaning equations for Xi and f} at $(0,0)$ also gives
    $D_y\Xi(0,0)=D_yf(0,0)=0$, since $E=\Ker\mathsf H$.
    Set
    \begin{equation}\label{eq:cleaning definition of G}
        G(y,t)\coloneqq
        \sA^\even(y+\Xi(y,t)+f(y,t)\bm e_n,t),
        \qquad \widehat y_k\coloneqq\pi_E\widehat x_k.
    \end{equation}
    We claim that
    \begin{equation}\label{eq:cleaning projected identities}
        G(\widehat y_k,\widehat t_k)=0,\qquad
        \nabla_yG(\widehat y_k,\widehat t_k)
        +\eps_kv_{\varphi_*}^\even=o(\eps_k),\qquad
        D_y^2G(\widehat y_k,\widehat t_k)=O(\eps_k).
    \end{equation}
    Indeed, by \eqref{eq:cleaning common gauge} and uniqueness in
    \eqref{eq:cleaning equations for Xi and f},
    \[
        \widehat x_k=\widehat y_k+\Xi(\widehat y_k,\widehat t_k)
        +f(\widehat y_k,\widehat t_k)\bm e_n,
    \]
    thus $\sA^\even(\widehat z_k)=0$ gives
    $G(\widehat y_k,\widehat t_k)=0$. 
    
Next,
    $\pi_F\nabla_x\sA^\even(\widehat z_k)=0$ gives
    \[
        \nabla_yG(\widehat y_k,\widehat t_k)
        =\pi_E\nabla_x\sA^\even(\widehat z_k)
        +\partial_n\sA^\even(\widehat z_k)
            \nabla_yf(\widehat y_k,\widehat t_k).
    \]
    Using \eqref{eq:cleaning low jet} (which in particular gives $\partial_n\sA^\even(\widehat z_k)=o(\eps_k)$), we obtain the second identity in
    \eqref{eq:cleaning projected identities}.

    Finally, to prove the Hessian estimate in \eqref{eq:cleaning projected identities},
    use the basis $(X_j)$ of $\widehat\nu_k^\perp$ consisting of
    \[
        \bm e_i+\partial_{y_i}\Xi(\widehat y_k,\widehat t_k)
        +\partial_{y_i}f(\widehat y_k,\widehat t_k)\bm e_n,\quad 1\le i\le m,
        \qquad
        \pi_{\widehat\nu_k^\perp}\bm e_j,\quad m<j<n.
    \]
    If $X$ is one of the first $m$ vectors and $m<j<n$,
    differentiating \eqref{eq:cleaning equations for Xi and f} gives
    $
        X\cdot D_x^2\sA^\even(\widehat z_k)\bm e_j=0,
    $
    and since $D_{y'}\partial_n\widehat\sA_{k,z}^\even(0,0)=0$,
    \eqref{eq:cleaning full Hessian comparison} yields
    \[
        X\cdot\widehat O_kD_x^2\widehat\sA_{k,z}^\even(0,0)
        \widehat O_k^T\pi_{\widehat\nu_k^\perp}\bm e_j=O(\eps_k).
    \]
    Differentiating \eqref{eq:cleaning definition of G} for $1\le i, j\le m$ gives
    \[
      \partial_{y_iy_j}G(\widehat y_k,\widehat t_k) = X_i^T\cdot D^2_x\sA^\even(\widehat z_k) X_j +\nabla_F\sA^\even(\widehat z_k)\cdot\partial_{y_i y_j}\Xi(\widehat y_k,\widehat t_k) + \partial_n\sA^\even(\widehat z_k) \partial_{y_iy_j}f(\widehat y_k,\widehat t_k).
    \]
    The second term vanishes due to \eqref{eq:cleaning common gauge}, while the third one is $o(\eps_k)$ thanks to analyticity of $f$ and \eqref{eq:cleaning low jet}.
    To bound the first term, in the basis obtained by applying $\widehat O_k^T$ to the vectors $(X_i)$,
    the matrix of the quadratic form $\mathsf H_k$ is
    \[
        \begin{pmatrix}A_k&B_k\\B_k^T&D_k\end{pmatrix},
        \qquad\text{where}\quad B_k=O(\eps_k),\quad D_k\longrightarrow\mathsf H|_F.
    \]
    Since $\widehat z_k\in\Sigma_{n-1}^m(\widehat u_k)$ and $D_k$ is
    invertible for  $k \gg 1$, block elimination gives
    \[
        \operatorname{rank}(A_k-B_kD_k^{-1}B_k^T)
        =\operatorname{rank}\mathsf H_k-\operatorname{rank}D_k
        =(n-1-m)-(n-1-m)=0,
    \]
    hence
    \[
        A_k=B_kD_k^{-1}B_k^T=O(\eps_k^2).
    \]
    Combining this with \eqref{eq:cleaning full Hessian comparison} gives $X_i^T\cdot D^2_x\sA^\even(\widehat z_k) X_j = O(\eps_k)$, thus proving \eqref{eq:cleaning projected identities}.

    \smallskip\noindent\textbf{Step 4.}
    By \eqref{eq:cleaning definition of G} and
    \eqref{eq:expansion cubic points transversality}, we have
    \[
        G(0,0)=0,\qquad \nabla_yG(0,0)=D_y^2G(0,0)=0,\qquad
        \partial_tG(0,0)=\partial_t\sA^\even(0,0)=6a>0.
    \]
   Also, by a Taylor expansion of 
    \eqref{eq:cleaning projected identities}  at $t=0$ we get
        \[
        |\nabla_yG(\widehat y_k,0)|\le C(\eps_k+|\widehat t_k|),
        \qquad
        0=G(\widehat y_k,0)+6a\widehat t_k+o(|\widehat t_k|).
    \]
    Thus, since 
    $G(y,0)=o(|\nabla_yG(y,0)|)$ (by \Cref{lem:lojasiewicz}), we get $|\widehat t_k|=o(\eps_k+|\widehat t_k|)$, and therefore
    $\widehat t_k=o(\eps_k)$. Thus, a Taylor expansion of \eqref{eq:cleaning projected identities} gives
    \begin{equation}\label{eq:cleaning time-zero identities}
        \nabla_yG(\widehat y_k,0)
        +\eps_kv_{\varphi_*}^\even=o(\eps_k),\qquad
        D_y^2G(\widehat y_k,0)=O(\eps_k).
    \end{equation}
    Now applying \Cref{lem:lojasiewicz} to $|\nabla_yG(\cdot,0)|^2$ and using the second equation in \eqref{eq:cleaning time-zero identities} gives
    \[
      |\nabla_yG(\hat y_k,0)| = o\left(\|D_y^2G(\widehat y_k,0)\|\right) = o(\eps_k).
    \]
   Thus the first equation in \eqref{eq:cleaning time-zero identities} yields $v_{\varphi_*}^\even = 0$ and therefore $J_{\le2}[\varphi_*]^\even=0$ by \eqref{eq:cleaning limiting low jet}. This contradicts the transversality of $\varphi_*$ and proves \eqref{eq:cleaning epsilon distance}.
\end{proof}

\section{Construction of transversal perturbations}
\label{section:construction transverse pertubations}

In this section, we fix a bounded smooth domain $\Omega\subset\R^n$ and a solution $u$ of \eqref{eq:stefan}--\eqref{eq:stefan-initial} satisfying \eqref{eq:boundary conditions u}.
Fix $m\in\{1,\dots,n-1\}$. We use the terminology of \Cref{definition:substratification} and write
\[
    \Gamma_+(u)\coloneqq\{(x,t)\in\partial\Omega\times(0,T):u(x,t)>0\}.
\]

We show that finitely many smooth boundary perturbations give transversality
near every point of a compact set, and that the same perturbations work
for nearby solutions.
\begin{theorem}\label{theorem:construction transversal perturbation}
    Let $K\ssubset\Omega\times(0,T)$ be compact, and let $u$ solve \eqref{eq:stefan}--\eqref{eq:stefan-initial} and satisfy \eqref{eq:boundary conditions u}. Suppose further that $u$ is
    $m$-singular in a neighborhood of $K$.
    Then there are $\delta=\delta(u,K)>0$ and finitely many boundary data
    \[
        g_1,\dots,g_N\in C_c^\infty(\Gamma_+(u))
    \]
    such that, for every $(x_0,t_0)\in K$, one of the solutions
    $\varphi_{g_i,u}$ of \eqref{eq:linearized problem} is transversal
    at every point of $\Sigma_{n-1}^m(u)$ inside
    $
        B_\delta(x_0)\times(t_0-\delta,t_0+\delta).
    $
    Moreover, suppose that $u_k,u$ and each $g_i$ satisfy the
    hypotheses of
    \Cref{lem:controlled background linearized compactness}.
    Then the same $\delta$ and $g_1,\dots,g_N$ satisfy the conclusion
    with $u$ replaced by $u_k$, for all sufficiently large $k$.
\end{theorem}

To prove the theorem, we construct a transversal solution of
\eqref{eq:linearized problem} on a space-time cylinder satisfying
the hypotheses of \Cref{prop:runge}, and then apply that proposition
to approximate it by solutions induced by boundary data.

Let $(0,0)$ be a singular point in the maximal stratum of $u$ with blow-up $\tfrac12(x\cdot\nu)^2$. It follows from \Cref{theorem:splitting at singular points} that, for some $r_\circ>0$, we can write
\[
    u = u^+ + u^-\quad\text{in }\Omega^\beta_1\cap\sC_{r_\circ},
\]
and the functions $u^\pm$ solve \eqref{eq:stefan} in their positive phases and satisfy
\[
    |u^\pm-\tfrac12(x\cdot\nu)_\pm^2|
        \le C|(x,t)|_\pb^3.
\]
By \Cref{corollary:uniform smooth expansion maximal stratum}, the polynomials
$\tfrac12(x\cdot\nu)^2$ depend continuously on the singular point, so
$\nu$, and consequently the labeling of $u^\pm$, can be chosen
continuously in a neighborhood of that point.

We first show that near the maximal stratum the solution is convex in
the direction normal to its quadratic blow-up, uniformly on compact
subsets of that stratum.
\begin{lemma}\label{lem:convexity at singular points}
    Let $K_0\ssubset\Sigma_{n-1}(u)$ be compact. There is
    $r_0=r_0(u,K_0)>0$ such that, if $z\in K_0$ has quadratic
    blow-up $\tfrac12(x\cdot\nu)^2$, then
    \[
        \partial_{\nu\nu}u\ge0\quad\text{in }\sC_{r_0}(z).
    \]
\end{lemma}
\begin{proof}
    The uniform quadratic expansion at singular points and
    \cite[Proposition 3.3]{FRS} imply that, after rescaling at $z\in K_0$,
    the hypotheses of \cite[Lemma 3.8]{FRS} hold at a radius independent
    of $z$. Applying that lemma proves the assertion.
\end{proof}

We next show that the closure of each component of $\{u(\cdot,t)>0\}$
meets $\partial\Omega$ where $u>0$, and that positivity of nonnegative
caloric functions propagates through the component from earlier times.
\begin{lemma}\label{lem:causal accessibility}
    Let $t\in(0,T)$ and let $\mathcal C$ be a connected component
    of $\{u(\cdot,t)>0\}$. Then
    $\overline{\mathcal C}\cap\partial\Omega$ contains a relatively
    open set on which $u(\cdot,t)>0$.
    Moreover, if $w\ge0$ is caloric in $\{u>0\}$ and
    $w(x_*,s)>0$ for some $x_*\in\mathcal C$ and $0<s<t$ with
    $u(x_*,s)>0$, then $w(x,t)>0$ for every $x\in\mathcal C$.
\end{lemma}
\begin{proof}
    If the first assertion were false, then $u(\cdot,t)=0$ on $\partial\mathcal C$, whereas
    \[
        \Delta u(\cdot,t)=1+\partial_tu(\cdot,t)\ge1
        \quad\text{in }\mathcal C,
    \]
    contradicting the maximum principle. Continuity and the smoothness
    of $\partial\Omega$ give the relatively open set.

    Next, since $\partial_tu\ge0$, the vertical segment from $(x_*,s)$
    to $(x_*,t)$ lies in $\{u>0\}$, and the parabolic strong maximum
    principle propagates positivity of $w$ along this segment.
    Moreover, by continuity of $u$, a compact path in $\mathcal C$ joining
    $x_*$ to $x$ remains in the positive phase for times sufficiently close
    to $t$ from below, and the parabolic strong maximum principle then gives
    $w(x,t)>0$.
\end{proof}

We will need to control some errors introduced when comparing solutions defined only on cusps $\Omega^\beta_R$ as in \eqref{eq:supercritical paraboloids} to a whole space-time cylinder. This is the purpose of the following estimates.

We suppose that $(0,0)$ is a singular point for $u$, which satisfies
\[
  u(rx,r^2t) = \frac{r^2}2x_n^2 + o(r^2).
\]
Following the notation at the beginning of \Cref{section:cleaning}, fix $\sigma\in\{+,-\}$ and consider $u_\sharp^\sigma$ defined in \eqref{eq:reflected branches}.
Write the free boundary of $u_\sharp^\sigma$ as $x_n=\gamma(x',t)$, and set
\begin{equation}\label{eq:Ft}
    F_t(y)\coloneqq(y',y_n+\gamma(y',t)),\qquad H_+\coloneqq\{y_n>0\}.
\end{equation}
Given a smooth function $\varphi$ in the positive phase of
$u_\sharp^\sigma$, set $q(y,t)\coloneqq\varphi(F_t(y),t)$. Then a direct computation gives 
\[
    (\Delta_x\varphi)\circ F_t
       =\nabla_y\cdot(DF_t^{-1}DF_t^{-T}\nabla_yq),
    \qquad
    (\partial_t\varphi)\circ F_t
       =\partial_tq-\partial_t\gamma\,\partial_nq,
\]
so $\Heat\varphi=0$ is equivalent to $\mathcal Lq=0$, where
\begin{equation}\label{eq:flattened linearized operator}
    \mathcal Lq\coloneqq \partial_tq-\nabla\cdot(A\nabla q)+b\cdot\nabla q,
    \qquad
    A\coloneqq DF_t^{-1}DF_t^{-T},\qquad b\coloneqq-\partial_t\gamma\,\bm e_n.
\end{equation}
We also note that the adjoint operator is given by (recall that $A$ is symmetric)
\[
    \mathcal L^*v=-\partial_tv-\nabla\cdot(A\nabla v)-\nabla\cdot(bv).
\]
By \Cref{lem:smoothness close maximal stratum} with $M=7$,
after a fixed parabolic rescaling the local graph function
admits a $C^{14,7}$ extension to $\R^{n-1}\times\R$.
Multiplying this extension by a smooth cutoff equal to one
near the origin, we obtain a compactly supported extension,
which we still denote by $\gamma$.
The resulting coefficients satisfy the following bounds on
$H_+\times I_0$, where $I_0$ is a fixed bounded open interval
containing $0$ and all $L^\infty$ norms are taken over
$H_+\times I_0$:
\begin{equation}\label{eq:flattened coefficient class}
    \lambda|\zeta|^2\le \zeta^TA\zeta\le\lambda^{-1}|\zeta|^2,
    \qquad
    \sum_{|a|\le6,\ 0\le j\le3}
       \bigl(\|D^a\partial_t^jA\|_\infty
             +\|D^a\partial_t^jb\|_\infty\bigr)\le\Lambda,
\end{equation}
where $\lambda,\Lambda>0$ depend on the chosen extension of $\gamma$.

We denote by $\mathcal K(z,t;\xi,s)$ the Dirichlet kernel for
$\mathcal L$ in $H_+$: for $s\in I_0$ and $h\in C_c^\infty(H_+)$,
the solution of $\mathcal Lq=0$ with $q(\cdot,s)=h$ and
$q=0$ on $\partial H_+$ is given by
\[
    q(z,t)=\int_{H_+}\mathcal K(z,t;\xi,s)h(\xi)\,d\xi,
    \qquad z\in H_+,\quad s<t,\quad t\in I_0.
\]
The next lemma bounds its derivatives when the two spatial points
are separated by at least $8\sqrt{t-s}$.
\begin{lemma}\label{lem:mixed Dirichlet kernel estimate}
   Let $s,t \in I_0$ with $s<t$, set $d=t-s$, and consider
    $z,\xi\in\overline{H_+}$ with $|z-\xi|\ge8\sqrt d$. Then
    \[
        |D_z^a\partial_t^jD_\xi^\mu\mathcal K(z,t;\xi,s)|
        \le C d^{-(n+|a|+2j+|\mu|)/2}
            \exp\!\left(-c\frac{|z-\xi|^2}{d}\right)
    \]
    for $|a|+2j\le3$ and $|\mu|\le1$, with $c,C>0$ depending
    only on $n,\lambda,\Lambda$, and $|I_0|$.
\end{lemma}
\begin{proof}
    By \cite[Theorem~6.1]{DanersHeatKernel}, $\mathcal K$ satisfies
    a Gaussian bound with an additional factor $e^{\omega(t-s)}$,
    which is bounded on $I_0$. Interior and boundary Schauder estimates,
    applied also to the differentiated equations on cylinders of
    radius $\sqrt d/64$, give the asserted derivatives in $(z,t)$,
    with a factor $d^{-1/2}$ for each parabolic order
    (see \cite[Theorem~2.2 and Remark~6.3]{DongZhangHigherOrderSchauder}).
    Applying the same estimates to the time-reversed adjoint equation
    gives the spatial derivative in $\xi$, while derivatives in the
    normal direction are recovered from the equation.
    Since $|z-\xi|\ge8\sqrt d$, the cylinders used in these estimates
    stay away from the singularity of the kernel, and rescaling
    therefore gives the stated bound.
\end{proof}

We use this estimate to control the terms of orders one and two at
the origin for solutions whose source is supported between two
shrinking spheres.
\begin{corollary}\label{corollary:annular source estimate}
    Given $u$ as in \Cref{theorem:construction transversal perturbation}, let $\beta$ denote the parameter $\beta$ in \Cref{theorem:splitting at singular points}, and fix
    $\theta\in(1/(2+\beta),1/2)$.
    For $\tau>0$ sufficiently small, let
    $Z\in L^2(-\tau,t;H^1_0(H_+))$ for every $t\in(-\tau,0)$
    solve
    \[
        \mathcal LZ=F_0+\nabla\cdot F_1
        \quad\text{in }H_+\times(-\tau,0),
        \qquad Z(\cdot,-\tau)=0.
    \]
    Suppose that
    \[
        \supp(F_0,F_1)\subset
        \{(\xi,s):-\tau<s<0,\ \xi_n>0,
          \ (-s)^\theta/8\le|\xi|\le(-s)^\theta/4\},
    \]
    with
    \[
        \|F_0\|_{L^2(H_+\times(-r,-r/2))}
        +\|F_1\|_{L^2(H_+\times(-r,-r/2))}
        \le r^{-1}\qquad\forall\,r\in(0,\tau].
    \]
    Then there is a unique pair $\alpha_Z\in\R$,
    $v_Z\in\bm e_n^\perp$ such that
    \[
        |\alpha_Z|+|v_Z|
        \le Ce^{-c\tau^{-(1-2\theta)}},
    \]
    and, with the notation of \eqref{eq:ansatz linearized problem},
    \[
        |Z(y,t)-\upvarphi_{u_\sharp^\sigma}[\alpha_Z,v_Z,0](F_t(y),t)|
        \le Ce^{-c\tau^{-(1-2\theta)}}|(y,t)|_\pb^3
    \]
    for $y_n>0$, $-\tau/4<t<0$, and $|y|<(-t)^\theta/16$.
\end{corollary}
\begin{proof}
    The equation and the source bounds give
    $\partial_tZ\in L^2(-\tau,t;H^{-1}(H_+))$ for every $t\in(-\tau,0)$,
    and hence $Z\in C([-\tau,t];L^2(H_+))$ by
    \cite[Section~5.9.2]{Evans}.
    Set $\sigma_j\coloneqq2^{-j-1}\tau$ and write
    $Z=\sum_{j\ge0}U_j$ by Duhamel's formula, where
    \begin{equation}\label{eq:annular Duhamel formula}
        U_j(y,t)\coloneqq
        \int_{(-2\sigma_j,-\sigma_j)\cap(-\infty,t)}\!\int_{H_+}
        \bigl[\mathcal K(y,t;\xi,s)F_0(\xi,s)
        -D_\xi\mathcal K(y,t;\xi,s)\cdot F_1(\xi,s)\bigr]
        \,d\xi\,ds.
    \end{equation}
    For $y_n>0$, $-\tau/4<t<0$, and $|y|<(-t)^\theta/16$,
    on the support of the integrands in
    \eqref{eq:annular Duhamel formula} we have
    \[
        |y-\xi|\ge c\sigma_j^\theta,\qquad 0<t-s\le2\sigma_j.
    \]
    Since $\theta<1/2$, these inequalities imply
    $|y-\xi|\ge8\sqrt{t-s}$ when $\tau$ is small enough, so
    \Cref{lem:mixed Dirichlet kernel estimate} gives, after decreasing $c$ to absorb the polynomial factors,
    \begin{equation}\label{eq:annular one slab estimate}
        |D_y^a\partial_t^\ell U_j(y,t)|
        \le Ce^{-c\sigma_j^{-(1-2\theta)}}
        \qquad\text{for }|a|+2\ell\le3.
    \end{equation}
    Let $P_j$ be the spatial Taylor polynomial of degree two of
    $U_j(\cdot,0)$ at the origin, and put $\rho\coloneqq|(y,t)|_\pb$.
    If $\sigma_j\ge4\rho^2$, the same derivative bounds hold on
    $(B_\rho\cap H_+)\times[-\rho^2,0]$, which is separated from
    the sources, so Taylor's formula and the zero boundary trace give
    \[
        |U_j(y,t)-P_j(y)|
        \le C\rho^3e^{-c\sigma_j^{-(1-2\theta)}}.
    \]
    If $\sigma_j<4\rho^2$, we bound $U_j$ and $P_j$ separately by
    \eqref{eq:annular one slab estimate} and use
    $\rho^{-3}\le8\sigma_j^{-3/2}$ to obtain the same estimate
    after decreasing $c$. Since the estimates above are summable in $j$,
    we can sum the Taylor polynomials and their remainders to obtain
    \[
        Z(y,t)=D_yZ(0,0)\cdot y+\tfrac12y\cdot D_y^2Z(0,0)y
        +O\!\left(e^{-c\tau^{-(1-2\theta)}}\rho^3\right),
    \]
    with both coefficients bounded by
    $Ce^{-c\tau^{-(1-2\theta)}}$.

    Since $Z=0$ on $\{y_n=0\}$, its linear term is
    $\alpha_Zy_n$, where $\alpha_Z=\partial_nZ(0,0)$;
    after subtracting $\alpha_Z(\partial_nu_\sharp^\sigma)\circ F_t$,
    the remaining quadratic term is harmonic, since both functions solve
    $\mathcal Lq=0$ in
    $\{y_n>0,\ |y|<(-t)^\theta/16,\ -\tau/4<t<0\}$ and
    $A(0,0)=Id$. Because it also vanishes on $\{y_n=0\}$, it must be
    equal to $y_n(y'\cdot v_Z)$ for some vector $v_Z\in\bm e_n^\perp$.
    By \eqref{eq:ansatz linearized problem}, this is the quadratic
    term of $\upvarphi_{u_\sharp^\sigma}[0,v_Z,0]\circ F_t$,
    which proves both the estimate and uniqueness.
\end{proof}

We can now locally construct transversal perturbations on domains to which \Cref{prop:runge} applies.
\begin{lemma}\label{lem:local construction}
    Let $z_0=(x_0,t_0)\in\Sigma_{n-1}^m(u)$ and let
    $\mathcal Q\ssubset\Omega\times(0,T)$ be a cylinder containing $z_0$.
    There are a smooth domain $D'\ni x_0$ and $t_-<t_0$, with
    $\overline D\ssubset\mathcal Q$ for $D\coloneqq D'\times(t_-,t_0)$,
    and a function $w\in H^1(D)$ solving \eqref{eq:linearized problem}
    in $D$ and transversal at $z_0$, such that $D$ satisfies
    assumptions \eqref{assumption:a1} and \eqref{assumption:a2} of
    \Cref{prop:runge}.
\end{lemma}
\begin{proof}
    After translation and rotation, assume that $z_0=(0,0)$ and
    its quadratic blow-up is $\tfrac12x_n^2$, and let $E,F$ be as in
    \Cref{definition:transversality}.
    In the next three steps we first construct a transversal solution $W$ in the region
    where \Cref{theorem:splitting at singular points} gives $u=u^++u^-$, we then use $W$ to construct a solution $\Psi$ on a
    cylinder and prove its transversality by comparing its expansion
    coefficients with those of $W$, and finally we choose a
    smaller cylinder $D$ satisfying the assumptions of \Cref{prop:runge}
    and set $w=\Psi|_D$.

    \smallskip\noindent\textbf{Step 1.}
    Choose $e\in E$ with $|e|=1$ and set
    \[
        W^\pm\coloneqq\pm\bigl[(x\cdot e)\partial_nu^\pm
                     -x_n\partial_eu^\pm\bigr],\qquad W=W^++W^-.
    \]
    Since spatial rotations commute with $\Heat$,
    $\Heat W^\pm=0$ in $\{u^\pm>0\}$, while
    \Cref{corollary:uniform smooth expansion maximal stratum}
    and local parabolic estimates give
    \[
        |W^\pm|\le C|(x,t)|_\pb^2,\qquad
        |\nabla_xW^\pm|+|\partial_tW^\pm|\le C|(x,t)|_\pb,
    \]
    and since $W^\pm=0$ on $\partial\{u^\pm>0\}$, their zero extensions
    belong to $H^1$ on smaller cusps.
    
    Note now that, by \eqref{eq:ansatz linearized problem} and
    \eqref{eq:lower branch ansatz dictionary}, their triples in the
    original coordinates are $(0,e,0)$ and $(0,-e,0)$, so
    $
        J_{\le2}[W]^\even=x\cdot e.
    $
    Also, by \eqref{eq:expansion cubic points transversality} and
    \Cref{lem:jets rototranslations},
    \[
        J_1[\partial_tu]^\even=6a>0,\qquad
        \nabla_xJ_2[v\cdot\nabla u]^\even=\mathsf Hv\in F
        \quad\text{for }v\in F,
    \]
    and therefore \Cref{definition:transversality} gives $\mu(W)>0$.

    \smallskip\noindent\textbf{Step 2.}
    Let $\beta$ be given by \Cref{theorem:splitting at singular points}
    and fix
    \[
        \frac1{2+\beta}<\theta<\frac12,
        \qquad R\coloneqq\tau^\theta.
    \]
    Take $\tau>0$ sufficiently small that
    $B_R\times[-\tau,0]\Subset\mathcal Q$ and lies in the cylinder
    supplied by \Cref{lem:convexity at singular points}, and let $f\in H^1_0(B_R\cap\{u(\cdot,-\tau)>0\})$
    be the product of $W(\cdot,-\tau)$ with a smooth cutoff
    equal to one on $B_{R/2}$ and supported in $B_{3R/4}$.
    
    Let $\Psi\in H^1(B_R\times(-\tau,0))$ satisfy
    \[
        \Heat\Psi=0\quad\text{in }\{u>0\}\cap(B_R\times(-\tau,0)),
        \qquad \Psi(\cdot,-\tau)=f,\qquad
        \Psi=0\quad\text{on }\partial B_R\times(-\tau,0),
    \]
    given by \Cref{remark:existence linear variant}.

    Let $F_t^\pm$ be the flattening maps for $u^+$ and the inward
    reflection of $u^-$, with the reflection included in $F_t^-$ (see \eqref{eq:Ft}),
    and denote the corresponding operators and coefficients in
    \eqref{eq:flattened linearized operator} by
    $\mathcal L^\sigma,A^\sigma,b^\sigma$, for $\sigma=\pm$.
    We then set
    \[
        q^\pm(y,t)\coloneqq(\Psi-W^\pm)(F_t^\pm(y),t),
        \qquad Z^\pm\coloneqq\eta q^\pm,
    \]
    where $\eta$ is a smooth cutoff equal to $1$ for
    $|y|\le(-t)^\theta/8$ and vanishing for $|y|\ge(-t)^\theta/4$,
    chosen so that
    \[
        |\nabla_y\eta|\le C(-t)^{-\theta},\qquad
        |\partial_t\eta|\le C(-t)^{-1}.
    \]
    Since $\theta>1/(2+\beta)$, for $\tau$ small the points
    $(F_t^\pm(y),t)$ with $y_n>0$ and $\eta(y,t)\ne0$
    lie in the cusp where $u^\pm$ are defined by
    \Cref{theorem:splitting at singular points}, and at $t=-\tau$
    their spatial coordinates lie in $B_{R/2}$, where $\Psi=W$;
    hence, after extension by zero to
    $H_+\times(-\tau,0)$, the functions $Z^\pm$ have zero initial
    and Dirichlet data and satisfy
    \[
        \mathcal L^\sigma Z^\sigma=F_0^\sigma+\nabla\cdot F_1^\sigma,
        \qquad \sigma\in\{+,-\},
    \]
    where
    \[
        F_0^\sigma=(\partial_t\eta+b^\sigma\cdot\nabla\eta)q^\sigma
                    -(A^\sigma\nabla q^\sigma)\cdot\nabla\eta,\qquad
        F_1^\sigma=-q^\sigma A^\sigma\nabla\eta.
    \]
    The sources $F_0^\sigma,F_1^\sigma$ are supported where
    $(-t)^\theta/8\le|y|\le(-t)^\theta/4$, and the energy estimates
    for $\Psi$ (see \Cref{remark:existence linear variant}) and the bounds for $W^\pm$ give
    \[
        \|F_0^\sigma\|_{L^2(H_+\times(-s,-s/2))}
        +\|F_1^\sigma\|_{L^2(H_+\times(-s,-s/2))}
        \le Cs^{-1/2} \leq s^{-1}\qquad\forall\,s\in(0,\tau].
    \]
    Let $\alpha_\Psi^\pm,v_\Psi^\pm$ be the coefficients of $\Psi$
    in \eqref{eq:assumption perturbation}, expressed in the original
    coordinates using \eqref{eq:lower branch ansatz dictionary};
    since $\eta=1$ for $|y|<(-t)^\theta/16$, comparison of the
    terms of orders one and two in
    \Cref{corollary:annular source estimate} and \Cref{prop:C3beta expansion linear}
    gives
    \[
        \sum_{\sigma=\pm}
        \left(|\alpha_\Psi^\sigma|
             +|v_\Psi^\sigma-\sigma e|\right)
        \le Ce^{-c\tau^{-(1-2\theta)}},
    \]
    and therefore \eqref{eq:quantitative transversality gaps} implies,
    for $\tau$ sufficiently small,
    \[
        \mu(\Psi)\ge\tfrac12\mu(W)>0.
    \]

    \smallskip\noindent\textbf{Step 3.}
    Fix $h\in(0,1/4)$ and set $r\coloneqq cR$, with $c>0$ small.
    Since $\tau/r^2\to0$ as $\tau\to0$, the quadratic expansion at
    $(0,0)$ gives, for $\tau$ sufficiently small,
    \[
        \{u=0\}\cap(B_{5r}\times[-\tau,0])
            \subset\{|x_n|<hr/4\},\qquad
        u>0\quad\text{on }\{|x_n|\ge hr\}
            \cap(B_{5r}\times[-\tau,0]).
    \]
    Let $D'$ be a smooth convex domain obtained by smoothing the edges
    of the cylinder $\{|x'|<r,\ |x_n|<2hr\}$ inwards, while leaving it
    unchanged where $|x_n|\le hr$, and set $q_\pm\coloneqq\pm4r\bm e_n$.
    For each $t\in[-\tau,0]$, \Cref{lem:convexity at singular points}
    implies that $\{u(\cdot,t)=0\}\cap B_{5r}$ intersects every
    vertical line $(x',0)+\R\bm e_n$ in an interval. Hence, by the choice of $D'$, any
    $p\in\partial D'\cap\{u(\cdot,t)>0\}$ can be joined to
    $q_+$ or $q_-$ by a path in $\{u(\cdot,t)>0\}$ outside
    $\overline{D'}$, apart from its starting point $p$.

    At $t=-\tau$, the splitting holds in $B_{5r}$, since
    $\theta>1/(2+\beta)$, and each path in
    $\overline{D'}\cap\{u(\cdot,-\tau)>0\}$ with endpoints on
    $\partial D'$ lies entirely in one of the two positivity sets $\{u^\pm(\cdot,-\tau)>0\}$. By the preceding
    observation, both endpoints can be joined to the same point
    $q_+$ or $q_-$ by paths in $\{u(\cdot,-\tau)>0\}$ outside
    $\overline{D'}$, apart from their starting points.
    Thus \Cref{lem:causal accessibility} gives paths $\sigma_\pm$
    joining $q_\pm$ to $\partial\Omega$, with
    \[
        \sigma_\pm\subset
        \{x\in\overline\Omega:u(x,-\tau)>0\}\setminus\overline{D'}.
    \]
    Thus, since $\partial_tu\ge0$, $u(\cdot,t)>0$ on $\sigma_\pm$ for every
    $t\in[-\tau,0]$.

    Fix $t\in[-\tau,0]$ and let $\mathcal C$ be a connected component
    of $\{u(\cdot,t)>0\}\setminus\overline{D'}$. If $\mathcal C$ is
    also a component of $\{u(\cdot,t)>0\}$, then
    \Cref{lem:causal accessibility} applies directly. Otherwise,
    choose $p\in\partial\mathcal C\cap\partial D'$ with $u(p,t)>0$.
    Then the path from $p$ to one of $q_\pm$ lies in $\mathcal C$
    (apart from $p$), so the corresponding path $\sigma_\pm$ also lies
    in $\mathcal C$ (apart from its endpoint on $\partial\Omega$)
    where $u(\cdot,t)>0$. Therefore, in either case,
    $\overline{\mathcal C}\cap\partial\Omega$ contains
    a relatively open set on which $u(\cdot,t)>0$.

    Set $D\coloneqq D'\times(-\tau,0)$ and $w\coloneqq\Psi|_D$.
    Then \eqref{assumption:a1} holds since $D'$ is smooth and
    $\overline D\ssubset\mathcal Q$, while \eqref{assumption:a2}
    follows from the preceding argument for $t\in[-\tau,0]$ and
    from \Cref{lem:causal accessibility} at all other times.
    Finally, by Step~2, $w\in H^1(D)$ solves
    \eqref{eq:linearized problem} and satisfies
    $\mu(w)=\mu(\Psi)>0$, which concludes the proof.
\end{proof}

We can now prove \Cref{theorem:construction transversal perturbation}.
\begin{proof}[Proof of \Cref{theorem:construction transversal perturbation}]
    Choose $K\ssubset U\ssubset\Omega\times(0,T)$ so that $u$ is
    $m$-singular in a neighborhood of $\overline U$.
    At each $z\in K\cap\Sigma_{n-1}^m(u)$, let $D_z,w_z$ be given
    by \Cref{lem:local construction}. By \Cref{prop:runge}, there are
    $g_{z,k}\in C_c^\infty(\Gamma_+(u))$ such that
    \[
        \varphi_{g_{z,k},u}\longrightarrow w_z
        \quad\text{in }L^2(D_z).
    \]
    The coefficient estimate in \Cref{prop:C3beta expansion linear},
    applied to $\varphi_{g_{z,k},u}-w_z$ on smaller cusps in $D_z$, gives
    \[
        \sum_{\sigma=\pm}
        \big\|J_{\le2}[\varphi_{g_{z,k},u}^\sigma-w_z^\sigma]\big\|
        \le C_z\|\varphi_{g_{z,k},u}-w_z\|_{L^2(D_z)}
        \longrightarrow0.
    \]
    Choose $k$ large and write $g_z=g_{z,k}$, so that
    $\varphi_{g_z,u}$ is transversal at $z$.
    By \Cref{lem:transversality open conditions}, we can choose open neighborhoods
    \[
        z\in V_z\ssubset U_z\ssubset U
    \]
    such that $\varphi_{g_z,u}$ is transversal on
    $\Sigma_{n-1}^m(u)\cap\overline{U_z}$.

    By \Cref{lem:stratification closed conditions} and \Cref{corollary:uniform smooth expansion maximal stratum},
    $\Sigma_{n-1}^m(u)$ is closed in $U$, so for each
    $z\in K\setminus\Sigma_{n-1}^m(u)$ we can choose
    $z\in V_z\ssubset U_z\ssubset U$ with
    $\overline{U_z}\cap\Sigma_{n-1}^m(u)=\emptyset$, and define
    $g_z=0$.
    By compactness, finitely many of the sets $V_z$ cover $K$.
    Denote the corresponding neighborhoods and boundary data by
    $V_i,U_i,g_i$, $1\le i\le N$, and choose $\delta>0$ so that
    \[
        B_\delta(x)\times(t-\delta,t+\delta)\subset U_i
        \qquad\forall\,(x,t)\in V_i,\quad 1\le i\le N.
    \]
    This proves the assertion for $u$.

    Finally, by \Cref{lem:controlled background linearized compactness},
    \[
        \varphi_{g_i,u_k}\longrightarrow\varphi_{g_i,u}
        \quad\text{in }L^2(\Omega\times(0,T)),\qquad 1\le i\le N.
    \]
    By \Cref{lem:stratification closed conditions}
     and \Cref{corollary:uniform smooth expansion maximal stratum}, any limit in
    $\overline{U_i}$ of points of $\Sigma_{n-1}^m(u_k)$
    belongs to $\Sigma_{n-1}^m(u)$. Compactness and
    \Cref{lem:transversality open conditions} therefore imply that
    $\varphi_{g_i,u_k}$ is transversal on
    $\Sigma_{n-1}^m(u_k)\cap\overline{U_i}$ for every $i$
    and all sufficiently large $k$; if $g_i=0$, this intersection
    is empty. Thus the same $\delta$ and $g_i$ work for $u_k$.
\end{proof}

\begin{remark}\label{remark:finite boundary perturbations}
    Let $u$ and $g_1,\dots,g_N$ be as in
    \Cref{theorem:construction transversal perturbation}, and write
    $b=u|_{\partial\Omega\times(0,T)}$.
    By \eqref{eq:boundary conditions u}, $b,\partial_tb>0$ on the
    compact set $\bigcup_i\supp g_i\subset\Gamma_+(u)$.
    Hence, for some $\eps_*>0$, the boundary data
    \begin{equation}\label{eq:finite boundary perturbations}
        b+\sum_{i=1}^N a_i g_i,
        \qquad \sum_i|a_i|<\eps_*,
    \end{equation}
    are nonnegative and nondecreasing in time, and agree with $b$
    near $t=0$ and $t=T$, so their solutions can all be extended to a common time
    interval. Moreover, by \eqref{eq:comparison principle}, after
    reducing $\eps_*$, these solutions remain positive in fixed
    neighborhoods of $\bigcup_i\supp g_i$ in
    $\overline\Omega\times[0,T]$.
    These properties will be used in the proof of
    \Cref{theorem:generic regularity}.
\end{remark}

\section{Proof of \texorpdfstring{\Cref{theorem:generic regularity}}{}}
\label{section:main result}

Set $Q\coloneqq\Omega\times(0,T)$ and let
\[
    \mathscr B\coloneqq
    \{b\in C(\partial\Omega\times\R):b=0\text{ for }t\le0,\
       b(x,\cdot)\text{ is nondecreasing}\},
\]
endowed with the topology of locally uniform convergence. For
$b\in\mathscr B$, denote by $\mathcal S(b)$ the unique solution of
\eqref{eq:stefan} on $\Omega\times\R$ with lateral datum $b$ and satisfying
$\mathcal S(b)=0$ for $t\le0$. Then
\[
    \mathscr S\coloneqq
    \{\mathcal S(b)|_{\overline Q}:b\in\mathscr B\},
\]
endowed with the topology inherited from $C^0(\overline Q)$, is precisely
the solution space in \Cref{theorem:generic regularity}; throughout this
section, global solutions are understood to be restricted to $\overline Q$
when regarded as elements of $\mathscr S$.
By the comparison principle,
\[
    \|\mathcal S(b_1)-\mathcal S(b_2)\|_{C^0(\overline Q)}
    =\|b_1-b_2\|_{C^0(\partial\Omega\times[0,T])}.
\]
The restrictions of elements of $\mathscr B$ to
$\partial\Omega\times[0,T]$ are precisely the nonnegative continuous
data that vanish at $t=0$ and are nondecreasing in time, so they form
a closed subset of $C(\partial\Omega\times[0,T])$ and the preceding
identity shows that $\mathscr S$ is complete and therefore Baire.

We use the strata $\Sigma_{n-1}^m$ introduced after
\eqref{eq:expansion cubic points} and the notion of $m$-singular solution
in \Cref{definition:substratification}. To prove the theorem, we shall successively approximate
$m$-singular solutions by $(m-1)$-singular ones and then apply Baire's
theorem. For this purpose, we introduce the class
\[
    \mathscr B_{\mathrm{sm}}\coloneqq
    \left\{b\in\mathscr B\cap C^\infty(\partial\Omega\times\R):
    \begin{array}{l}
        \partial_tb>0\text{ on }\{b>0\},\\
        b=0\text{ for }t\le\tau_b\text{ for some }\tau_b>0
    \end{array}\right\}.
\]
The following lemma collects the properties of these boundary data needed below.

\begin{lemma}\label{lem:approximation extension}
    The class $\mathscr B_{\mathrm{sm}}$ is dense in $\mathscr B$.
    Moreover, for every $b\in\mathscr B_{\mathrm{sm}}$ the solution
    $\mathcal S(b)$ satisfies \eqref{eq:boundary conditions u}.
\end{lemma}
\begin{proof}
    Given $b\in\mathscr B$, let
    $e^{s\Delta_{\partial\Omega}}$ denote the heat semigroup on
    $\partial\Omega$ generated by its Laplace--Beltrami operator.
    Choose $\delta_j\downarrow0$, a nonnegative mollifier
    $\rho_{\delta_j}$ supported in $(0,\delta_j)$, and a smooth
    function $h$ with $h=0$ on $(-\infty,0]$ and $h'>0$ on $(0,\infty)$,
    and set
    \[
        b_j(t)\coloneqq
        \int_0^{\delta_j}\rho_{\delta_j}(s)
          e^{\delta_j\Delta_{\partial\Omega}}
          b(t-3\delta_j-s)\,ds
        +\delta_jh(t-2\delta_j).
    \]
    This formula defines smooth, nonnegative, nondecreasing data for all
    $t\in\R$, with $b_j=0$ for $t\le2\delta_j$ and
    $\partial_tb_j>0$ on $\{b_j>0\}$, so $b_j\in\mathscr B_{\mathrm{sm}}$.
    Since $b_j\to b$ locally uniformly on $\partial\Omega\times\R$,
    this proves density. Finally, for any $b\in\mathscr B_{\mathrm{sm}}$, standard
    regularity for \eqref{eq:stefan}, together with the vanishing of $b$
    near $t=0$ and $\partial_tb>0$ on $\{b>0\}$, gives
    \eqref{eq:boundary conditions u} for $\mathcal S(b)$.
\end{proof}

By \Cref{lem:approximation extension}, solutions with boundary data in $\mathscr B_{\mathrm{sm}}$ are dense in $\mathscr S$. Since they satisfy \eqref{eq:boundary conditions u}, all results from \Cref{section:linearized problem,section:construction transverse pertubations} apply.

For $b\in\mathscr B_{\mathrm{sm}}$ fixed, let $u=\mathcal S(b)$ and recall the notation
\[
    \Gamma_+(u)=\{(x,t)\in\partial\Omega\times(0,T):u(x,t)>0\}.
\]
We fix finitely many $g_i\in C_c^\infty(\Gamma_+(u))$ and extend each $g_i$ by zero outside $(0,T)$.
As in \Cref{remark:finite boundary perturbations}, for sufficiently small coefficients the data $b+\sum_i a_i g_i$ remain in $\mathscr B_{\mathrm{sm}}$, and their solutions remain positive in fixed neighborhoods of $\bigcup_i\supp g_i$.
For each such solution $v$, we also extend by zero to negative times the linearized solutions $\varphi_{g_i,v}$ of \eqref{eq:linearized problem} with boundary data $g_i$, given by \Cref{lem:existence linear}.
On every bounded time interval, comparison gives uniform convergence when the coefficients converge, while comparison and boundary and interior estimates give uniform $C^1_{x,t}(\overline Q)$ bounds and a uniform positive lower bound for the time derivatives on these supports, so \Cref{lem:controlled background linearized compactness} gives $L^2(Q)$-continuity of $\varphi_{g_i,v}$ with respect to the coefficients.

The following proposition will allow us to show that the parameters for which
$\Sigma_{n-1}^m$ persists form a set of measure zero.
\begin{proposition}\label{prop:gmt}
    Let $(X,d)$ be a complete separable metric space, let
    $E\subset X\times(-1,1)$ be Borel, and denote by
    \[
        \pi_1\colon X\times(-1,1)\to X,\qquad
        \pi_2\colon X\times(-1,1)\to(-1,1)
    \]
    the coordinate projections. Assume that $0<\beta\le s$, that
    \[
        \cH_d^\beta(\pi_1(E))<+\infty,
    \]
    and that, for every $(z_0,\eps_0)\in E$, there is a modulus of continuity $\omega_{z_0,\eps_0}$ such that
    \[
        \{(z,\eps):\eps-\eps_0>
        \omega_{z_0,\eps_0}(d(z,z_0))d(z,z_0)^s\}\cap E=\emptyset.
    \]
    Then $\cH^{\beta/s}(\pi_2(E))=0$.
\end{proposition}
\begin{proof}
    The proof of \cite[Proposition 7.7(a)]{FRSObst} applies with metric
    balls in $(X,d)$, using, for each $\delta>0$, a Borel partition of
    $E$ on which the one-sided estimate holds with constant $\delta$
    below a fixed radius.
\end{proof}
We shall apply this proposition with $X=\R^{n+1}$ equipped with the
parabolic metric $\dist_\pb$ introduced at the beginning of
\Cref{section:preliminaries}, writing $\cH_\pb^\beta$ and $\dim_\pb$ for the
corresponding Hausdorff measure and dimension. 

We next use time translations to obtain strictly monotone families of
solutions while preserving transversality in the sense of
\Cref{definition:transversality} (recall also
\eqref{eq:quantitative transversality gaps}).

We shall call a continuous
family $(u_\eps)_{\eps\in I}\subset\mathscr S$, with $I\subset\R$ an
interval, \emph{strictly monotone} if, writing
$b_\eps=u_\eps|_{\partial\Omega\times(0,T)}$, for $\eps_1<\eps_2$ one has
\[
    b_{\eps_1}\le b_{\eps_2},\qquad\text{and}\qquad
    b_{\eps_1}<b_{\eps_2}\quad\text{on }\{b_{\eps_1}>0\}.
\]
We shall use the time translations
\[
    T_\sigma f(x,t)\coloneqq f(x,t+\sigma),\qquad \sigma\in\R.
\]

\begin{proposition}\label{prop:delayed ordered family}
    Let $b\in\mathscr B_{\mathrm{sm}}$ and $u=\mathcal S(b)$. Given
    $g\in C_c^\infty(\Gamma_+(u))$, there are $C,s_0>0$ such that,
    setting $v_s\coloneqq\mathcal S(b+sg)$, the family
    \[
        u_s\coloneqq T_{Cs}v_s,\qquad -s_0\le s\le0,
    \]
    is strictly monotone and vanishes at negative times. Moreover,
    $s\mapsto u_s$ is differentiable in $L^2(Q)$, with
    \[
        \partial_su_s=
        T_{Cs}\bigl(\varphi_{g,v_s}+C\partial_tv_s\bigr).
    \]
    Finally, if $z\in Q$ and $y\coloneqq z+(0,Cs)\in Q\cap\Sigma_{n-1}^m(v_s)$
    for some $1\le m\le n-1$, then $\partial_su_s$ is transversal to
    $u_s$ at $z$ if and only if $\varphi_{g,v_s}$ is transversal to
    $v_s$ at $y$.
\end{proposition}
\begin{proof}
    Choose $C$ sufficiently large so that $C\partial_tb>2|g|$ on $\supp g$, and then
    choose $s_0>0$ small enough that $b+sg\in\mathscr B_{\mathrm{sm}}$
    for $-s_0\le s\le0$ and
    \[
        \partial_s\bigl[T_{Cs}(b+sg)\bigr]
        =T_{Cs}\bigl[g+C(\partial_tb+s\partial_tg)\bigr]>0
        \quad\text{where }T_{Cs}(b+sg)>0.
    \]
    Since the derivative vanishes where the trace is zero, integrating
    in $s$ and applying comparison proves strict monotonicity; also, since
    $s\le0$, the time translation preserves the zero initial
    and negative-time values.
    Next, the derivative formula follows from \Cref{lem:existence linear}
    and the chain rule in $L^2$, while
    \eqref{eq:quantitative transversality gaps} shows that
    $\varphi_{g,v_s}+C\partial_tv_s$ and $\varphi_{g,v_s}$ are
    transversal simultaneously at $y$.
    Translating both the background solution and the perturbation by
    $T_{Cs}$ then gives the last assertion at $z=y-(0,Cs)$.
\end{proof}

We next show that the maximal strata of solutions in a strictly
monotone family are disjoint for distinct parameter values, so that
the parameter is a well-defined function on their union.
\begin{lemma}\label{lem:property monotonicity}
    Let $u_\eps$ be a strictly monotone family of solutions on $Q$.
    If $(x_0,t_0)\in Q\cap\Sigma_{n-1}(u_{\eps_1})\cap
    \Sigma_{n-1}(u_{\eps_2})$, then $\eps_1=\eps_2$.
\end{lemma}
\begin{proof}
    Suppose $\eps_1<\eps_2$ and set
    $w\coloneqq u_{\eps_2}-u_{\eps_1}\ge0$, so that
    $\Heat w=0$ in $\{u_{\eps_1}>0\}$.
    Since every positive-phase component meets the positive lateral
    boundary by \Cref{lem:causal accessibility}, strict boundary
    ordering and the parabolic strong maximum principle give
    $w>0$ in $\{u_{\eps_1}>0\}$.
    
    After a translation and a rotation, assume that $(x_0,t_0)=(0,0)$
    and the blow-up of $u_{\eps_1}$ is $x_n^2/2$.
    The caloric barrier constructed in the proof of
    \cite[Lemma 5.8(b)]{FRS}, applied to $w$, then gives
    \[
        w(r\bm e_n,-r^2)\ge cr^\alpha
        \qquad\text{for some }\alpha<2,
    \]
    contradicting $w=O(r^2)$ in $\sC_{2r}$ (which follows from the
    quadratic expansions of both solutions at the common singular point).
\end{proof}

\begin{proof}[Proof of \Cref{theorem:generic regularity}]
    We shall prove that, for every compact $K\ssubset Q$, the
    $0$-singular solutions on $K$ form an open dense subset of
    $\mathscr S$. Indeed, applying Baire's theorem along a compact
    exhaustion of $Q$ then gives a residual set of solutions satisfying
    \[
        \Sigma_{n-1}(u)=\Sigma_{n-1}^0(u),
    \]
    and this maximal stratum is discrete by
    \Cref{lem:stratification maximal stratum}.

    By \Cref{lem:stratification closed conditions}, $m$-singularity
    on $K$ extends to a neighborhood of $K$ and is an open condition
    in $\mathscr S$, so only density remains to be proved.
    Since every solution is $(n-1)$-singular, we can establish density
    by successively approximating $m$-singular solutions on $K$ by
    $(m-1)$-singular ones, for $m=n-1,\ldots,1$; the last step,
    $m=1$, removes $\Sigma_{n-1}^1$ and leaves only
    $\Sigma_{n-1}^0$ in the maximal stratum on $K$.

  So, we fix $K\ssubset Q$ and $1\le m\le n-1$, and let $u$ be
    $m$-singular on $K$. By \Cref{lem:approximation extension} and
    openness, we may assume that $u=\mathcal S(b)$ with
    $b\in\mathscr B_{\mathrm{sm}}$.
    Then \Cref{theorem:construction transversal perturbation} gives compact sets $K_1,\ldots,K_N\subset K$ covering
    $K$ and data $g_i\in C_c^\infty(\Gamma_+(u))$ such that, for every
    solution $v$ obtained by a sufficiently small boundary perturbation of the form
    \eqref{eq:finite boundary perturbations}, $v$ remains $m$-singular
    near $K$ and $\varphi_{g_i,v}$ is transversal on
    $K_i\cap\Sigma_{n-1}^m(v)$ for every $i$.

    We perturb successively in these directions. At the $i$-th step,
    write $\widehat b$ for the current datum and set
    \[
        v_s\coloneqq\mathcal S(\widehat b+sg_i),\qquad
        u_s\coloneqq T_{As}v_s,\qquad -s_0\le s\le0,
    \]
    as in \Cref{prop:delayed ordered family}, with $s_0<1$ small
    enough that $K_i-(0,As)\Subset Q$ for every $s\in[-s_0,0]$.
    Consider
    \[
        E\coloneqq
        \{(z,s):s\in[-s_0,0],\
          z+(0,As)\in K_i,\
          z\in\Sigma_{n-1}^m(u_s)\}.
    \]
    By \Cref{lem:stratification closed conditions}, the continuity
    of the cubic coefficients in
    \Cref{corollary:uniform smooth expansion maximal stratum}, and the fact that $u_s$ are $m$-singular, $E$ is
    compact, while \Cref{lem:property monotonicity} makes $\pi_1$
    injective on $E$. Thus $\pi_1$ is a homeomorphism onto its image,
    so $s$ is a continuous function of $z$ on $\pi_1(E)$.

    Fix $(z_*,s_*)\in E$. For $(z,s)\in E$ with $s\ne s_*$ and
    $z\to z_*$, we have $s\to s_*$, and the differentiability in
    \Cref{prop:delayed ordered family} gives
    \[
        \frac{u_s-u_{s_*}}{s-s_*}\longrightarrow\partial_su_{s_*}
        \qquad\text{strongly in }L^2(Q).
    \]
    Also, this limit is transversal at $z_*$ by the same proposition and satisfies
    \eqref{eq:assumption perturbation} by
    \Cref{prop:C3beta expansion linear}.
    Since changing the sign of the perturbation preserves transversality,
    \Cref{prop:cleaning with transversality} applied with perturbation
    size $|s-s_*|$ and exponent $n+4$ gives $C_{z_*,s_*}>0$ such that
    \[
        |s-s_*|\le C_{z_*,s_*}\dist_\pb(z,z_*)^{n+4}
        \qquad\text{for }(z,s)\in E\text{ with }z\text{ close to }z_*.
    \]
    Since $\pi_1(E)$ is bounded, it has finite $\cH_\pb^{n+2}$ measure,
    so \Cref{prop:gmt}, with both exponents equal to $n+2$ and a
    modulus of continuity comparable to $r^2$, yields
    \[
        \cH^1(\pi_2(E))=0.
    \]
    We can therefore choose $s<0$ arbitrarily close to zero outside
    $\pi_2(E)$, which, by the definition of $E$ and the time
    translation, gives $\Sigma_{n-1}^m(v_s)\cap K_i=\emptyset$.
    Since $(m-1)$-singularity on compact sets is open, we can also
    preserve this property on $K_1,\ldots,K_{i-1}$; hence, after
    finitely many such choices, we obtain an $(m-1)$-singular solution
    on $K$, proving the required approximation.
\end{proof}

We now immediately get the validity of \Cref{corollary:rectifiability}.

\begin{proof}[Proof of \Cref{corollary:rectifiability}]
This follows from \Cref{theorem:generic regularity} and
\cite[Lemma 6.3]{genericStefan4D}.
\end{proof}

\bibliographystyle{aomalpha}
\bibliography{biblio.bib}

\providecommand{\bysame}{\leavevmode\hbox to3em{\hrulefill}\thinspace}
\providecommand{\noopsort}[1]{}
\providecommand{\mr}[1]{\href{http://www.ams.org/mathscinet-getitem?mr=#1}{MR~#1}}
\providecommand{\zbl}[1]{\href{http://www.zentralblatt-math.org/zmath/en/search/?q=an:#1}{Zbl~#1}}
\providecommand{\jfm}[1]{\href{http://www.emis.de/cgi-bin/JFM-item?#1}{JFM~#1}}
\providecommand{\arxiv}[1]{\href{http://www.arxiv.org/abs/#1}{arXiv~#1}}
\providecommand{\doi}[1]{\url{https://doi.org/#1}}
\providecommand{\MR}{\relax\ifhmode\unskip\space\fi MR }
\providecommand{\MRhref}[2]{%
  \href{http://www.ams.org/mathscinet-getitem?mr=#1}{#2}
}
\providecommand{\href}[2]{#2}
\begin{thebibliography}{CMSW25}

\bibitem[BK24]{multeplicity1MCF}
\bgroup\scshape{}R.~H. Bamler\egroup{} and
  \bgroup\scshape{}B.~Kleiner\egroup{}, On the multiplicity one conjecture for
  mean curvature flows of surfaces, 2024. Available at
  \url{https://arxiv.org/abs/2312.02106}.

\bibitem[BM88]{CiteLojasiewicz}
\bgroup\scshape{}E.~Bierstone\egroup{} and \bgroup\scshape{}P.~D.
  Milman\egroup{}, Semianalytic and subanalytic sets,  \emph{Inst. Hautes
  \'Etudes Sci. Publ. Math.} no.~67 (1988), 5--42.

\bibitem[Bla06]{parabolicMonneau2}
\bgroup\scshape{}A.~Blanchet\egroup{}, On the singular set of the parabolic
  obstacle problem,  \emph{J. Differential Equations} \textbf{231} no.~2
  (2006), 656--672.

\bibitem[CPS04]{parabolicObstacle}
\bgroup\scshape{}L.~Caffarelli\egroup{},
  \bgroup\scshape{}A.~Petrosyan\egroup{}, and
  \bgroup\scshape{}H.~Shahgholian\egroup{}, Regularity of a free boundary in
  parabolic potential theory,  \emph{J. Amer. Math. Soc.} \textbf{17} no.~4
  (2004), 827--869.

\bibitem[Caf77]{Caffarelli77}
\bgroup\scshape{}L.~A. Caffarelli\egroup{}, The regularity of free boundaries
  in higher dimensions,  \emph{Acta Math.} \textbf{139} no.~3-4 (1977),
  155--184.

\bibitem[CC24]{carducci2024genericregularityfreeboundaries}
\bgroup\scshape{}M.~Carducci\egroup{} and \bgroup\scshape{}R.~Colombo\egroup{},
  Generic regularity of free boundaries in the obstacle problem for the
  fractional {Laplacian}, 2024. Available at
  \url{https://arxiv.org/abs/2412.16066}.

\bibitem[CCTL26]{carducci2026singularsetstefanproblem}
\bgroup\scshape{}M.~Carducci\egroup{}, \bgroup\scshape{}R.~Colombo\egroup{},
  and \bgroup\scshape{}C.~Torres-Latorre\egroup{}, The singular set in the thin
  {Stefan} problem, 2026. Available at \url{https://arxiv.org/abs/2610.10751}.

\bibitem[CCMS24]{huiskenConjectureMCF}
\bgroup\scshape{}O.~Chodosh\egroup{}, \bgroup\scshape{}K.~Choi\egroup{},
  \bgroup\scshape{}C.~Mantoulidis\egroup{}, and
  \bgroup\scshape{}F.~Schulze\egroup{}, Mean curvature flow with generic
  initial data,  \emph{Invent. Math.} \textbf{237} no.~1 (2024), 121--220.

\bibitem[CMS24]{CMSImproved}
\bgroup\scshape{}O.~Chodosh\egroup{}, \bgroup\scshape{}C.~Mantoulidis\egroup{},
  and \bgroup\scshape{}F.~Schulze\egroup{}, Improved generic regularity of
  codimension-1 minimizing integral currents,  \emph{Ars Inveniendi Analytica}
  (2024), Paper No. 3, 16.

\bibitem[CMS26]{CMS}
\bgroup\scshape{}O.~Chodosh\egroup{}, \bgroup\scshape{}C.~Mantoulidis\egroup{},
  and \bgroup\scshape{}F.~Schulze\egroup{}, Generic regularity for minimizing
  hypersurfaces in dimensions 9 and 10,  \emph{Publ. Math. Inst. Hautes
  \'Etudes Sci.} \textbf{143} (2026), 143--188.

\bibitem[CMSW25]{chodosh2025genericregularityminimizinghypersurfaces}
\bgroup\scshape{}O.~Chodosh\egroup{}, \bgroup\scshape{}C.~Mantoulidis\egroup{},
  \bgroup\scshape{}F.~Schulze\egroup{}, and \bgroup\scshape{}Z.~Wang\egroup{},
  Generic regularity for minimizing hypersurfaces in dimension 11, 2025.
  Available at \url{https://arxiv.org/abs/2506.12852}.

\bibitem[Col25]{genericStefan4D}
\bgroup\scshape{}G.~Colombo\egroup{}, Generic regularity in time for solutions
  of the {S}tefan problem in 4+1 dimensions, 2025. Available at
  \url{https://arxiv.org/abs/2504.11876}.

\bibitem[Dan00]{DanersHeatKernel}
\bgroup\scshape{}D.~Daners\egroup{}, Heat kernel estimates for operators with
  boundary conditions,  \emph{Math. Nachr.} \textbf{217} (2000), 13--41.

\bibitem[DGPT17]{DanielliPTOP}
\bgroup\scshape{}D.~Danielli\egroup{}, \bgroup\scshape{}N.~Garofalo\egroup{},
  \bgroup\scshape{}A.~Petrosyan\egroup{}, and \bgroup\scshape{}T.~To\egroup{},
  Optimal regularity and the free boundary in the parabolic {S}ignorini
  problem,  \emph{Mem. Amer. Math. Soc.} \textbf{249} no.~1181 (2017), v+103.

\bibitem[DPSV21]{branchingTwoPhaseBernoulli}
\bgroup\scshape{}G.~De~Philippis\egroup{},
  \bgroup\scshape{}L.~Spolaor\egroup{}, and
  \bgroup\scshape{}B.~Velichkov\egroup{}, Regularity of the free boundary for
  the two-phase {B}ernoulli problem,  \emph{Invent. Math.} \textbf{225} no.~2
  (2021), 347--394.

\bibitem[DZ15]{DongZhangHigherOrderSchauder}
\bgroup\scshape{}H.~Dong\egroup{} and \bgroup\scshape{}H.~Zhang\egroup{},
  {Schauder} estimates for higher-order parabolic systems with time irregular
  coefficients,  \emph{Calc. Var. Partial Differential Equations} \textbf{54}
  no.~1 (2015), 47--74.

\bibitem[Duv73]{Duvaut}
\bgroup\scshape{}G.~Duvaut\egroup{}, R\'esolution d'un probl\`eme de {S}tefan
  (fusion d'un bloc de glace \`a{} z\'ero degr\'e),  \emph{C. R. Acad. Sci.
  Paris S\'er. A-B} \textbf{276} (1973), A1461--A1463.

\bibitem[Eva10]{Evans}
\bgroup\scshape{}L.~C. Evans\egroup{}, \emph{Partial differential equations},
  second ed., \emph{Graduate Studies in Mathematics} \textbf{19}, American
  Mathematical Society, Providence, RI, 2010.

\bibitem[EG15]{EvansGariepy}
\bgroup\scshape{}L.~C. Evans\egroup{} and \bgroup\scshape{}R.~F.
  Gariepy\egroup{}, \emph{Measure theory and fine properties of functions},
  revised ed., \emph{Textbooks in Mathematics}, CRC Press, Boca Raton, FL,
  2015.

\bibitem[FRTL23]{genericSignorini3d}
\bgroup\scshape{}X.~Fernández-Real\egroup{} and
  \bgroup\scshape{}C.~Torres-Latorre\egroup{}, Generic regularity of free
  boundaries for the thin obstacle problem,  \emph{Advances in Mathematics}
  \textbf{433} (2023), 109323.

\bibitem[FRY23]{fernandezreal2023genericpropertiesfreeboundary}
\bgroup\scshape{}X.~Fernández-Real\egroup{} and
  \bgroup\scshape{}H.~Yu\egroup{}, Generic properties in free boundary
  problems, 2023, to appear in Amer. J. Math. Available at
  \url{https://arxiv.org/abs/2308.13209}.

\bibitem[FRY25]{fernandezreal2025linearizedequationgenericregularity}
\bgroup\scshape{}X.~Fernández-Real\egroup{} and
  \bgroup\scshape{}H.~Yu\egroup{}, Linearized equation and generic regularity
  in the {Alt-Caffarelli} problem, 2025. Available at
  \url{https://arxiv.org/abs/2510.18330}.

\bibitem[Fig18]{surveyStefan}
\bgroup\scshape{}A.~Figalli\egroup{}, Regularity of interfaces in phase
  transitions via obstacle problems---{F}ields {M}edal lecture,  in
  \emph{Proceedings of the {I}nternational {C}ongress of
  {M}athematicians---{R}io de {J}aneiro 2018. {V}ol. {I}. {P}lenary lectures},
  World Sci. Publ., Hackensack, NJ, 2018, pp.~225--247.

\bibitem[FROS20]{FRSObst}
\bgroup\scshape{}A.~Figalli\egroup{}, \bgroup\scshape{}X.~Ros-Oton\egroup{},
  and \bgroup\scshape{}J.~Serra\egroup{}, Generic regularity of free boundaries
  for the obstacle problem,  \emph{Publ. Math. Inst. Hautes \'Etudes Sci.}
  \textbf{132} (2020), 181--292.

\bibitem[FROS24]{FRS}
\bgroup\scshape{}A.~Figalli\egroup{}, \bgroup\scshape{}X.~Ros-Oton\egroup{},
  and \bgroup\scshape{}J.~Serra\egroup{}, The singular set in the {S}tefan
  problem,  \emph{J. Amer. Math. Soc.} \textbf{37} no.~2 (2024), 305--389.

\bibitem[FS19]{FS19}
\bgroup\scshape{}A.~Figalli\egroup{} and \bgroup\scshape{}J.~Serra\egroup{}, On
  the fine structure of the free boundary for the classical obstacle problem,
  \emph{Invent. Math.} \textbf{215} no.~1 (2019), 311--366.

\bibitem[FZ25]{CinftyObstacle}
\bgroup\scshape{}F.~Franceschini\egroup{} and
  \bgroup\scshape{}W.~Zato\'n\egroup{}, {$C^\infty$} partial regularity of the
  singular set in the obstacle problem,  \emph{Anal. PDE} \textbf{18} no.~1
  (2025), 199--264.

\bibitem[HS85]{HS}
\bgroup\scshape{}R.~Hardt\egroup{} and \bgroup\scshape{}L.~Simon\egroup{}, Area
  minimizing hypersurfaces with isolated singularities,  \emph{J. Reine Angew.
  Math.} \textbf{362} (1985), 102--129.

\bibitem[Lin90]{uniqueContinuationHeat}
\bgroup\scshape{}F.-H. Lin\egroup{}, A uniqueness theorem for parabolic
  equations,  \emph{Comm. Pure Appl. Math.} \textbf{43} no.~1 (1990), 127--136.

\bibitem[LM15]{parabolicMonneau1}
\bgroup\scshape{}E.~Lindgren\egroup{} and \bgroup\scshape{}R.~Monneau\egroup{},
  Pointwise regularity of the free boundary for the parabolic obstacle problem,
   \emph{Calc. Var. Partial Differential Equations} \textbf{54} no.~1 (2015),
  299--347.

\bibitem[MSNP25]{genericRegularityIsoperimetricRegions}
\bgroup\scshape{}K.~Marshall-Stevens\egroup{},
  \bgroup\scshape{}G.~Niu\egroup{}, and \bgroup\scshape{}D.~Parise\egroup{},
  Generic regularity of isoperimetric regions in dimension eight, 2025.
  Available at \url{https://arxiv.org/abs/2511.03795}.

\bibitem[MRO26]{boundSingularSetLipRHS}
\bgroup\scshape{}A.~Martínez\egroup{} and
  \bgroup\scshape{}X.~Ros-Oton\egroup{}, Dimension of the singular set in the
  parabolic obstacle problem, 2026. Available at
  \url{https://arxiv.org/abs/2603.06352}.

\bibitem[Mon03]{monneauGeneric2d}
\bgroup\scshape{}R.~Monneau\egroup{}, On the number of singularities for the
  obstacle problem in two dimensions,  \emph{J. Geom. Anal.} \textbf{13} no.~2
  (2003), 359--389.

\bibitem[RS19]{runge}
\bgroup\scshape{}A.~Rüland\egroup{} and \bgroup\scshape{}M.~Salo\egroup{},
  Quantitative {Runge} approximation and inverse problems,  \emph{Int. Math.
  Res. Not.} \textbf{2019} no.~20 (2019), 6216--6234.

\bibitem[Sch77]{schaefferSingularities}
\bgroup\scshape{}D.~G. Schaeffer\egroup{}, Some examples of singularities in a
  free boundary,  \emph{Ann. Scuola Norm. Sup. Pisa Cl. Sci. (4)} \textbf{4}
  no.~1 (1977), 133--144.

\bibitem[Ste90]{Stefan}
\bgroup\scshape{}J.~Stefan\egroup{}, {\"U}ber die {T}heorie der {E}isbildung,
  \emph{Monatsh. Math. Phys.} \textbf{1} no.~1 (1890), 1--6.

\bibitem[Ves08]{VessellaParabolicUC}
\bgroup\scshape{}S.~Vessella\egroup{}, Quantitative estimates of unique
  continuation for parabolic equations, determination of unknown time-varying
  boundaries and optimal stability estimates,  \emph{Inverse Problems}
  \textbf{24} no.~2 (2008), 81.

\end{thebibliography}
\end{document}